\documentclass[reqno,12pt]{article}
\usepackage{amssymb}
\usepackage[utf8]{inputenc}
\usepackage{amsmath}
\usepackage{amsfonts}
\usepackage{pdfpages}
\usepackage{graphicx}
\usepackage[a4paper, total={6in, 10in}]{geometry}
\usepackage{amssymb}
\usepackage{amsmath}
\usepackage{xfrac}
\usepackage{authblk}
\usepackage{setspace}
\usepackage{array}
\usepackage{xcolor}
\usepackage{amsthm}
\newtheorem{theorem}{Theorem}[section]
\newtheorem{definition}{Definition}
\newtheorem{lemma}[theorem]{Lemma}
\newtheorem{corollary}{Corollary}[theorem]
\newtheorem{assumption}{Assumption}
\newtheorem{remark}{Remark}

\usepackage{hyperref}
\hypersetup{
     colorlinks   = true,
     linkcolor    = blue,
     citecolor    = blue
}
\usepackage[english]{babel}
\usepackage{subcaption}
\usepackage[ruled,vlined]{algorithm2e}
\usepackage{subcaption}
\usepackage{bigints}

\usepackage{soul}

\usepackage{bigints}

\title{Sharp error estimates for target measure diffusion maps with applications to the committor problem}
\author[1]{Shashank Sule\thanks{ssule25@umd.edu}}
\author[2]{Luke Evans\thanks{levans@flatironinstitute.org}}
\author[1]{Maria Cameron\thanks{mariakc@umd.edu}}
\affil[1]{\small{Department of Mathematics, University of Maryland, College Park, MD 20742, USA}}
\affil[2]{\small{Flatiron Institute, 162 Fifth Avenue
New York, NY 10010}}
\begin{document}
\maketitle
\begin{abstract}
    We obtain asymptotically sharp error estimates for the consistency error of the Target Measure Diffusion map (TMDmap) (Banisch et al. 2020), a variant of diffusion maps featuring importance sampling and hence allowing input data drawn from an arbitrary density. The derived error estimates include the bias error and the variance error. The resulting convergence rates are consistent with the approximation theory of graph Laplacians. The key novelty of our results lies in the explicit quantification of all the prefactors on leading-order terms. We also prove an error estimate for solutions of Dirichlet BVPs obtained using TMDmap, showing that the solution error is controlled by consistency error. We use these results to study an important application of TMDmap in the analysis of rare events in systems governed by overdamped Langevin dynamics using the framework of transition path theory (TPT). The cornerstone ingredient of TPT is the solution of the committor problem, a boundary value problem for the backward Kolmogorov PDE. Remarkably, we find that the TMDmap algorithm is particularly suited as a meshless solver to the committor problem due to the cancellation of several error terms in the prefactor formula. Furthermore, significant improvements in bias and variance errors occur when using a quasi-uniform sampling density. Our numerical experiments show that these improvements in accuracy are realizable in practice when using $\delta$-nets as spatially uniform inputs to the TMDmap algorithm.
\end{abstract}

{\bf Keywords.}
Target measure diffusion map, bias error, variance error, committor, $\delta$-net, overdamped Langevin dynamics, manifold learning

\section{Introduction}
\label{sec: Intro}
The diffusion map algorithm (Dmap) \cite{coifman2006diffusion} is a popular tool for the analysis of high-dimensional data arising from stochastic dynamical systems. For example, applications of Dmap have guided important discoveries in chemical physics \cite{nadler2006diffusion}, structural biology \cite{kim2015systematic}, and ocean mixing \cite{banisch2017understanding}. Dmap is specifically amenable to such applications because it has a means for modulating the sampling distribution of the data. In particular, it is assumed that the experimental data $\mathcal{X} \subseteq \mathbb{R}^{m}$ is modeled by a stochastic process $X_t$ following the overdamped Langevin dynamics on a compact $d$-dimensional Riemannian manifold $\mathcal{M}$ written formally \cite{li2023riemannian} as 
\begin{align}
    dX_t = -\nabla_{\mathcal{M}}V(X_t)\,dt + \sqrt{2\beta^{-1}} dW_t. \label{eq: old}
\end{align}
Here $V: \mathcal{M} \to \mathbb{R}$ is a potential function, $\nabla_{\mathcal{M}}$ is the manifold gradient operator, and $W_t$ is a manifold Brownian motion. It is well-established \cite{hsu2002stochastic} that the dynamics $X_t$ are generated by the second-order elliptic operator known as the \emph{backward Kolmogorov operator} 
\begin{align}
    \mathcal{L}:= \beta^{-1}\Delta_{\mathcal{M}} - \langle \nabla V, \nabla \rangle_{\mathcal{M}}. \label{eq: generator}
\end{align}
Here $\Delta_{\mathcal{M}}$ is the Laplace-Beltrami operator on $\mathcal{M}$. {Throughout the rest of the paper, $\Delta_{\mathcal{M}}$, $\nabla_{\mathcal{M}}$, and $\langle X,Y \rangle_{\mathcal{M}}$ will be denoted by $\Delta$, $\nabla$, and $X \cdot Y$ for brevity.}

The generator $\mathcal{L}$ holds immense information about the SDE \eqref{eq: old}. For instance, the eigenvectors of $\mathcal{L}$ form reaction coordinates for the system \eqref{eq: old} and statistics of \eqref{eq: old} such as mean first passage times, reaction rates, or exit times can be studied through solutions to PDEs involving $\mathcal{L}$. Additionally, the spectral information of $\mathcal{L}$ reflects the implicit timescales or the transitions between metastable states in the cases when $V$ has finitely many attracting basins \cite{bovier2016metastability}. 

The advantage of Dmap over its predecessor Laplacian eigenmap~\cite{belkin2003laplacian} is that Dmap has a means for controlling the effect of sampling density.
With a properly chosen renormalization, Dmap produces a discrete operator
$L^{\textsf{Dmap}}_{\epsilon,1/2}$ approximating $\mathcal{L}$ from a finite set of points sampled through the invariant density of the Langevin dynamics \eqref{eq: old}. Remarkably, a priori knowledge of the manifold $\mathcal{M}$ is not required. In the notation $L^{\textsf{Dmap}}_{\epsilon,1/2}$, $\epsilon$ is the bandwidth parameter and 1/2 is the value of the renormalization parameter $\alpha$. In particular, as the number of points $n$ tends to infinity and the kernel bandwidth $\epsilon$ tends to zero, $L^{\textsf{Dmap}}_{\epsilon, 1/2} \to \mathcal{L}$ in a pointwise sense. Thus $L^{\textsf{Dmap}}_{\epsilon, 1/2}$ is a consistent estimator of $\mathcal{L}$; in fact it has also been established that the spectral information of $L^{\textsf{Dmap}}_{\epsilon, 1/2}$ is a consistent estimator of the spectral information of $\mathcal{L}$~\cite{belkin2006convergence,cheng2022eigen,wormell2021spectral}. 

These approximation properties of Dmap and its variants tailored to more general Ito diffusions have motivated numerous applications to unsupervised learning problems especially in natural sciences. 
However, the construction of $L^{\textsf{Dmap}}_{\epsilon, 1/2}$ has an important limitation. The discrete operator $L^{\textsf{Dmap}}_{\epsilon, 1/2}$ approximates $\mathcal{L}$ correctly \emph{only if the input data is sampled i.i.d.\footnotemark[1] from the invariant Gibbs density} $\mu \propto \exp\left(-\beta V\right)$. \footnotetext[1]{The abbreviation i.i.d. stands for \emph{independent, identically distributed}.}
This is quite a restrictive requirement. For instance, Refs. \cite{lai2018point,banisch2020diffusion} point out that sampling from $\mu$ can be undesirable or unfeasible due to the metastability of \eqref{eq: old} or the slow mixing times of simulating the SDE. To fix this problem, Banisch et al.~\cite{banisch2020diffusion} introduced the \emph{Target Measure Diffusion Map} (TMDmap), the subject of this paper. 

\subsection{Rare events, the committor problem, and the TMDmap}

In TMDmap, we seek to approximate the operator $\mathcal{L}$ generating the dynamics \eqref{eq: old} with a known \emph{target} density $\mu$ proportional to $\exp\left(-\beta V(x)\right)$. Remarkably, the input data to TMDmap can come from \emph{any} sampling density $\rho$. As long as $\mu$ is absolutely continuous with respect to $\rho$, i.e the support of $\mu$ is contained within the support of $\rho$,  a Gaussian kernel of bandwidth $\epsilon$
\begin{align}
    k_{\epsilon}(x,y) = \exp\left(-\frac{\|x-y\|^2}{\epsilon}\right) \label{eq: kernel}
\end{align}
can be appropriately renormalized to produce a random walk on the data generated by $L^{\sf TMD}_{\epsilon,\mu}:= L^{(n)}_{\epsilon,\mu}$. Banisch et al. \cite{banisch2020diffusion} have shown that the Monte-Carlo limit of the generator matrix
$L^{(n)}_{\epsilon,\mu}$ as $n \to \infty$ is an $O(\epsilon)$ perturbation of the action of $\mathcal{L}$ on a suitable test function $f$:
\begin{align}
    L^{(n)}_{\epsilon,\mu}f(x) \to \mathcal{L}f(x) + O(\epsilon).
    \label{eq: banish result}
\end{align}
Hence, the estimator $L^{\sf TMD}_{\epsilon,\mu}$ has the same approximation property as $L^{\sf Dmap}_{\epsilon,\alpha}$ but can be computed from an arbitrary sampling density supported on $\mathcal{M}$. This augments diffusion maps with importance sampling techniques, particularly with enhanced sampling and data post-processing strategies from molecular dynamics. These enhanced sampling techniques allow the exploration of $\mathcal{M}$ beyond the minima of $V$ where numerical simulations of \eqref{eq: old} tend to cluster due to metastability. 

 For a finite dataset on $\mathcal{M}$ sampled from an arbitrary sampling density $\rho$, the TMDmap generator $L^{\sf TMD}_{\epsilon,\mu}$ gives a discretization of the operator $\mathcal{L}$ to this dataset. This meshless discretization allows us to approximate boundary-value problems (BVPs) involving the generator $\mathcal{L}$.
One such BVP is the \emph{committor problem}, the pivotal problem in 
\emph{transition path theory} (TPT) \cite{vanden2006towards}. TPT is a framework for the study of metastable systems where the process $X_t$ spends most of its time near attractors of the drift field and transitions rarely between them. The committor function $q$ is defined as follows: given two disjoint regions $A$ and $B$ in $\mathcal{M}$, $q(x)$ is the probability that the process \eqref{eq: old} started at $x$ visits $B$ before $A$. The committor $q$ satisfies the following Dirichlet BVP known as the \emph{committor problem} \cite{vanden2006towards}: 
\begin{align}
    \mathcal{L}q(x) = 0, \, x \in \mathcal{M} \setminus (A \cup B), \quad q(\partial A) = 0, \quad q(\partial B) = 1. \label{eq: committor bvp}
\end{align}
If we have data $\mathcal{X}(n) = \{x_i\}_{i=1}^{n}$ sampled through \eqref{eq: old} (for instance, through an Euler-Maruyama scheme) then we can use it as input to the Dmap algorithm and obtain a discretization to problem \eqref{eq: committor bvp}: 
\begin{align}
    L^{\sf Dmap}_{\epsilon,1/2}q^{\sf Dmap}(x_i) = 0, \, x_{i} \in \mathcal{M} \setminus (A \cup B), \quad q^{\sf Dmap}(x_i \in \partial A) = 0, \quad q^{\sf Dmap}(x_i \in \partial B) = 1. \label{eq: dmap committor bvp}
\end{align}
Sufficiently sampled data $\mathcal{X}(n)$ from \eqref{eq: old} will resemble samples of the invariant density $\mu$. But in metastable systems, $\mu$ undersamples the transition regions between $A$ and $B$, leading to poor accuracy of $q^{\sf Dmap}$ in precisely the regions where the committor probabilities are needed. Consequently, it is advantageous to discretize \eqref{eq: committor bvp} through TMDmap to a point cloud generated by an enhanced sampling algorithm, e.g. metadynamics \cite{metadynamics_2002}: 
\begin{align}
    L^{(n)}_{\epsilon,\mu}q^{\sf TMD}(x_i) = 0, \, x_{i} \in \mathcal{M} \setminus (A \cup B), \quad q^{\sf TMD}(x_i \in \partial A) = 0, \quad q^{\sf TMD}(x_i \in \partial B) = 1 \label{eq: tmdmap committor bvp}.
\end{align}

\subsection{The goal of this work and summary of main results}

The committor problem was solved numerically using a variant of TMDmap with Mahalanobis kernel in \cite{evans2022computing}. It was observed that subsampling data to make it spatially quasi-uniform using \emph{$\delta$-nets}~\cite{crosskey2017atlas} improved the accuracy of the TMDmap-based committor $q^{\sf TMD}$~\cite{evans2022computing}. Here, $\delta$-nets are defined as maximal subsets of the point cloud $\mathcal{X}(n)$ where any two points are distance at least $\delta$ apart. The utility of such uniform subsampling raises the question of \emph{what an optimal sampling density should be}. Motivated by this question, the goal of this work is to quantify the error of the TMDmap algorithm in terms of its parameters. The specific objectives are the following.
\begin{itemize}
\item The first objective is to derive an error formula for the discrete generator as a
function of the sampling density $\rho$, the bandwidth parameter $\epsilon$, the number of sample points $n$, and the target measure $\mu$. The error has two
components, the \emph{bias error} that decays as the bandwidth $\epsilon$ tends to zero, and the \emph{variance
error} that decays as the number $n$ of samples tends to infinity but blows up as $\epsilon\rightarrow 0$ unless $n\rightarrow \infty$ fast enough.
\item The second objective is to propose a way to reduce the error of TMDmap via simple post-processing of the input data and test it on benchmark problems.
\end{itemize}
Our results are summarized as follows: 
\begin{enumerate}
    \item { We derive a sharp error bound and establish a relationship between $n$ and $\epsilon$ necessary for the convergence the TMDmap generator.}

   \begin{theorem} [The total error bound] 
   \label{thm:main}
   Let $\mathcal{M}$ be a compact $d$-dimensional manifold without boundary.  Let $\mathcal{X}(n)\subset\mathcal{M}$ be a point cloud sampled i.i.d. with density $\rho$, $0<\rho_{\min}\le \rho(x)\le\rho_{\max}<\infty$. Let $x \in \mathcal{M}$ be an arbitrary point. Let $f \in C^{2}(\mathcal{M})$ be an arbitrary function. Furthermore, let $\epsilon$ be the kernel bandwidth and $\mu$ be the target density used for constructing the TMDmap generator $L^{(n)}_{\epsilon,\mu}$. Then as $n \to \infty$ and $\epsilon \to 0$ so that\footnotemark[2]
   \footnotetext[2]{The notation $\log$ with no subscript is used for the natural logarithm throughout this paper.}
   \begin{equation}
    \label{eq:thmM4_en}
    \lim_{\substack{n\rightarrow\infty\\\epsilon\rightarrow 0}}\frac{n\epsilon^{2+d/2}}{\log n} = \infty,
\end{equation}
 with probability greater than $1 - 2n^{-3}$, we have:    
\begin{align}
    |4\beta^{-1}L^{(n)}_{\epsilon,\mu}f(x) - \mathcal{L}f(x)| & \le \underbrace{\frac{\alpha\epsilon}{\rho^{1/2}(x_i)}\left(2\|\nabla f(x)\|\epsilon^{1/2} + 11|f(x)|\right)}_{\text{variance error}} \label{eq: variance error term}\\
    &+ \underbrace{\epsilon \left|\mathcal{B}_{1}[f,\mu] + \mathcal{B}_{2}[f,\mu, \rho] + \mathcal{B}_{3}[f,\mu, \rho] \right| + O(\epsilon^2)}_{\text{bias error}}. \label{eq: bias error term}
\end{align}

The expressions for $\alpha, \mathcal{B}_1, \mathcal{B}_{2}$, and $\mathcal{B}_3$ are given by:
\begin{align}
    \alpha &= \frac{1}{(2\pi)^{d/4}}\sqrt{\frac{\log n}{n \epsilon^{4+d/2}}}, \\
    \mathcal{B}_1[f,\mu] &:= \frac{1}{4}\left[\mathcal{Q}\left(f\mu^{1/2}\right)- f \mathcal{Q}(\mu^{1/2}) \right] + \frac{1}{16}\left(2\nabla f \cdot \nabla \left(\mu^{1/2}\omega \right) + (\mu^{1/2}\omega)\Delta f \right),  \label{eq: bias error term 1} \\
    \mathcal{B}_{2}[f,\mu,\rho] &:= - \frac{1}{16}\left(2\nabla f \cdot \nabla \left(\mu^{1/2}\frac{\Delta \rho}{\rho} \right) + \left(\mu^{1/2}\frac{\Delta \rho}{\rho}\right)f \right), \label{eq: bias error term 2} \\
    \mathcal{B}_{3}[f,\mu,\rho] &:= \frac{1}{16}\left[\frac{\Delta (\mu^{1/2})}{\mu^{1/2}} - \left(\frac{\Delta \rho}{\rho} - \omega \right)\right]\left[f\frac{\Delta (\mu^{1/2})}{\mu^{1/2}} - \frac{\Delta (\mu^{1/2}f)}{\mu^{1/2}}\right].  \label{eq: bias error term 3}
\end{align}
Here, $\mathcal{Q}$ is a non-linear differential operator and $\omega$ is a smooth function on $\mathcal{M}$. 
\end{theorem}
The precise forms of $\omega$ and $\mathcal{Q}$ are further detailed in Section \ref{sec: results}. 
{The proof of this theorem is done using techniques similar to those in \cite{hein2007graph, singer2006graph, belkin2006convergence}.}

    Our key contribution is the explicit formulas for the prefactors in the leading-order terms of the bias and variance errors. To our knowledge, this is the first consistency result that enumerates the prefactor from the bias error. 
    
    \begin{remark}
    {The assumption that the manifold $\mathcal{M}$ is compact without boundary is not essential. It is adopted to keep the statement of Theorem \ref{thm:main} more concise. Often, this assumption does not hold in applications. A common case is when the manifold $\mathcal{M}$ is not compact, but the data points are located within a bounded region $\Omega\subset\mathcal{M}$. In this case, the error bound in Theorem \ref{thm:main} remains valid for all points $x\in\Omega$ lying far enough from the boundary of $\Omega$, i.e., at a distance greater than $m\sqrt{\epsilon/2}$ from $\partial\Omega$ where $m$ is a positive number such that $\exp\left(-\tfrac{m^2}{2}\right)<\epsilon^2$. }
    \end{remark}

    { \item In Theorem \ref{thm: BVP error estimate} we use the \emph{method of comparisons} \cite{morton_mayers_2005} to prove an error estimate for solutions to Dirichlet BVPs using TMDmap (e.g. \eqref{eq: tmdmap committor bvp}), showing that making the \emph{consistency error} on the left-hand side of \eqref{eq: variance error term}--\eqref{eq: bias error term} smaller reduces the \emph{solution error} between the numerical solution and the true solution. Thus improvements to the bias and variance error transfer over as improvements in the solution error. }
    
    \item The decay of the bias error may be sped up by canceling some of the prefactors. For instance, such cancellations occur when $(i)$ the manifold has zero curvature, $(ii)$ the sampling density is quasi-uniform, or $(iii)$ when the test function satisfies $\mathcal{L}f = 0$. Additionally, when the manifold has constant but non-zero curvature, a prefactor may still be eliminated if $\mathcal{L}f = 0$. Importantly, conditions $(i)$, $(ii)$, and $(iii)$ are satisfied in a large class of problems of interest. Condition $(i)$ occurs naturally when the system exhibits a separation in slow and fast time scales due to a stiff component in the diffusion matrix. Condition $(i)$ also holds in angular systems such as butane or alanine dipeptide where $\mathcal{M}$ is given by Cartesian products of flat tori $\mathbb{R}\setminus \mathbb{Z}$. Condition $(ii)$ can be obtained when $\mathcal{M}$ is sampled through well-tempered metadynamics \cite{barducci2011metadynamics}. Such samples can be made additionally spatially uniform using $\delta$-nets. Most importantly, condition $(iii)$ is precisely the committor equation \eqref{eq: committor bvp}. A consequence of these observations is that the error bound for the committor, $|q^{{\sf TMD}} - q|$, is smaller than the general error bound implied by Theorem \ref{thm: BVP error estimate} (see Corollary \ref{cor: committor error estimate}) implying that \emph{TMDmap is particularly suited to solving the committor problem especially using quasi-uniform sampling densities}. 
    
     \item We show that these speedups in bias error are realizable in practice for a variety of systems including a 1D periodic system embedded in 2D, Mueller's potential in 2D, and a double well potential 2D. Combined with similar results obtained in \cite[Fig. 9]{evans2022computing} for the alanine-dipeptide system, this work adds to the growing body of evidence that improvements in the approximation theory of diffusion maps are of practical importance. 
    
\end{enumerate}

\subsection{Related work}
\subsubsection{Diffusion maps: theory} 
\label{sec:dmap-theory}
The consistency of kernel-based estimators $L$ of backward Kolmogorov operators on manifolds is an active topic in unsupervised learning. The approximation theory of $L$ splits into independent analyses of the bias and variance error. Early work on these errors concerned pointwise consistency, i.e. the pointwise convergence of the estimator $L$ applied to a test function $f$. The seminal works on pointwise consistency include \cite{belkin2003laplacian, von2008consistency, lafon2004diffusion, hein2005graphs, hein2007graph}. Building on this literature, \cite{singer2006graph} improved the bias error rate in \cite{hein2005graphs} from $O(\sqrt{\epsilon})$ to $O(\epsilon)$ and the variance error rate for Laplacian eigenmaps. An important development to this subject was introduced in \cite{belkin2006convergence} where spectral convergence--the convergence of the discrete eigenvectors to the continuous counterparts--was proved for eigenvectors of the normalized graph Laplacians. Further work on spectral convergence includes \cite{singer2006graph, cheng2020convergence, cheng2022eigen, wang2015spectral}. A recent exciting development in the subject has been the integration of optimal transport-based techniques for variance error estimation \cite{garcia2020error, wormell2021spectral}. 

\subsubsection{Diffusion maps: applications.} Coifman and Lafon in \cite{coifman2006diffusion} initiated an expansive program of adapting diffusion maps to data modeled by vector bundles \cite{singer2012vector}, group invariant manifolds \cite{hoyos2023diffusion}, and diffusion processes \cite{nadler2006diffusion} for applications such as coarse-graining \cite{coifman2008diffusion,singer2009detecting}, image segmentation \cite{wassermann2008diffusion}, data fusion \cite{katz2019alternating}, data representation \cite{little2017multiscale, allard2012multi, bremer2006diffusion}, PDEs on manifolds \cite{liang2021solving,antil2021fractional,jiang2023ghost}, and rare event quantification \cite{evans2022computing,williams2015data,evans2021computing}. One line of research has concerned algorithmic improvements to the Dmap algorithm such as the use of self-tuning, variable bandwidth, or $k$-nearest neighbors (kNN) sparsified kernels \cite{cheng2022convergence,berry2016variable,berry2016local}. Another line of work has centered on improving the diffusion map embedding through its stability \cite{long2019landmark,kohli2021ldle} or its statistical properties \cite{li2020variational}. In particular, the setting of Dmap was adapted to isotropic diffusions from invariant densities to anisotropic diffusions from general densities. Key works on this subject include the extension to diffusions obtained from collective variables through Mahalanobis diffusion maps (mmap) \cite{singer2008non}, the subsequent generalization to symmetric positive definite diffusion tensors \cite{evans2021computing}, and Ref.~\cite{banisch2020diffusion} proposing the TMDmap.  
{ In \cite{evans2022computing}, the TMDmap was upgraded with the Mahalanobis kernel for solving the committor problem for time-reversible processes with variable and anisotropic diffusion tensors arising in chemical physics applications.}

In the context of these works, the present paper can be viewed as an extension of the approximation theory of graph Laplacians to TMDmap. In addition, it leverages this approximation theory to gain concrete insights for developing a more robust practical application to rare event quantification.


\section{Background}
\label{sec: tmdmaps}
The Dmap~\cite{coifman2006diffusion} and TMDmap~\cite{banisch2020diffusion} algorithms are summarized in Table \ref{table:DmapTMDmap}. Let us elaborate on the construction of TMDmap.

Suppose that the data points $\mathcal{X}(n) = \{x_i\}_{i=1}^n\subset \mathcal{M} \subseteq \mathbb{R}^m$ are sampled from a density $\rho$ using any suitable enhanced sampling technique such as temperature acceleration or metadynamics \cite{metadynamics_2002,barducci2011metadynamics}.
The goal of TMDmap is to approximate the action of the infinitesimal generator $\mathcal{L}$ \eqref{eq: generator} for the time-reversible dynamics governed by the SDE \eqref{eq: old} on any smooth function $f: \mathcal{M} \to \mathbb{R}$.
The invariant density $\mu(x)$ for SDE \eqref{eq: old} is known up to a proportionality constant: 
\begin{align}
    \mu(x) \propto \exp\left(-\beta V(x)\right).
\end{align}

The TMDmap steps are the following:
\begin{enumerate}
\item A symmetric kernel matrix $[K^{(n)}_{\epsilon}]_{ij} = k_{\epsilon}(x_i,x_j)$, $1\le i,j\le n$, $x_i,x_j\in\mathcal{X}(n)$ is formed using the Gaussian kernel
\begin{equation}
\label{ker1}
k_{\epsilon}(x,y):=\exp\left(-\frac{\|x-y\|^2}{\epsilon}\right).
\end{equation}
The discrete \emph{kernel density estimate} (KDE) $\rho^{(n)}_{\epsilon}(x_i)$ and the corresponding diagonal matrix $D^{(n)}_{\epsilon}$ are constructed out of rows of $K^{(n)}_{\epsilon}$ as follows: 
\begin{align}
    {[D^{(n)}_{\epsilon}]_{ii} =  
    \rho^{(n)}_{\epsilon}(x_i): = \frac{1}{n}\sum_{j=1}^{n}\left[K^{(n)}_{\epsilon}\right]_{ij}.}
\end{align}

The continuous KDE is given as the law of large numbers (LLN) limit of $\rho^{(n)}_{\epsilon}(x)$
\begin{equation}
\label{qeps}
\rho_{\epsilon} (x): = \int_{\mathcal{M}}k_{\epsilon}(x,y)\rho(y)dy.
 \end{equation}
\item Next, the kernel matrix is normalized by dividing the rows by the KDE and multiplying by $\mu^{1/2}$: 
\begin{equation}
\label{Kmatr}
\left[M_{\mu}^{(n)}\right]_{ii} = \mu^{1/2}(x_i), \quad K^{(n)}_{\epsilon,\mu} = K^{(n)}_{\epsilon}\left(D^{(n)}_{\epsilon}\right)^{-1}\left(M_{\mu}^{(n)}\right).
\end{equation}
Note that for any fixed $x_i\in\mathcal{X}(n)$,
\begin{align}
    \lim_{n\rightarrow\infty}\frac{1}{n}\sum_{j=1}^{n}\left[K^{(n)}_{\epsilon,\mu}\right]_{ij}f(x_j)  =  \mathcal{K}_{\epsilon,\mu}f(x_i) := \int_{\mathcal{M}}k_{\epsilon,\mu}(x_i,y)f(y)\rho(y)\,dy.
\end{align}
Here, $k_{\epsilon,\mu}$ is the renormalized kernel given by 
\begin{align}
    k_{\epsilon,\mu}(x,y):= \frac{k_{\epsilon}(x,y)\mu^{1/2}(y)}{\rho_{\epsilon}(y)}.
\end{align}
{This is the only step in which TMDmap differs from Dmap. In Dmap, the renormalization of the kernel is done by dividing $k_{\epsilon}(x,y)$ by $\rho^{\alpha}_{\epsilon}(y)$.}

\item Next, a Markov matrix $P^{(n)}_{\epsilon,\mu}$ is defined by renormalizing the rows of $K^{(n)}_{\epsilon,\mu}$ to make row sums equal to one: 
\begin{equation}
\label{Pmatr}
D^{(n)}_{\epsilon,\mu} = \textsf{diag}\big(K^{(n)}_{\epsilon,\mu}\mathbf{1}_{n}\big), \quad P^{(n)}_{\epsilon,\mu} = \left(D^{(n)}_{\epsilon,\mu}\right)^{-1} K^{(n)}_{\epsilon,\mu}.
\end{equation}
The LLN limit of $P^{(n)}_{\epsilon,\mu}$ defines its continuous counterpart, the Markov operator
\begin{equation}
\label{Pop}
\mathcal{P}_{\epsilon,\mu}f(x) := \frac{\mathcal{K}_{\epsilon,\mu}f(x)}{\mathcal{K}_{\epsilon,\mu}e(x)},       
\end{equation}
where $e(x) =  1$ for all $x\in\mathcal{M}$.
\item
The generator of the random walk defined by the Markov matrix $P^{(n)}_{\epsilon,\mu}$ is given by
\begin{equation}
\label{Lmatr}
L^{(n)}_{\epsilon,\mu}:=\frac{ P^{(n)}_{\epsilon,\mu} - I}{\epsilon}.
\end{equation}
The LLN limit of $L^{(n)}_{\epsilon,\mu}$ is therefore 
\begin{equation}
\label{Leps}
\mathcal{L}_{\epsilon,\mu}f(x) : = \frac{\mathcal{P}_{\epsilon,\mu}f(x) - f(x)}{\epsilon}.
\end{equation}
\end{enumerate}
For a fixed $x \in \mathcal{M}$, $L^{(n)}_{\epsilon, \mu}f(x) \to \mathcal{L}_{\epsilon, \mu}f(x)$ almost surely as $n\rightarrow\infty$~\cite{banisch2020diffusion}:

\begin{theorem}[Theorem 3.1 from \cite{banisch2020diffusion}]
\label{Lthm}
{Let $\mathcal{L}_{\epsilon,\mu}$ be the generator defined by \eqref{eq: generator} and let $f:\mathcal{M}\rightarrow \mathbb{R}$ be smooth. Then in the limit $n\rightarrow\infty$ and $\epsilon\rightarrow 0$ for every point $x_i\in\mathcal{X}(n)\subset\mathcal{M}$,
\begin{equation}
\label{Lop}
\mathcal{L}_{\epsilon,\mu}f(x)\rightarrow \frac{1}{4} \left(\Delta f +\nabla \log \mu \cdot \nabla f\right) \equiv \frac{\beta}{4}\mathcal{L}.
\end{equation} 
}
\end{theorem}
We are interested in the exact formula for the bias error and the rate of convergence of the term $L^{(n)}_{\epsilon, \mu}f(x)$ to the differential operator $\tfrac{\beta}{4}\mathcal{L}f(x)$. 
\begin{remark}

    \begin{enumerate}
        \item Here we have presented the TMDmap algorithm as originally proposed in \cite{banisch2020diffusion} with an asymmetric renormalization in step 3. In the diffusion maps literature, it is common to construct $K^{(n)}_{\epsilon,\mu}$ with a \emph{symmetric} renormalization given by 
        \begin{align}
            K^{(n)}_{\epsilon, \mu} = \left(D^{(n)}_{\epsilon}\right)^{-1}M^{1/2}K^{(n)}_{\epsilon,\alpha}\left(D^{(n)}_{\epsilon}\right)^{-1}M^{1/2}\label{eq:symmetric renormalization}.
        \end{align}
        Using either equation, \eqref{Kmatr} or \eqref{eq:symmetric renormalization}, for the construction of $K^{(n)}_{\epsilon,\mu}$ leads to the Markov matrix $P^{(n)}_{\epsilon,\mu}$. However, the difference between these renormalizations is as follows: in the case of the symmetric renormalization, the matrix $D^{(n)}_{\epsilon,\mu}$ represents the invariant density of the discrete random walk defined by $P^{(n)}_{\epsilon,\mu}$. As a consequence, when constructing the TMDmap, the eigenvectors of $L^{(n)}_{\epsilon,\mu}$ must be renormalized again with $\left(D^{(n)}_{\epsilon,\mu}\right)^{-1/2}$ to obtain a diagonalization of $P^{(n)}_{\epsilon,\mu}$. In this paper we are not concerned with these eigenvector embeddings; instead, we focus on the generator $L^{(n)}_{\epsilon,\mu}$ for which it does not matter which normalization was used. 
        \item In Dmap, the renormalization step constructs the matrix 
        \begin{align}
            K^{(n)}_{\epsilon, \mu} = K^{(n)}_{\epsilon}D^{-\alpha}_{\epsilon}.
        \end{align}
        The parameter $\alpha$ is intended to modulate the effect of the sampling density $\rho$ appearing as drift term in the $\epsilon \to 0$ limit of the generator $\mathcal{L}$. The $\alpha = 1$ setting removes the effect of the density entirely. In TMDmap, this $\alpha = 1$ renormalization is combined with a reweighting by $\mu^{1/2}$, leading to the reappearance of a drift term related to $\mu$ instead of $\rho$ in Theorem \ref{Lthm}. Thus, TMDmap can be thought of as \emph{Dmap with density reweighting or importance sampling}.  
    \end{enumerate}
\end{remark}

\renewcommand{\arraystretch}{2.5}
\begin{table}[h]
\begin{tabular}{|l|l|l|}
\hline
 &
  Dmap &
  \begin{tabular}[c]{@{}l@{}}TMDmap \end{tabular} \\ \hline
Input &
  \begin{tabular}[c]{@{}l@{}}Dataset $\mathcal{X}(n) \subset \mathcal{M}$, \\ Parameter $\alpha \in [0,1]$,\\ Kernel Bandwidth $\epsilon > 0$
  \end{tabular} 
  &
  \begin{tabular}[c]{@{}l@{}}Dataset $\mathcal{X}(n) \subset \mathcal{M}$, \\ Target density $\mu: \mathbb{R}^d \to [0,1]$,\\ Kernel Bandwidth $\epsilon > 0$ 
  \end{tabular} 
  \\ \hline
 
Kernel matrix & $[K_{\epsilon}]_{i,j} = k_{\epsilon}(x_i, x_j)$
& $[K_{\epsilon}]_{i,j} = k_{\epsilon}(x_i, x_j)$
\\ \hline
Density estimate &
  $D^{(n)}_{\epsilon} = \text{diag}(K_{\epsilon}\mathbf{1})$ &
  \begin{tabular}[c]{@{}l@{}}$D^{(n)}_{\epsilon} = \text{diag}(K_{\epsilon}\mathbf{1})$,\\ $M = \text{diag}([\mu(x_i)]_{x_i \in \mathcal{X}(n)})$\end{tabular} \\ \hline
Normalization &
  $K^{(n)}_{\epsilon, \alpha} = K_{\epsilon}D^{-\alpha}_{\epsilon}$ &
  $K^{(n)}_{\epsilon, \mu} = K_{\epsilon}D^{-1}_{\epsilon}M^{1/2}$ 
   \\ \hline
Markov Process &
  \begin{tabular}[c]{@{}l@{}}$D^{(n)}_{\epsilon, \alpha} = \text{diag}(K^{(n)}_{\epsilon, \alpha}\mathbf{1})$, \\ $P^{(n)}_{\epsilon, \alpha} = \left(D^{(n)}_{\epsilon, \alpha}\right)^{(n)}K_{\epsilon, \alpha}^{(n)}$\end{tabular} &
  \begin{tabular}[c]{@{}l@{}}$D^{(n)}_{\epsilon, \mu} = \text{diag}(K^{(n)}_{\epsilon, \mu}\mathbf{1})$, \\ $P^{(n)}_{\epsilon, \mu} = \left(D^{(n)}_{\epsilon, \mu}\right)K^{(n)}_{\epsilon, \mu}$\end{tabular} \\ \hline
Generator &
  $L^{(n)}_{\epsilon, \alpha} = \epsilon^{-1}(I - P^{(n)}_{\epsilon, \alpha})$ &
  $L^{(n)}_{\epsilon, \mu} = \epsilon^{-1}(I - P^{(n)}_{\epsilon, \mu})$ \\ \hline
\end{tabular}
\caption{Summary of Dmap and TMDmap algorithms. The kernel function $k_\epsilon$  is given by \eqref{ker1}.
}
\label{table:DmapTMDmap}
\end{table}

\subsection{Relevant differential geometry}
\label{subsec: diff geo}
Throughout this paper, it is assumed that $\mathcal{M}$ is a Riemannian manifold, i.e it is equipped with a  \emph{Riemannian metric} $g$ acting on the tangent space $T_x\mathcal{M}$ and varying smoothly with $x$. More precisely, fixing a chart $(U_\alpha, X_\alpha)$ on $\mathcal{M}$ we define the tangent space $T_x \mathcal{M}$ at $x \in \mathcal{M}$: 
\begin{align*}
T_x\mathcal{M} = \{v \in \mathbb{R}^d \mid v = (X_{\alpha}^{-1} \circ \gamma)'(0) \text{ for } \, \gamma : (-\epsilon, \epsilon) \to \mathcal{M} \text{ with } \gamma(0) = x \}.
\end{align*}
Then the metric $g$ can be defined in local coordinates as follows:
\begin{definition}[Riemannian metric]
    A Riemannian metric $g$ is a family of symmetric positive semi-definite 2-tensors $\{g_x\}_{x \in \mathcal{M}}$ where 
    \begin{align}
    g(x)[v,w] = g_{ij}(x)v^{i}w^{j}\quad \forall \, v,w \in T_{x}\mathcal{M}
    \end{align}
    and for every chart $X_\alpha: U_{\alpha} \to \mathbb{R}^d$, 
    the functions $g_{ij}\circ X_{\alpha}^{-1}$ are smooth functions on $\mathbb{R}^d$. 
\end{definition}
Here, the Einstein summation convention is used that implies the summation over all indices repeated as superscripts and subscripts, i.e 
\begin{align*}
    a^ib_i := \sum_{i=1}^{k}a_ib_i.
\end{align*}

The metric tensor defines a distance between points on $\mathcal{M}$.
\begin{definition}
    Let $\tilde{\gamma}: [0,1] \to \mathcal{M}$ be a differentiable curve on $\mathcal{M}$. The energy associated with this curve is defined as 
    \begin{align}
        E[\tilde{\gamma}] = \int_{a}^{b}g(\tilde{\gamma}(t))(\dot{\tilde{\gamma}}(t), \dot{\tilde{\gamma}}(t))\,dt. 
    \end{align}
    A geodesic is defined as the energy minimizing curve between any two points $a,b \in \mathcal{M}$:
    \begin{align}
        \gamma := \underset{\substack{\tilde{\gamma}(0)=a\\\tilde{\gamma}(1)=b}}{\arg\min}\,E[\tilde{\gamma}].
    \end{align}
\end{definition}
It follows from the theory of ODEs that 
for every $x \in \mathcal{M}$ and a tangent vector $v \in T_{x}(\mathcal{M})$ there exists a geodesic $\gamma_v: (-\epsilon, \epsilon) \to \mathcal{M}$ 
satisfying $\gamma(0) = x$, $\dot{\gamma} = v$. By rescaling the given tangent vector, the domain of the geodesic can be mapped to an interval containing $[-1,1]$. Therefore, there is a ball $B_{r(x)}(0) \subseteq T_{x}(\mathcal{M})$ for which tangent vectors can be mapped to geodesics defined on $[-1,1]$. This motivates the definition of the \emph{exponential map}. 
\begin{definition}[Exponential map]
    Fix $x \in \mathcal{M}$. Then for some radius $r(x) > 0$, the exponential map $\exp_{x}: B_{r(x)}(0) \to \mathcal{M}$ is the map taking tangent vectors to geodesics, i.e. $\exp_{x}(v) := \gamma_{v}(1)$. 
\end{definition}
The largest radius $r(x)$ such that $\exp_x$ is a diffeomorphism from $B_{r(x)}$ onto its image is the \emph{injectivity radius at $x$}. The infimum of all such radii over $x \in \mathcal{M}$ is termed the \emph{injectivity radius of $\mathcal{M}$}:
\begin{align}
    \text{inj}(\mathcal{M}) = \underset{x \in \mathcal{M}}{\inf}\, \sup\{r \mid \exp_{x}: B_{r}(0) \to \mathcal{M} \text{ is a diffeomorphism}\}.
\end{align}
For compact manifolds $\mathcal{M}$, $0 < \text{inj}(\mathcal{M}) < \infty$. Therefore, the exponential map defines a set of coordinates on $\mathcal{M}$ such that each vector in $B_{\text{inj}(\mathcal{M})}(0)$ is mapped to a geodesic. This is the system of \emph{normal coordinates}.
\begin{definition}
    Let $U = B_{\text{inj}(\mathcal{M})}(0)$ and $V_x = \exp_{x}B_{\text{inj}(\mathcal{M})}(0)$. Then the map $s_x: U \to V_x$ taking $v$ to $ \exp_{x}(v)$ is a diffeomorphism at $x$. Consequently, $(V_x, s_{x}^{-1}) := (V_x, \hat{s}_{x})$ defines a chart on $\mathcal{M}$ termed normal coordinates.
\end{definition}
Whenever the point $x$ is clear from context, we will drop it as a subscript to denote $s_x$ simply as $s$. The metric $g$ at $x$ becomes the identity tensor in normal coordinates. Assuming normal coordinates and fixing $f \in C^{\infty}(\mathcal{M})$, the Laplace-Beltrami operator can be written in terms of the Euclidean Laplacian: 
\begin{align}
    \Delta f(x) := \partial_{i}^{i}f(x) = \Delta_{\mathbb{R}^d} f \circ s_{ x} (0).
\end{align}

\subsection{Transition path theory}
\label{subsec: tpt}
TMDmap was particularly designed as an enhancement of Dmap for the cases when i.i.d. samples of the invariant measure $\mu$ are difficult to generate or unreasonable to use due to metastability. 
Metastability is a widespread feature of many biophysical systems such as genetic switches \cite{allen2006simulating} or all-atom simulations of macromolecules \cite{vani2022computing}. 
Understanding phenomena such as conformal changes in biomolecules or protein folding relies on an effective quantification of rare events. In this setting, a popular framework for quantifying rare events is transition path theory (TPT) \cite{vanden2006towards}. 

Although TPT allows the study of general Ito diffusions, this paper focuses only on systems governed by the overdamped Langevin dynamics~\eqref{eq: old}. { Let $A, B \subseteq \mathcal{M}$ be two closed disjoint sets with nonempty interiors such that their boundaries coincide with the boundaries of their interiors. }  TPT is concerned with quantifying transitions from $A$, the \emph{reactant set}, to $B$, the \emph{product} set. The quantities of interest are \emph{reactive current} whose flowlines and density concentration delineate the transition channels, the \emph{transition rates} from $A$ to $B$, and the \emph{escape rate} from $A$. In TPT, these quantities are calculated through the forward committor function, or simply the committor, \(q(x)\). The committor characterizes the probability that the trajectory $X_{t}^x$ starting  at $x\in\mathcal{M}\backslash(A\cup B)=:\mathcal{M}_{AB}$ will hit $B$ before $A$. 

The Feynman-Kac formula can be used to show that the committor is the solution to the boundary value problem involving the generator of the process:
\begin{align}
    \mathcal{L}q(x) &= 0,\quad
    x \in \mathcal{M}_{AB},\notag \\
    q& = 0,   \quad x\in \partial A, \label{eq: committor bvp}\\
  q& = 1, \quad x\in \partial B\notag
\end{align}
The time-reversibility of \eqref{eq: old} leads to significant simplifications. First, the \emph{backward committor}, the probability that the process arriving at $x$ last hit $A$ rather than $B$ is given by $1-q(x)$. Second, the committor problem \eqref{eq: committor bvp} admits a variational formulation
\begin{align}
    q =\min_{\substack {u\in C^1(\mathcal{M}_{AB})\\u(\partial A) = 0\\
    u(\partial B) = 1}}\int_{\mathcal{M}_{AB}}\|\nabla  u(x)\|_{2}^{2}\,e^{-\beta V(x)}\,d\text{vol}(x). \label{eq: var_comm}
\end{align}

The key quantities of interest are defined via the committor.  
\begin{enumerate}
    \item \textbf{Reactive Current}:  The reactive current 
    \begin{equation}
        \label{eq:rcurrent}
        J(x) = \beta^{-1}\mu(x)\nabla q(x)
    \end{equation}
    demarcates the tube containing the most likely transition pathways from $A$ to $B$.

    \item \textbf{Transition rate}: The transition rate from $A$ to $B$ is defined by
    \begin{equation}
        \label{eq:ratedef}
        \nu_{AB} = \lim_{T\rightarrow\infty} \frac{N_{AB}}{T},
    \end{equation}
    where $N_{AB}$ is the number of transitions from $A$ to $B$ observed during a time interval of length $T$. It can be expressed as \cite{vanden2006towards}
    \begin{equation}
        \label{eq:rateAB}
        \nu_{AB} = \int_{\Sigma_{AB}} J(x) \cdot n(x) \, d\sigma(x) = \beta^{-1}\int_{\mathcal{M}_{AB}}\|\nabla q(x)\|_2^2\mu(x)d{\rm vol}(x)
    \end{equation}
    where \(n(x)\) is the unit normal vector to the dividing surface $\Sigma_{AB}$ separating the sets \(A\) and \(B\).
    It represents the average number of transitions from $A$ to $B$  per unit time. A common choice for $\Sigma_{AB}$ is the \emph{iso-committor} surface $\{x\in\mathcal{M}: q(x) = 0.5\}$.

    \item 
    \textbf{{The probability that $A$ was hit more recently than $B$}}: The probability that an infinitely long trajectory of \eqref{eq: old} last hit $A$ rather than $B$ at a randomly picked time $t$
    is given by 
    \begin{align}
        \rho_A = \int_{\mathcal{M}\backslash B}(1 - q(x))\mu(x)\,dx.
    \end{align}
    \item \textbf{Escape rate}:  The escape rate \(k_{AB}\) from the set \(A\) to $B$ is defined as 
    \begin{equation}
        \label{eq:escrate}
        k_{AB} = \lim_{T\rightarrow\infty} \frac{N_{AB}}{T_A},  
    \end{equation}
    where $T_A$ is the total time within the interval of length $T$ during which the trajectory last hit $A$. It is related to the transition rate $\nu_{AB}$ via     
    \begin{equation}
        \label{eq:kAB}
        k_{AB} = \frac{{\nu_{AB}}}{{\rho_A}}.
    \end{equation}
\end{enumerate}

\subsection{Solving the committor problem with TMD map}
\label{subsec: solving cp with tmd}
The application of TPT relies on solving the committor problem \eqref{eq: committor bvp}. The
committor BVP \eqref{eq: committor bvp} can be solved analytically only in special cases. Therefore, the development of numerical methods for solving the committor problem is very important. For low-dimensional cases, i.e. when $m=2$ or $3$, one may use a finite difference or finite element method to approximate the solution to \eqref{eq: committor bvp} by discretizing $\mathcal{L}$ on a finite mesh in $\mathbb{R}^m$. For dimensions $m \geq 4$, meshless methods of discretizing $\mathcal{L}$ are needed. Such methods for solving the committor problem
include a variety of neural network-based solvers using the variational
formulation \eqref{eq: var_comm} \cite{khoo2019solving, li2019computing}, the finite expression method \cite{song2023finite}, and solvers based on diffusion maps \cite{Trstanova2019LocalAG, evans2021computing, evans2022computing}. Specifically, TMDmap can be applied to solve the committor problem as follows. Let $L:= L^{(n)}_{\epsilon,\mu}$ be the generator obtained through the TMDmap algorithm outlined in Section \ref{sec: tmdmaps} with input dataset $\mathcal{X}(n) := \{x_i\}_{i=1}^{n}$. Furthermore, let the vector $q_{n, \epsilon}$ be the $n$-dimensional vector denoting the desired approximation to the committor at the points $\{x_i\}$ where $[q_{n,\epsilon}]_i \approx q(x_i)$. 
Let $\mathcal{I} = \{i \mid x_i \in \mathcal{M}_{AB}\}$ and $\mathcal{D} = \{i \mid x_i \in A \cup B\}$. Reindexing the set $\mathcal{X}(n)$ so that $\mathcal{I} = \{1, \ldots j\}$ and $\mathcal{D} := \{j+1,\ldots,n\}$ gives a convenient block form of the numerical committor 
\begin{align}
    q_{n,\epsilon} = \left[\begin{array}{c}q^{\mathcal{I}}_{n,\epsilon}\\ q^{\mathcal{D}}_{n,\epsilon}\end{array}\right]\quad{\rm where}\quad [q^{\mathcal{D}}_{n,\epsilon}]_i=\begin{cases}0,&x_i\in A\\ 1,& x_i\in B\end{cases},\label{eq:qne}
\end{align}
and the generator matrix $L$
\begin{align}
    L := \begin{bmatrix} L^{\mathcal{I}\mathcal{I}} & L^{\mathcal{I}\mathcal{D}} \\ L^{\mathcal{D} \mathcal{I}} & L^{\mathcal{D}\mathcal{D}} \end{bmatrix}.
\end{align}
The discrete version of the committor problem \eqref{eq: committor bvp} is 
\begin{align}
    \Big[L^{\mathcal{I}\mathcal{I}} \quad L^{\mathcal{I}\mathcal{D}}\Big]
    \begin{bmatrix} q^{\mathcal{I}}_{n,\epsilon} \\ q^{\mathcal{D}}_{n,\epsilon} \end{bmatrix} = \mathbf{0},
\end{align}
or, equivalently, 
\begin{align}
    L^{\mathcal{I}\mathcal{I}} q^{\mathcal{I}}_{n,\epsilon} = -L^{\mathcal{I}\mathcal{D}}q^{\mathcal{D}}_{n,\epsilon}. \label{eq: solving committor problem algorithm}
\end{align}
A more general Dirichlet BVP on $\mathcal{M}$,
\begin{align}
    \mathcal{L}u(x) &= f, \quad x \in  \mathcal{M} \setminus \Omega, \notag \\
    u &= g, \quad x \in \partial \Omega, \label{eq: general boundary value problem}
\end{align}
can be solved by TMDmap using a similar linear system. 
The TMDmap numerical solution to \eqref{eq: general boundary value problem} is the vector $u_{n,\epsilon}$ that satisfies
\begin{align}
4\beta^{-1}L^{\mathcal{I}\mathcal{I}}u^{\mathcal{I}}_{n,\epsilon} = f^{\mathcal{I}} -4\beta^{-1}L^{\mathcal{I} \mathcal{D}}g,\quad u_{n,\epsilon}^{\mathcal{D}}= g. \label{eq:unum}
\end{align}
Note that $L^{\mathcal{I}\mathcal{I}}$ is the submatrix of $L$ obtained by choosing the rows and columns with indices from $\mathcal{I}$. $L^{\mathcal{I}\mathcal{I}}$ is the Dirichlet Laplacian in spectral graph theory and is known to be invertible (due to, for instance, the Gershgorin disk theorem). Thus the above linear system is well-posed.


\section{Results}
\label{sec: results}
In this section, we present our 
theorems quantifying the convergence of  $L^{(n)}_{\epsilon, \mu}f$ to $\mathcal{L}f$ as $\epsilon\rightarrow 0$ and $n\rightarrow\infty$.
Their proofs utilize techniques from the theory of the convergence of Graph Laplacians to the differential operators on a manifold $\mathcal{M}$ \cite{belkin2006convergence,singer2006graph,hein2007graph}. 
For the sake of easier navigation through our theorems, we place all long proofs in  Section \ref{sec: results_proofs} devoted to proofs. 

\subsection{Problem setup} 
To set up the main problem in this paper, we make the following assumptions. 
\begin{assumption}
\label{Ass1}
    The point cloud $\mathcal{X}(n)$ is sampled i.i.d. from the density $\rho$ supported on $\mathcal{M}$. 
\end{assumption}
\begin{assumption}
\label{Ass2}
    Let $x\in\mathcal{M}$ be an arbitrary point. It can but does not have to belong to the point cloud $\mathcal{X}(n)$. 
\end{assumption}
\begin{assumption}
\label{Ass3}
    We fix $f \in C^{\infty}(\mathcal{M})$ to be an arbitrary function. Abusing notation, we will denote both the smooth function $f$ and the vector $[f(x_i)]_{i=1}^n$ by the same symbol $f$.
\end{assumption}
\begin{assumption}
\label{Ass4}
    Let $V: \mathcal{M} \to \mathbb{R}$ be a smooth potential function such that SDE \ref{eq: old} admits the invariant Gibbs density $\mu\propto \exp(-\beta V)$. The measure $\mu$ is chosen as the target measure. The sampling density $\rho$ is such that $\mu$ is absolutely continuous with respect to $\rho$. 
\end{assumption}
Let $L^{(n)}_{\epsilon,\mu}$, $\mathcal{L}_{\epsilon,\mu}$, and $\mathcal{L}$ be the generators defined by \eqref{Lmatr}, \eqref{Leps}, and \eqref{eq: generator} respectively. 
The law of large numbers implies  the convergence $L_{\epsilon, \mu}^{(n)}f(x)\longrightarrow \mathcal{L}_{\epsilon,\mu}f(x) $  as $n\rightarrow \infty$, while Theorem \ref{Lthm} \cite{banisch2020diffusion} guarantees the convergence $\mathcal{L}_{\epsilon,\mu}f(x)\longrightarrow \tfrac{\beta}{4} \mathcal{L}f(x) $ as $\epsilon\rightarrow 0$.

The bias and the variance errors are defined by:
\begin{align}
    \text{Bias error:}\qquad &|4\beta^{-1}\mathcal{L}_{\epsilon,\mu}f(x) - \mathcal{L}f(x)| \\
    \text{Variance error:}\qquad &|L^{(n)}_{\epsilon,\mu}f(x) - \mathcal{L}_{\epsilon,\mu}f(x)| 
\end{align}

\subsection{Bias Error}
\label{subsec: bias error}
The bias error formula including the prefactors for the $O(\epsilon)$ term is \eqref{Lop} is established in the following theorem.

\begin{theorem}[Bias Error]
    \label{thm: bias error}
    Under Assumptions \ref{Ass1}--\ref{Ass4}, the bias error is given by
    \begin{align}
    4\beta^{-1}\mathcal{L}_{\epsilon, \mu} - \mathcal{L} &= \frac{\epsilon}{4}\left[\mathcal{Q}\left(f\mu^{1/2}\right)- f \mathcal{Q}(\mu^{1/2}) \right] + \frac{\epsilon}{16}\left(2\nabla f \cdot \nabla \left(\mu^{1/2}\omega \right) + (\mu^{1/2}\omega)\Delta f \right) \notag \\
    &- \frac{\epsilon}{16}\left(2\nabla f \cdot \nabla \left(\mu^{1/2}\frac{\Delta \rho}{\rho} \right) + \left(\mu^{1/2}\frac{\Delta \rho}{\rho}\right)f \right) \notag \\
    &+ \frac{\epsilon}{16}\left[\frac{\Delta (\mu^{1/2})}{\mu^{1/2}} - \left(\frac{\Delta \rho}{\rho} - \omega\right)\right]\left[f\frac{\Delta (\mu^{1/2})}{\mu^{1/2}} - \frac{\Delta (\mu^{1/2}f)}{\mu^{1/2}}\right] + O(\epsilon^2) \label{eq: full bias error}  \\
    &= :\epsilon\left(\mathcal{B}_{1}[f,\mu](x) + \mathcal{B}_{2}[f,\mu,\rho](x) + \mathcal{B}_{3}[f,\mu,\rho](x)\right) + O(\epsilon^2).\notag
    \end{align}
 \end{theorem}
A detailed proof of Theorem \ref{thm: bias error} is found in Section \ref{subsec: bias error proofs}. Its pivotal component is the second-order kernel expansion formula presented in Lemma \ref{lem: expansion} below whose proof is also postponed until Section \ref{subsec: bias error proofs}. Lemma \ref{lem: expansion} is a second-order upgrade of the first-order kernel expansion formula from \cite{coifman2006diffusion,hein2005graphs} used in several previous works \cite{berry2016variable,singer2006graph,cheng2020convergence}.
The first-order kernel expansion formula from \cite{coifman2006diffusion,hein2005graphs} is insufficient for our goal of establishing how the bias error depends on $\rho$, $\mu$, and $f$ along with $\epsilon$ because it leads to hiding these dependencies into an $O(\epsilon)$ term in the expansion of the generator $\mathcal{L}_{(\cdot)}$.

\begin{lemma}[Second-Order Kernel Expansion Formula] 
\label{lem: expansion}
    Let $f \in C^{\infty}(\mathcal{M})$ and $\mathcal{G}_{\epsilon}$ be the integral operator defined by 
    \begin{align}
        \mathcal{G}_{\epsilon}f(x) = \int_{\mathcal{M}}k_{\epsilon}(x,y)f(y)\,d{\rm vol}(y).
    \end{align}
    Then, for small enough $\epsilon$, $\mathcal{G}_{\epsilon}f$ admits the following expansion at $x$: 
    \begin{align}
        (\pi \epsilon)^{-d/2}\mathcal{G}_{\epsilon}f(x) = f(x) + \frac{\epsilon}{4}(\Delta f(x) - \omega(x)f(x)) + \frac{\epsilon^2}{4}\mathcal{Q}f(x) + O(\epsilon^3), \label{eq: expansion equation}
    \end{align}
    where $\mathcal{Q}$ is a fourth-order differential operator on $\mathcal{M}$.
    In particular, if $\mathcal{M}$ is isometric to $\mathbb{R}^d$ in a neighborhood of $x$ (i.e $\mathcal{M}$ is locally flat) then
    \begin{align}
        \mathcal{Q} = \partial^{i}_{i}\partial^{j}_{j} + 2\partial^{ii}_{ii}. \label{eq: expansion for R^d}
    \end{align}
\end{lemma}

 \subsection{Reducing bias error}
Formula \eqref{eq: full bias error}, or, equivalently, \eqref{eq: bias error term 1}, \eqref{eq: bias error term 2}, and \eqref{eq: bias error term 3}, allow us to identify the situations in which some of the prefactors of the bias error are zero. This, in turn, \emph{possibly} speeds up the pointwise convergence of the TMD map estimator $L^{(n)}_{\epsilon,\mu}f(x)$. 
We say \emph{possibly} because the prefactors $\mathcal{B}_1, \mathcal{B}_2$, and $\mathcal{B}_3$ may have opposing signs and produce favorable cancellations. Nonetheless, several bias error terms zero out in the following important cases.
\begin{itemize}
    \item If the sampling density $\rho$ is uniform, the term $\mathcal{B}_2[f,\mu,\rho]$ in \eqref{eq: full bias error} is zero.
    \item If $f$ is such that $\mathcal{L}f = 0$, i.e., $f$ is the committor, then the term $\mathcal{B}_3[f,\mu,\rho]$ in \eqref{eq: full bias error} is zero. This follows from the observation that the second factor in $\mathcal{B}_3[f,\mu,\rho]$ is $\mathcal{L}f$.
    \item If the manifold $\mathcal{M}$ is flat, e.g. $\mathcal{M}$ is a flat $d$-dimensional torus, then $\omega = 0$. Then the second summand of $\mathcal{B}_1[f,\mu]$ in \eqref{eq: full bias error} is zero as $\omega = 0$.
\end{itemize}

\begin{corollary}\label{cor: reducing bias}
  
    If  $\mathcal{L}f = 0$, $\mathcal{M}$ is flat, i.e., $\omega = 0$, and the sampling density $\rho$ is uniform, i.e.,  $\rho(x) = const$, then the bias error reduces to
    \begin{equation}
        \label{eq:bias_error_reduced}
         4\beta^{-1}\mathcal{L}_{\epsilon, \mu}f - \mathcal{L}f =  \frac{\epsilon}{4}\left[\mathcal{Q}\left(f\mu^{1/2}\right)- f \mathcal{Q}(\mu^{1/2}) \right] + O(\epsilon^2),
    \end{equation}
    where $\mathcal{Q}$ is the fourth-order differential operator given by \eqref{eq: expansion for R^d}.
    
\end{corollary}


\subsection{Variance Error}
\label{subsec: variance error}

The variance error depends on the point cloud, a set of $n$ independent identically distributed (i.i.d.) points sampled from the density $\rho$. Therefore, it is natural to estimate it by obtaining a concentration inequality of the form
\begin{align}
    \mathbb{P}\left(|L^{(n)}_{\epsilon,\mu}f(x) - \mathcal{L}_{\epsilon,\mu}f(x)| > \alpha \right) \leq \mathcal{E}(n,\alpha, \epsilon,f,\mu). \label{eq: concentration inequality}
\end{align}

An important general-purpose concentration inequality is Bernstein's inequality, which states that if $X_i$, $i=1,\ldots,n$, are independent random variables with zero mean and the absolute values bounded by $|X_i|\le M$, $1\le i\le n$, then
\begin{equation}
    \label{eq:Bernstein0}
    \mathbb{P}\left(\sum_{i=1}^n X_i \ge t\right) \le \exp\left[-\frac{t^2}{2\sum_{i=1}^n \mathbb{E}[X_i^2] +\tfrac{2}{3}Mt}\right]\quad\forall t>0.
\end{equation}

Inequality \eqref{eq:Bernstein0} allows one to derive bounds for sums of random variables that hold with high probability. In the following subsections of Section \ref{subsec: variance error}, we will construct such a bound for $|L^{(n)}_{\epsilon,\mu}f(x) - \mathcal{L}_{\epsilon,\mu}f(x)|$ in a sequence of steps starting from bounding the discrepancy between the discrete, $\rho_{\epsilon}^{(n)}$, and continuous, $\rho_{\epsilon}$, kernel density estimates (KDEs). 

It is worth noting that in a commonly used strategy for derivation of the variance errors \cite{singer2006graph,berry2016local,berry2016variable}, the point $x_i$ at which the error estimate is computed is removed from the point cloud. Here we show that this is not necessary. Our argument is somewhat more involved but it results in a sharp error estimate.

\subsubsection{Bounds for the error in the kernel density estimate}
Let $x\in\mathcal{M}$ be any point on the manifold. Our first goal is to bound the error in the KDE $\rho_{\epsilon}(x)$. Lemma \ref{lem: expansion} implies that
\begin{equation}
    \label{eq:rhoeps}
    \rho_{\epsilon}(x) := \int_{\mathcal{M}}k_{\epsilon}(x,y)\rho(y)d{\rm vol}(y) = (\pi\epsilon)^{d/2}\left(\rho(x) + O(\epsilon)\right).
\end{equation}
Its discrete estimate $\rho^{(n)}_{\epsilon}(x)$ on the point cloud $\mathcal{X}(n)$ is
\begin{equation}
\label{eq:rhoepsn}
\rho^{(n)}_{\epsilon}(x) := \frac{1}{n}\sum_{j=1}^nk_{\epsilon}(x,x_j).
\end{equation}


The discrepancy between $\rho_{\epsilon}(x)$ and $\rho^{(n)}_{\epsilon}(x)$ is quantified by the following theorem.
\begin{theorem} [Discrete kernel density estimate]
    \label{thm:M1}
    Let $\mathcal{X}(n) = \{x_j\}_{j=1}^n$ be a point cloud sampled from the density $\rho(x)$. Let $x\in\mathcal{M}$ be either any fixed point on the manifold $\mathcal{M}$ or any point selected from the point cloud $\mathcal{X}(n)$. Then for $\epsilon\rightarrow 0$ and $n\rightarrow\infty$ so that
    \begin{equation}
        \label{eq:thmM1_epsn}
        \lim_{\substack{n\rightarrow\infty\\\epsilon\rightarrow 0}}\frac{n\epsilon^{2+d/2}}{\log n} = \infty,
    \end{equation}
    the discrepancy between $\rho_{\epsilon}(x)$ and $\rho^{(n)}_{\epsilon}(x)$ is bounded, with probability at least $1-n^{-4}$, by
    \begin{equation}
        \label{eq:thmM1}
        \left|\rho^{(n)}_{\epsilon}(x) - \rho_{\epsilon}(x)\right| < 5(\pi\epsilon)^{d/2}\rho^{1/2}(x) \epsilon^2\alpha
    \end{equation}
    where
    \begin{equation}
        \label{eq:thmM1_alpha}
        \alpha = \frac{1}{(2\pi)^{d/4}}\sqrt{\frac{\log n}{n\epsilon^{4+d/2}}}.
    \end{equation}  
 \end{theorem}
The proof of Theorem \ref{thm:M1} is detailed in Section \ref{sec:proofThm9}.

\subsubsection{Bound for the discrepancy $L_{\epsilon,\mu}^{(n)}f$ and $L_{\epsilon,\mu}f$}
Our next goal is to estimate the discrepancy between the discrete generators,  $L_{\epsilon,\mu}^{(n)}f$ and $L_{\epsilon,\mu}f$  obtained, respectively,  using the discrete, $\rho_{\epsilon}^{(n)}$, and continuous, $\rho_{\epsilon}$, KDEs. 

Let $\mathcal{X}(n)=\{x_j\}_{j=1}^n$ be a point cloud sampled from the density $\rho$. Let $x\in\mathcal{M}$ be either an arbitrary fixed point on the manifold $\mathcal{M}$ or any point of $\mathcal{X}(n)$. To facilitate our calculations, we introduce random variables $F_j^{(n)}$, $F_j$, $G_j^{(n)}$, and $G_j$ defined as
\begin{align}
    &F^{(n)}_j = \frac{k_{\epsilon}(x,x_j)\mu^{1/2}(x_j)f(x_j)}{\rho_{\epsilon}^{(n)}(x_j)},\quad &
    F_j = \frac{k_{\epsilon}(x,x_j)\mu^{1/2}(x_j)f(x_j)}{\rho_{\epsilon}(x_j)},\label{F_j}\\
    &G^{(n)}_j = \frac{k_{\epsilon}(x,x_j)\mu^{1/2}(x_j)}{\rho_{\epsilon}^{(n)}(x_j)},\quad 
    &G_j = \frac{k_{\epsilon}(x,x_j)\mu^{1/2}(x_j)}{\rho_{\epsilon}(x_j)}.\label{G_j}
\end{align}
The discrete Markov operators $P^{(n)}_{\epsilon,\mu}$ and $P_{\epsilon,\mu}$ acted on the function $f$ and evaluated at $x$ are, respectively,
\begin{equation}
    \label{eq:Pndef}
  \left[ P^{(n)}_{\epsilon,\mu}f\right](x) = \frac{\sum_{j=1}^n F^{(n)}_j}{\sum_{j=1}^n G^{(n)}_j}\quad{\rm and}\quad 
    \left[P_{\epsilon,\mu}f \right](x)= \frac{\sum_{j=1}^n F_j}{\sum_{j=1}^n G_j}.
\end{equation}
The discrete generators $L_{\epsilon,\mu}^{(n)}f$ and $L_{\epsilon,\mu}f$ applied to $f(x)$ are, respectively,
\begin{equation}
    \label{eq:Lndef}
     \left[L^{(n)}_{\epsilon,\mu}f\right](x) =\frac{1}{\epsilon}\left(  \left[ P^{(n)}_{\epsilon,\mu}f\right](x) - f(x)\right)\quad{\rm and}\quad 
    \left[L_{\epsilon,\mu}f\right](x) =\frac{1}{\epsilon}\left( \left[ P_{\epsilon,\mu}f\right](x) - f(x)\right). 
\end{equation}
\begin{theorem}[Discrepancy between the discrete TMDmap generators with discrete and continuous KDEs]
\label{thm:M2}
Under the settings of Theorem \ref{thm:M1}, the discrepancy between  $\left[L^{(n)}_{\epsilon,\mu}f\right](x)$ and $ \left[L_{\epsilon,\mu}f\right](x) $ defined by \eqref{eq:Lndef} is bounded, with probability at least $1-n^{-3}$, by 
\begin{equation}
    \label{eq:thmM2_L}
    \left|\left[L^{(n)}_{\epsilon,\mu}f\right](x) -\left[L_{\epsilon,\mu}f\right](x)  \right| < C\epsilon\alpha(1+o(1))
\end{equation}
where the constants $\alpha$ and $C$ are given, respectively, by
\begin{equation}
\label{eq:thmM2_alpha}
\alpha = \frac{1}{(2\pi)^{d/4}}\sqrt{\frac{\log n}{n\epsilon^{4+d/2}}}
\end{equation}
and
\begin{equation}
    \label{eq:themM2_C}
    C = \frac{10{|f(x)|}}{\rho^{1/2}(x)}.
\end{equation}
\end{theorem}
The proof of Theorem \ref{thm:M2} is found in Section \ref{sec:proofThm10}.

\subsubsection{Bound for the discrepancy between $L_{\epsilon,\mu}$ and $\mathcal{L}_{\epsilon,\mu}$}
Let $x_i\in\mathcal{X}(n)$ be any selected point in the point cloud. Our next goal is to estimate the discrepancy between $[L_{\epsilon,\mu}f](x_i)$ and $[\mathcal{L}_{\epsilon,\mu}f](x_i)$ defined, respectively, by
\begin{equation}
    \label{eq:Lemf}
[L_{\epsilon,\mu}f](x_i) = \frac{1}{\epsilon} \left(\frac{\sum_{j=1}^n\frac{k_{\epsilon}(x_i,x_j)\mu^{1/2}(x_j)f(x_j)}{\rho_{\epsilon}(x_j)}}{\sum_{j=1}^n\frac{k_{\epsilon}(x_i,x_j)\mu^{1/2}(x_j)}{\rho_{\epsilon}(x_j)}} - f(x_i)\right) \equiv  
\frac{1}{\epsilon} \left(\frac{\sum_{j=1}^n F_j}{\sum_{j=1}^n G_j} - f(x_i)\right)
\end{equation}
and
\begin{equation}
    \label{eq:mathcalLemf}
[\mathcal{L}_{\epsilon,\mu}f](x_i) = \frac{1}{\epsilon} \left(\frac{\int_{\mathcal{M}}\frac{k_{\epsilon}(x_i,y)\mu^{1/2}(y)f(y)}{\rho_{\epsilon}(y)}dy}{\int_{\mathcal{M}}\frac{k_{\epsilon}(x_i,y)\mu^{1/2}(y)}{\rho_{\epsilon}(y)}dy} - f(x_i)\right) \equiv 
\frac{1}{\epsilon} \left(\frac{\mathbb{E}[F]}{\mathbb{E}[G]} - f(x_i)\right).
\end{equation}
Equations \eqref{eq:Lemf} and \eqref{eq:mathcalLemf} show that 
\begin{equation}
    \label{eq:LP_M3}
     \left[L_{\epsilon,\mu}f\right](x) -\left[\mathcal{L}_{\epsilon,\mu}f\right](x)   = \frac{1}{\epsilon}
     \left( \frac{\sum_{j=1}^n F_j}{\sum_{j=1}^n G_j}  - \frac{m_F}{m_G}\right).
\end{equation}
For brevity, $\mathbb{E}[F]$ and $\mathbb{E}[G]$ are denoted by $m_F$ and $m_G$ respectively. 
\begin{theorem}[Discrepancy between the discrete and continuous TMDmap generators with continuous KDE]
    \label{thm:M3}
    Let $\{x_j\}_{j=1}^n\equiv \mathcal{X}(n)$ be a point cloud of $n$ points sampled independently from the density $\rho$. Let $x_i\in\mathcal{X}(n)$ be any selected point. Let $\epsilon\rightarrow 0$ and $n\rightarrow\infty$ so that  
    \begin{equation}
        \label{eq:thmM3_en}        \lim_{\substack{n\rightarrow\infty\\\epsilon\rightarrow 0}}n\epsilon^{1+d/2}\log n = \infty.
    \end{equation}
    Let the function $f$ be smooth and bounded.
    Then the discrepancy between  $\left[L_{\epsilon,\mu}f\right](x_i)$ and $ \left[\mathcal{L}_{\epsilon,\mu}f\right](x_i) $ defined by \eqref{eq:Lemf} and \eqref{eq:mathcalLemf} respectively is bounded, with probability at least $1-n^{-3}$, by 
\begin{equation}
    \label{eq:thmM3_L}
    \left|\left[L_{\epsilon,\mu}f\right](x_i) -\left[\mathcal{L}_{\epsilon,\mu}f\right](x_i)  \right| < C\alpha\epsilon^{3/2}
\end{equation}
where the constants $\alpha$ and $C$ are given, respectively, by
\begin{equation}
\label{eq:thmM3_alpha}
\alpha = \frac{1}{(2\pi)^{d/4}}\sqrt{\frac{\log n}{n\epsilon^{4+d/2}}}
\end{equation}
and
\begin{equation}
    \label{eq:themM3_C}
    C = \frac{2\|\nabla f(x_i)\|}{\rho^{1/2}(x_i)}.
\end{equation}
\end{theorem}
The proof of Theorem \ref{thm:M3} is detailed in Section \ref{sec:proofThm11}.

\subsubsection{Bound for the variance error $L^{(n)}_{\epsilon,\mu}f - \mathcal{L}_{\epsilon,\mu}f$}
Theorems \ref{thm:M2} and \ref{thm:M3} allow us to estimate the variance error for TMDmap. The result is presented in Theorem \ref{thm:M4}.
\begin{theorem}
    \label{thm:M4}
    Let $\{x_j\}^n_{j=1} \equiv \mathcal{X}(n)$ be a point cloud of $n$ points sampled independently from
the density $\rho$. Let $x_i\in\mathcal{X}(n)$ be any selected point. Let $\epsilon\rightarrow 0$ and $n\rightarrow\infty$ so that
\begin{equation}
    \label{eq:thmM4_en}
    \lim_{\substack{n\rightarrow\infty\\\epsilon\rightarrow 0}}\frac{n\epsilon^{2+d/2}}{\log n} = \infty.
\end{equation}
 Let the density $\rho$ be bounded from above and away from zero, $0<\rho_{\min}\le\rho\le\rho_{\max}$, and the function $f$ be smooth and bounded.
Then the discrepancy $\left[L^{(n)}_{\epsilon,\mu}f\right](x_i)-\left[\mathcal{L}_{\epsilon,\mu}f\right](x_i) $ between the discrete TMDmap generator applied to $f$ and its continuous counterpart is bounded, with probability at least $1-2n^{-3}$, by 
\begin{equation}
    \label{eq:thmM4_L}
    \left|\left[L^{(n)}_{\epsilon,\mu}f\right](x_i) -\left[\mathcal{L}_{\epsilon,\mu}f\right](x_i)  \right| < \frac{\alpha\epsilon}{\rho^{1/2}(x_i)}\left(2\|\nabla f(x_i)\|\epsilon^{1/2} + 11{|f(x_i)|}\right)
\end{equation}
for large enough $n$ and small enough $\epsilon$, where the constant $\alpha$ is given by
\begin{equation}
\label{eq:thmM4_alpha}
\alpha = \frac{1}{(2\pi)^{d/4}}\sqrt{\frac{\log n}{n\epsilon^{4+d/2}}}.
\end{equation}
\end{theorem}
\begin{proof}
    The result readily follows from Theorems \ref{thm:M2} and \ref{thm:M3}. The definitions of $\alpha$ in these Theorems match. Condition \eqref{eq:thmM1_epsn} on $n$ and $\epsilon$ from Theorem \ref{thm:M2} implies condition \eqref{eq:thmM3_en} on $n$ and $\epsilon$ from Theorem \ref{thm:M3}. Condition \eqref{eq:thmM4_en} on $n$ and $\epsilon$ coincides with \eqref{eq:thmM1_epsn}. Therefore, the conditions of both Theorems   \ref{thm:M2} and \ref{thm:M3} hold. Furthermore, for small enough $\epsilon$ and large enough $n$, the term $10(1+o(1))$ in Theorem \eqref{thm:M2} is less than 11. Inequalities \eqref{eq:thmM2_L} and \eqref{eq:thmM3_L} fail to hold with probabilities at most $n^{-3}$. Hence, with probability at least $1-2n^{-3}$, we have
    \begin{align}
       &\thinspace \left|\left[L^{(n)}_{\epsilon,\mu}f\right](x_i) -\left[\mathcal{L}_{\epsilon,\mu}f\right](x_i)  \right| \notag\\
       \le &\thinspace \left|\left[L^{(n)}_{\epsilon,\mu}f\right](x_i) -\left[{L}_{\epsilon,\mu}f\right](x_i)  \right| + \left|\left[L_{\epsilon,\mu}f\right](x_i) -\left[\mathcal{L}_{\epsilon,\mu}f\right](x_i)  \right| \notag \\
       < &\thinspace
        \frac{\alpha\epsilon}{\rho^{1/2}(x_i)}\left(\|\nabla f(x_i)\|\epsilon^{1/2} + 11{|f(x_i)|}\right)
    \end{align}
    as desired.
\end{proof}
Theorem \ref{thm:M4} allows us to make several observations.
\begin{itemize}
\item The bound of the pointwise variance error \eqref{eq:thmM4_L} is inversely proportional to $\rho^{1/2}$. Therefore, to minimize this prefactor, one needs to choose the uniform sampling density $\rho$.
\item If the function $f$ is the committor, it is smooth  and assumes values between 0 and 1. However, if the temperature $\beta^{-1}$ is small, the committor may have a large gradient in the transition region. Therefore, at a finite bandwidth parameter $\epsilon$, either term in the parentheses in the right-hand side of \eqref{eq:thmM4_L} can be larger.
\item The KDE established in Theorem \ref{thm:M1} is valid for any sampling density $\rho$. In contrast, the error bound for the variance error in the generator $\mathcal{L}$ applied to a smooth function $f$ is derived under the assumption that the sampling density $\rho$ is bounded away from zero: $0<\rho_{\min}\le\rho\le \rho_{\max}$. While the upper bound is always finite in real-life applications, the lower bound $\rho_{\min}$ is often zero. Enhanced sampling algorithms such as metadynamics or temperature acceleration generate point clouds supported within a compact region in many applications of interest. This issue can be readily resolved as follows. We use all sampled points to obtain the KDE $\rho_{\epsilon}^{(n)}$. Furthermore, we construct the discrete generator $L_{\epsilon,\mu}^{(n)}$ using all points. Then, we select a compact region $\Omega\subset\mathcal{M}$ in which the sampling density is above a desired threshold $\rho_{\min}$ and use the computed solution only in this region. The error bound obtained in Theorem \ref{thm:M4}  applies in this region. 

\item In order for $\alpha$ to be $O(1)$, the number of points in the point cloud, $n$, must scale so that
\begin{equation}
    \label{eq:neps_scaling}
    \frac{(2\pi)^{d/2}n}{\log n} = \epsilon^{-4-d/2}.
\end{equation}
This means that choosing larger $\epsilon$ can significantly reduce the required size of the point cloud to guarantee the same bound for the variance error. For example, if one chooses $n=10^4$ points, a value of $\epsilon$ predicted by \eqref{eq:neps_scaling} is about $0.17$ -- see Fig. \ref{fig:neps_scale}. In applications, $\epsilon$ can be an order of magnitude smaller \cite{evans2021computing}.

\end{itemize}

\begin{figure}[h]
    \centering
    \includegraphics[width=0.45\textwidth]{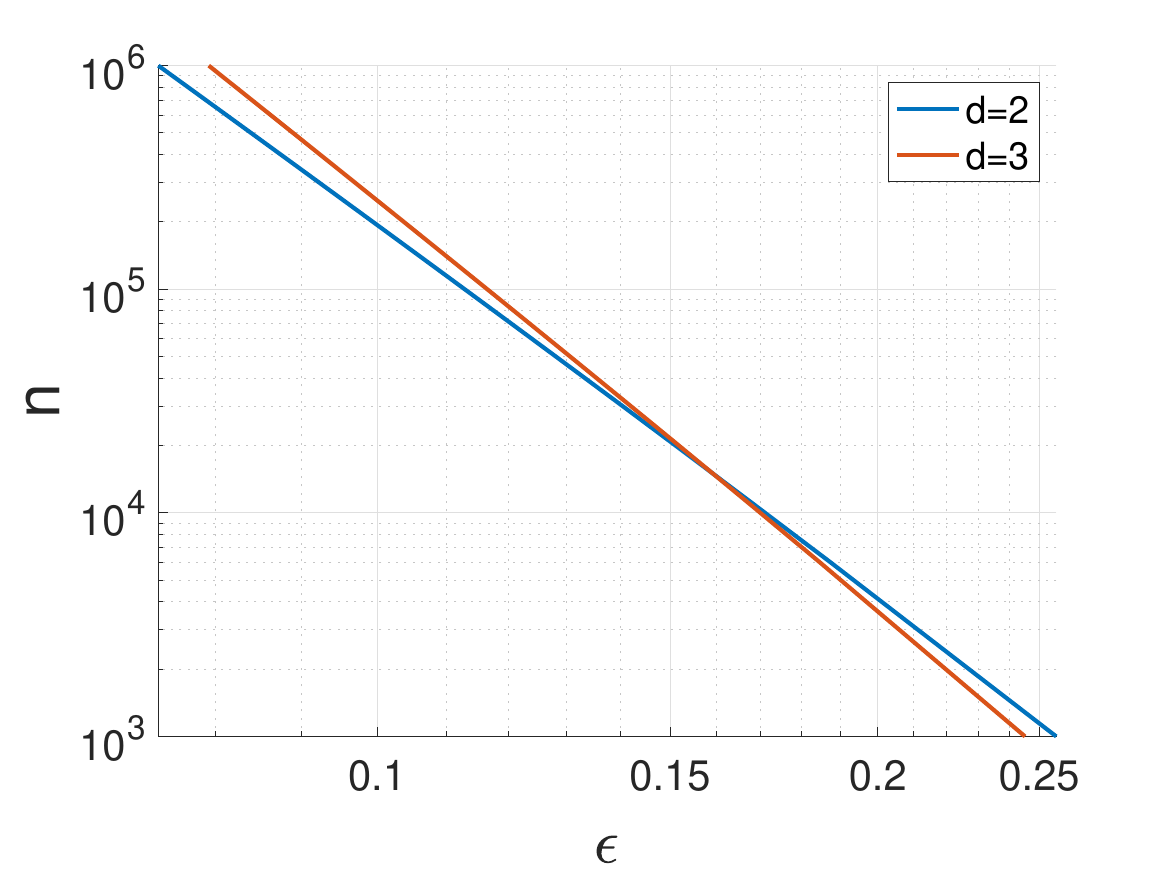}
    \caption{The relationship between the number of sample points $n$ on the $d$-dimensional manifold $\mathcal{M}$ and the bandwidth parameter $\epsilon$ at which the parameter $\alpha$ in Theorems \ref{thm:M2}, \ref{thm:M3}, and \ref{thm:M4} is 1.} 
    \label{fig:neps_scale}
\end{figure}

\subsection{Typical error in the kernel density estimate}
The error bound obtained in Theorem \ref{thm:M4} 
is based on Bernstein's general inequality. It 
is significantly larger than the typical error observed in our numerical experiments. In this section, we find the typical variance error in estimates by computing means and standard deviations of appropriate random variables. { Unfortunately, the scaling for $n$ and $\epsilon$ for the typical error derived in Theorem \ref{thm:M5} is similar to the one in \eqref{eq:neps_scaling}. Therefore, the reason for a good performance of TMDmap with uniform $\rho$ achieved via $\delta$-net is different.  Some insights on this issue will be given in Section \ref{sec:trigrid}.}

\begin{theorem}
    \label{thm:M5}
    Let $\mathcal{X}(n)=\left\{x_j\right\}_{j=1}^n$ be a point cloud sampled from the density $\rho$. Let $x_i\in\mathcal{X}(n)$ be an arbitrarily selected point.
    The bias and the variance of the error  in the kernel density estimate at $x_i$ are
    \begin{equation}
        \label{eq:exp_err_kde}
        \mathbb{E}\left[\rho_{\epsilon}^{(n)}(x_i) - \rho_{\epsilon}(x_i)\right] = \frac{1}{n} \left(1-(\pi\epsilon)^{d/2}\left[\rho(x_i)+O(\epsilon)\right]\right),
    \end{equation}
    \begin{equation}
        \label{eq:var_err_kde}
        {\sf Var}\left(\rho_{\epsilon}^{(n)}(x_i) - \rho_{\epsilon}(x_i)\right) = \frac{1}{n}\left(\frac{\pi\epsilon}{2}\right)^{d/2} \left(\rho(x_i) + O(\epsilon^{\min\{1,d/2\}}) + O(n^{-1})\right).
    \end{equation}
\end{theorem}
\begin{proof}
    We will use expressions obtained in the proofs of Lemmas \ref{lemma:M1} and \ref{lemma:M2} in Section \ref{subsec: variance error proofs} below. Recall that the error in the KDE is given by 
    \begin{equation}
    \label{eq:rhoediff_M5}
\rho^{(n)}_{\epsilon}(x_i) - \rho_{\epsilon}(x_i) = 
\frac{1}{n}\sum_{j=1}^n\left(k_{\epsilon}(x_i,x_j) - \rho_{\epsilon}(x_i)\right) =: X_i +\sum_{j\neq i} X_j,
\end{equation}
where $X_j$, $j\neq i$, are independent random variables with mean zero and variance
\begin{equation}
    \label{eq:VarX_j_M5}
    {\sf Var}(X_j) = \frac{1}{n^2}\left[ \left(\frac{\pi\epsilon}{2}\right)^{d/2}\left(\rho(x_i) + O(\epsilon^{\min\{1,d/2\}}\right)\right].
\end{equation}
In contrast, $X_i$ is a deterministic number
\begin{equation}
    \label{eq:X_i_M5}
        X_i = \frac{1}{n}\left(k_{\epsilon}(x_i,x_i)-\rho_{\epsilon}(x_i)\right) = \frac{1}{n} \left(1-(\pi\epsilon)^{d/2}\left[\rho(x_i)+O(\epsilon)\right]\right).
\end{equation}
Therefore,
\begin{align}
\label{eq:exp_err_kde_M5}
    \mathbb{E}\left[\rho_{\epsilon}^{(n)}(x_i) - \rho_{\epsilon}(x_i)\right] = X_i +(n-1)\mathbb{E}[X_j] = \frac{1}{n} \left(1-(\pi\epsilon)^{d/2}\left[\rho(x_i)+O(\epsilon)\right]\right),
\end{align}
and
\begin{align}
        {\sf Var}\left(\rho_{\epsilon}^{(n)}(x_i) - \rho_{\epsilon}(x_i)\right) &= (n-1){\sf Var}(X_j) \notag \\
       & = \frac{1}{n}\left(\frac{\pi\epsilon}{2}\right)^{d/2} \left(\rho(x_i) + O(\epsilon^{\min\{1,d/2\}}) + O(n^{-1})\right).\label{eq:var_err_kde_M5}
\end{align}
Note that estimates \label{eq:exp_err_kde} and \eqref{eq:var_err_kde} do not assume any relationship between the rates at which $\epsilon\rightarrow 0$ and $n\rightarrow\infty$.
\end{proof}

The standard deviation of the error of the KDE \eqref{eq:var_err_kde} is
\begin{equation}
    \label{eq:kde_std}
    \sqrt{{\sf Var}\left(\rho_{\epsilon}^{(n)}(x_i) - \rho_{\epsilon}(x_i)\right)} = \frac{1}{n^{1/2}}\left(\frac{\pi\epsilon}{2}\right)^{d/4} \left(\rho^{1/2}(x_i) + o(1)\right).
\end{equation}
Equating it with the bound \eqref{eq:thmM1} for the KDE in Theorem \ref{thm:M1} and setting $\alpha =1$ we obtain
\begin{equation}
    \label{eq:comp_kde_scaling}
    \frac{1}{n^{1/2}}\left(\frac{\pi\epsilon}{2}\right)^{d/4} \left(\rho^{1/2}(x_i) + o(1)\right) = 3(\pi\epsilon)^{d/2}\rho^{1/2}(x_i)\epsilon^2.
\end{equation}
Therefore, to make the typical error as small as the bound in Theorem \ref{thm:M1} one needs to choose 
\begin{equation}
    \label{eq:n_scaling_new}
    n = \frac{(\pi\epsilon)^{d/2}\epsilon^4}{3\cdot2^{d/2}}.
\end{equation}
This scaling for $n$ is just a minor improvement over \eqref{eq:neps_scaling}. 

\subsection{Kernel density estimators for a regular triangular grid}
\label{sec:trigrid}
It is evident from the proof of Theorem \ref{thm:M5} that the inclusion of the point $x_i$ in the kernel density estimator at $x_i$  makes it biased. Therefore, it is reasonable to ask, should we remove the point $x_i$ from the kernel density estimator at $x_i$?
On one hand, the removal of $x_i$ makes the estimator unbiased. On the other hand, thinking of $\delta$-nets, the removal of $x_i$ creates a ``hole" around it making the point cloud uniform to a lesser extent in the neighborhood of $x_i$. 

To probe this issue quantitatively, we consider the following example. Let $\mathcal{X}(n)$ be a subset of an equilateral triangular grid with step $\delta=n_r^{-1}$ shaped as a right hexagon of radius 1 shown in Fig. \ref{fig:hex} (Left). The relationship between $n$ and $\delta$ is
\begin{equation}
    \label{eq:hex1}
    n = 1 + 6\sum_{k=1}^{n_r} k = 3n_r^2+3n_r + 1 = 3\delta^{-2}+3\delta^{-1} + 1.
\end{equation}
This hexagon is an instance of delta-net with parameter $\delta$, an ideal delta-net.
The area of the hexagon is $\tfrac{1}{2}3\sqrt{3}$ and the density is uniform.  Hence, by \eqref{eq:rhoeps}, 
\begin{equation}
    \label{eq:hex2}
    \rho_{\epsilon} = (\pi\epsilon)^{d/2}\frac{2}{3\sqrt{3}}
\end{equation}
if the integration were done over $\mathbb{R}^2$.
Let us calculate the KDE at a point $x_i\in\mathcal{X}(n)$ lying at a distance greater or equal to $3\sqrt{\epsilon}$ from the boundary of the hexagon with and without the removal of $x_i$ from the estimator. The minimal distance of $3\sqrt{\epsilon}$ from the boundary is motivated by the commonly used sparsification in diffusion map algorithms.
The KDE $\rho_{\epsilon}^{(n)}(x_i)$ is
\begin{equation}
    \label{eq:hex3} 
    \rho_{\epsilon}^{(n)}(x_i) = \begin{cases}\frac{1}{n}\left(1 + 6\sum_{k=1}^{\lfloor \frac{3\sqrt{\epsilon}}{\delta}\rfloor}\exp\left(-\frac{k^2\delta^2}{\epsilon}\right)\right),&\text{with $x_i$}\\ \frac{6}{n}\sum_{k=1}^{\lfloor \frac{3\sqrt{\epsilon}}{\delta}\rfloor}\exp\left(-\frac{k^2\delta^2}{\epsilon}\right),&\text{without $x_i$}.
    \end{cases}
\end{equation}
\begin{figure}[h]
    \centering
    \includegraphics[width=0.95\textwidth]{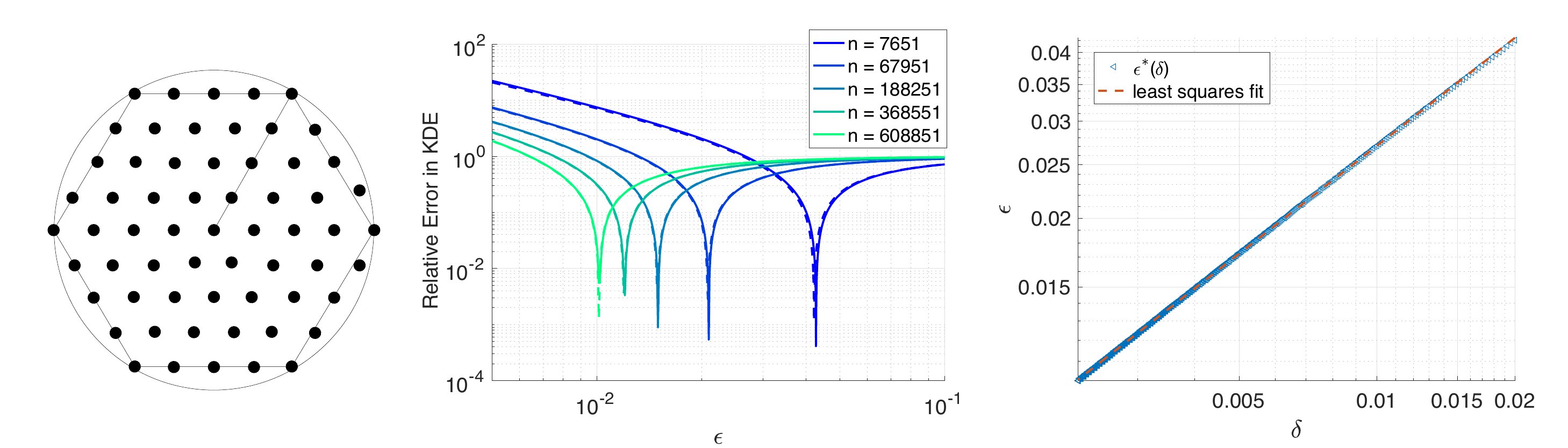}
    \caption{Left: A point cloud is a hexagon-shaped regular equilateral triangular grid. The radius of the circumcircle is one. Middle: Plots of biased (solid) and unbiased (dashed) kernel density estimates (KDEs) \eqref{eq:hex3} for $n_r = 50$, $150$, $250$, $350$, and $450$. The corresponding values of $n$ calculated using \eqref{eq:hex1} are shown in the legend. Right: The graph of the optimal value of $\epsilon$ that minimizes the relative error in $\rho_{\epsilon}$ versus $\delta =n_r$ and the least squares fit for it. } 
    \label{fig:hex}
\end{figure}
The plots of the relative error in the KDE given by \eqref{eq:hex3} with respect to $\rho_{\epsilon}$ given by \eqref{eq:hex2},
\begin{equation}
    \label{eq:rerr_kde}
    E_{rel} = \frac{\left|\rho_{\epsilon}^{(n)} - \rho_{\epsilon}\right|}{\rho_{\epsilon}}
\end{equation}
for the biased (with $x_i$, solid curves) and unbiased (without $x_i$, dashed curves) estimators  for $n_r = 50$, $150$, $250$, $350$, and $450$ are displayed in Fig. \ref{fig:hex} (Middle). The corresponding values of $n$ are given by \eqref{eq:hex1} and are shown in the legend. 

Inspection of these graphs allows us to make the following observations.
\begin{itemize}
    \item First, the relative errors for the biased (solid) and unbiased (dashed) are close to each other. A suboptimal choice of $\epsilon$ affects the relative error much stronger than the choice of the biased versus unbiased estimator. Unsurprisingly, at larger $\epsilon$, the biased estimator gives a slightly smaller error. On the other hand, at smaller $\epsilon$, the unbiased one is slightly more accurate. Therefore, in practice, the choice of a biased or unbiased estimator is not essential.
    \item These graphs show that the minimum of the relative error is achieved at a certain relationship between $\delta =n_r^{-1}$ and $\epsilon$. The plot in Fig. \ref{fig:hex} (Right) shows the optimal $\epsilon$ for each $\delta$. The least squares fit to a power law (the red dashed line in Fig. \ref{fig:hex} (Right))  yields the following relationship between $\delta$ and $\epsilon$:
\begin{equation}
\label{eq:hex4}
\epsilon^{\ast} = 0.5425\cdot\delta^{0.6509}.
\end{equation}
Relationship \eqref{eq:hex4} can be used as a guideline for choosing $\epsilon$ given the parameter $\delta$ of a $\delta$-net.
\item 
Equations \eqref{eq:hex4} and \eqref{eq:hex1} imply the following scaling for $\epsilon$ and $n$:
\begin{equation}
    \label{eq:hex5}
    n = 0.8829\epsilon^{-1.3018}.
\end{equation}
This scaling is much better than the one in \eqref{eq:neps_scaling}.
\end{itemize}

\begin{remark}
    In practice, $\delta$-nets are obtained by subsampling point clouds and do not have points at a distance closer than $\delta$. Such a subsampling breaks the i.i.d. hypothesis on which the KDE error in Theorem \ref{thm:M1} relies. The lack of the i.i.d. hypothesis requires the development of a different approach for the derivation of the variance error. This problem will be addressed in future work.  
\end{remark}

\subsection{Total error model}
\label{subsec: complete error}
Theorems \ref{thm: bias error} and \ref{thm:M4}  imply the total error bound given in Theorem \ref{thm:main} placed in Section \ref{sec: Intro}.  
Theorem \ref{thm:main} implies the following modes of convergence given the limits of $n$ and $\epsilon$: 
\begin{enumerate}
    \item To obtain almost sure convergence using the Borel-Cantelli lemma as $n \to \infty$ and  $\epsilon \to 0$ it suffices to take  
    \begin{align}
     \lim_{\substack{n\rightarrow\infty\\\epsilon\rightarrow 0}}\frac{n\epsilon^{2+d/2}}{\log n} = \infty.       
    \end{align}
    This would leave the variance error converging at $o(1)$ rate which is slower than $O(\epsilon)$. 
   \item To make the variance and bias errors both $O(\epsilon)$,  we must require the stronger limit 
   \begin{align}
        \lim_{\substack{n\rightarrow\infty\\\epsilon\rightarrow 0}}\frac{n\epsilon^{4+d/2}}{\log n} = \infty. \label{eq: slower convergence}   
    \end{align}
    The additional deceleration of $\epsilon$ is needed in \eqref{eq: slower convergence} to ensure the discrete kernel density estimate $\rho^{(n)}_{\epsilon}$ is within $O(\epsilon^2)$ of the continuous kernel density estimate $\rho_{\epsilon}$ for each point $x_i$ as $n \to \infty$ and $\epsilon \to 0$, and then the division by $\epsilon$ would get us to within $O(\epsilon)$ with summably decaying probability. 
\end{enumerate}

\subsection{ Numerical solution of Dirichlet BVPs with TMD maps}
\label{subsec: numerical solutions of bvp}
 We now focus on the numerical solution to the Dirichlet boundary value problem (BVP) \eqref{eq: general boundary value problem}. Here we make the following assumption about the domain $\Omega$: 
\begin{assumption}
\label{Ass5}
    $\Omega \subseteq \mathcal{M}$ is a closed subset of $\mathcal{M}$ with $C^1$ boundary. 
\end{assumption}
 Numerical solution of BVP \eqref{eq: general boundary value problem} on manifolds is an important application of diffusion maps-based estimators of diffusion operators such as $\mathcal{L}$ \cite{berry2016local, jiang2023ghost, antil2021fractional}. An instance of this BVP is the \emph{committor problem} \eqref{eq: committor bvp}. We will solve the BVP \eqref{eq: general boundary value problem} numerically using the TMDmap algorithm. The resulting solution will be denoted by $v\equiv v_{n,\epsilon}$. The exact solution to BVP \eqref{eq: general boundary value problem} will be denoted by $u$. Using a maximum principle argument, we can transfer the consistency estimate from Theorem \ref{thm:main} to an error estimate for quantifying the accuracy of the numerical solution $v_{n,\epsilon}$.
\begin{theorem}[Error bound for the numerical solution]
\label{thm: BVP error estimate}
Consider the point cloud $\mathcal{X}(n)$ with $n$ i.i.d. samples of $\rho$. Let $v_{n,\epsilon}$ be a TMDmap numerical solution to BVP \eqref{eq: general boundary value problem}, and let $u$ be the exact solution to \eqref{eq: general boundary value problem}.
Furthermore, let $n \to \infty$ and $\epsilon \to 0$ so that
    \begin{align}
    \label{eq:nepsthm9}
\lim_{\substack{n\rightarrow\infty\\\epsilon\rightarrow 0}}\frac{n\epsilon^{2+d/2}}{\log n} &= \infty.
    \end{align}
    Then there exist constants $\epsilon_0 > 0$ and $n_0 \in \mathbb{N}$ such that for all $\epsilon \leq \epsilon_0$ and $n \geq n_0$  with probability greater or equal to $1 - 4n^{-2} - \exp{\left(-\mathbb{E}[\mathbb{I}_{\Omega}]\right)}$ 
            \begin{align}
        |v_{n,\epsilon}(x_i) - u(x_i)| \leq \epsilon\left[ \max_{x\in\mathcal{M}}C_{u,\mu,\rho}(x) + 1\right]|\phi(x_i)| \label{eq: committor simplified model} 
    \end{align}
    where $\phi$ is the exact solution to $\mathcal{L}\phi = 1$, $x\in\mathcal{M}\backslash\Omega$ and $\phi = 0$, $x\in\Omega$, 
    and
    \begin{align} 
        C_{u,\mu,\rho}(x)&:= \max_{\substack{\epsilon\le\epsilon_0\\n\ge n_0}}\left[\frac{\alpha\left(2\|\nabla f(x)\|\epsilon^{1/2} + 11{|f(x)|}\right)}{\rho^{1/2}(x)}\right.      \label{eq:C}\\
    &+ \left.\left|\mathcal{B}_{1}[u,\mu](x) + \mathcal{B}_{2}[u,\mu, \rho](x) + \mathcal{B}_{3}[u,\mu, \rho](x) \right| + O(\epsilon)\right],\notag
    \end{align}
    where the functions $\mathcal{B}_j[\cdot](\cdot)$, $j=1,2,3$, are defined in Theorem \ref{thm:main}.
\end{theorem}

{ The proof of Theorem \ref{thm: BVP error estimate} is found in Section \ref{subsec: numerical solutions of bvp proofs}. Moreover, Theorem \ref{thm: BVP error estimate} combined with Corollary \ref{cor: reducing bias} straightaway implies the following error estimate for the solution to the committor problem}. 

\begin{corollary}\label{cor: committor error estimate}
    Under Assumptions \ref{Ass1}--\ref{Ass5}, let $q_{n,\epsilon}$ be the TMDmap numerical solution to the committor problem \eqref{eq: solving committor problem algorithm}. Let $n\rightarrow\infty$ and $\epsilon\rightarrow 0$ so that \eqref{eq:nepsthm9} holds. Suppose that $\mathcal{M}$ is flat and $\rho = {\rm vol}(\mathcal{M})^{-1}$ is the uniform density on $\mathcal{M}$. Then there exists a constant $C_q(x) > 0$ given by
    \begin{align} 
        C_{q}(x) &:= \max_{\substack{\epsilon\le\epsilon_0\\n\ge n_0}}\big[\alpha\,{\rm vol}(\mathcal{M})^{1/2}\left(2\|\nabla q(x)\|\epsilon^{1/2}+ 11 q(x)\right)\label{eq:C} \\
         &+ \left|\frac{1}{4}\left[\mathcal{Q}\left(q\mu^{1/2}\right)- q \mathcal{Q}(\mu^{1/2}) \right]\right| + O(\epsilon)\bigg], 
    \end{align}
    such that  
    \begin{align}
        |q_{n,\epsilon}(x_i) - q(x_i)| &\leq \epsilon\left[ \max_{x\in\mathcal{M}}C_{q}(x) + 1\right]|\phi(x_i)| + O(\epsilon^2) \label{eq: committor error estimate} 
    \end{align}
    with high probability.
\end{corollary}

\section{Results: proofs}\label{sec: results_proofs}
\subsection{Bias Error: proofs}\label{subsec: bias error proofs}
We begin with Lemma \ref{lem: expansion} that gives the second-order kernel expansion formula \eqref{eq: expansion equation}. 
\begin{proof}
    The strategy of proving \eqref{eq: expansion equation} is standard in manifold learning, see e.g. S. Lafon's dissertation \cite{lafon2004diffusion}. We present it here nonetheless, specifically to prove that the $O(\epsilon^2)$ prefactor is a fourth-order differential operator and that it is given by \eqref{eq: expansion for R^d} in the case if $\mathcal{M}$ is locally flat. 
    
    The computation of the Taylor expansion of the integral 
    \begin{align}
        \int_{\mathcal{M}}k_{\epsilon}(x,y)f(y)\,d{\rm vol}(y) \label{eq: main integral}
    \end{align}
in terms of $\epsilon$ is done in three steps.

\begin{enumerate}
    \item \textbf{Reduce to a ball:} We first note that the integral over the manifold in \eqref{eq: main integral} can be replaced with an integral over any arbitrarily small open set around $U \ni x$, incurring an error that decays faster than any polynomial in $\epsilon$ as $\epsilon \to 0$. This is formalized in \cite[Lemma 4.1]{belkin2008towards} as follows: 
    \begin{align}
         \int_{\mathcal{M} \setminus U}k_{\epsilon}(x,y)f(y)\,d{\rm vol}(y) = o(\epsilon^l)
    \end{align}
    for any $l > 0$. This error will turn out to be dominated by the $O(\epsilon)$ error of approximating the Laplacian. This approach allows us to focus on computing the integral over $U$, which can now be expressed as the image of the exponential map on a small ball in the tangent space. 
    \item \textbf{Use exponential coordinates:} To compute the integral 
    \begin{align}
        \int_{U}k_{\epsilon}(x,y)f(y)\,d{\rm vol}(y),
    \end{align}
    we write $U = \exp_x{(B_{r}(0))}$ where $\exp_x$ is the exponential map from $T_x\mathcal{M}$ to $\mathcal{M}$ taking $0 \to x$ and $r = \epsilon^{l}$ is smaller than $B_{\text{inj}(\mathcal{M})}$ where $0 < l < 1/2$. Note that the Euclidean coordinates $(u_1, \ldots, u_d)$ are mapped by $\exp_x$ to a geodesic: $tu_i \to \gamma_i(t)$ and $s(u)$ is the respective system of normal coordinates $U \to \mathcal{M}$ where $s(0) = x$. Since $y = s(u)$ is a diffeomorphism from $B_{r}(0)$ onto its image we can use a change of variables: 
    \begin{align}
        \int_{\exp(B_{r}(0))}k_{\epsilon}(x,y)f(y)\,d{\rm vol}(y) = \int_{B_r(0)}k_{\epsilon}(y(0), y(u))f(y(u))dy(u)\,du \label{eq: change of vars}.
    \end{align}
    Here $dy(u)$ is the Jacobian of the parametrization of $y$ in terms of the Euclidean coordinates $u$. To pull the integral back onto the tangent space, we need to expand the \emph{extrinsic} distance on $\mathcal{M}$ in terms of the distance on the tangent plane. Furthermore, we will write the Jacobian as a power series in terms of $u$. All of these expansions will have the following correction terms {\color{blue}\cite[Corollary 3, Prop. 6]{smolyanov2007chernoff}}: 
    \begin{align}
        \|s(u) - s(0)\|_{2}^{2} = \|u\|_{2}^{2} + \mathcal{C}^{1}_{x}(u), \quad \mathcal{C}^{1}_{x}(u) = O(\|u\|_{2}^{4}) \label{eq: correction distance}\\
        dy(u) = 1 + \mathcal{C}^{2}_{x}(u), \quad \mathcal{C}^{2}_{x}(u) = O(\|u\|_{2}^{2}) \label{eq: correction jacobian} 
    \end{align}
    \item \textbf{Analysis on Euclidean space:} After expressing the integral in \eqref{eq: change of vars} over a subset of Euclidean space, we compute these integrals for up to the fourth order power series expansions of $f\circ y$ in terms of $u$. For notational convenience, denote $F: U \to \mathbb{R}$ as $F:= f \circ s$.
    \begin{align}
        (\pi \epsilon)^{-d/2}\mathcal{G}_{\epsilon}f(x) &= \int_{B_r(0)}k_{\epsilon}(y(0), y(u))f(y(u))dy(u)\,du + o(\epsilon^l)\notag \\
        &= \int_{B_r(0)}k_{\epsilon}(y(0), y(u))f(y(u))dy(u)\,du  + o(\epsilon^l) \notag \\
        &= \int_{B_r(0)}\exp{\left(-\frac{\|u\|_{2}^{2}}{\epsilon} - \mathcal{C}^{1}_{x}(u)\right)}T_{4}F(u) (1 + \mathcal{C}^{2}_{x}(u))\,du + o(\epsilon^l)  \notag \\
        &= \int_{\mathbb{R}^d}\exp{\left(-\frac{\|u\|_{2}^{2}}{\epsilon} - \mathcal{C}^{1}_{x}(u)\right)}T_{4}F(u) (1 + \mathcal{C}^{2}_{x}(u))\,du + o(\epsilon^l) \notag \\
        &= f(x) + \frac{\epsilon}{4}\left(\Delta f(x) - \omega(x)f(x)\right) + \frac{\epsilon^2}{4}\mathcal{Q}f(x) + O(\epsilon^3). \label{eq: R^d integral with expansion} 
    \end{align}
    
    The integral above can be distributed and solved using Gaussian integral formulas. All odd powers in $u$ will cancel, resulting in the appearance of leading order terms $O(1), O(\epsilon),$ and $O(\epsilon^2)$, with the prefactors depending on $x,f,$ and $\mathcal{M}$. The function $\omega$ arises in the $O(\epsilon)$ term due to the correction factor $\mathcal{C}^1_{x}(u)$; the analytical formula for $\omega$ can be found in \cite[Prop. 7]{lafon2004diffusion}. Additionally, the $O(\epsilon^2)$ terms arise as the integration of fourth-order monomials against the Gaussian kernel. These monomials will arise multiplied with \emph{at most} the fourth derivatives of $f$. Consequently, the prefactor $\mathcal{Q}$ on the $O(\epsilon^2)$ term is a fourth-order (non-linear) differential operator. 
\end{enumerate}
Additionally, if $g \equiv I$ for a small enough neighborhood around $x$, then the correction terms in \eqref{eq: correction distance} and \eqref{eq: correction jacobian} are zero. As a consequence, \eqref{eq: R^d integral with expansion} simplifies to: 
\begin{align*}
    (\pi \epsilon)^{-d/2}\mathcal{G}_{\epsilon}f(x) &= \int_{\mathbb{R}^d}\exp{\left(-\frac{\|u\|_{2}^{2}}{\epsilon} - \mathcal{C}^{1}_{x}(u)\right)}T_{4}F(u) (1 + \mathcal{C}^{2}_{x}(u))\,du + o(\epsilon^l) \\
    &=\int_{\mathbb{R}^d}\exp{\left(-\frac{\|u\|_{2}^{2}}{\epsilon}\right)}T_{4}F(u)\,du + o(\epsilon^l) \\ 
    &= \int_{\mathbb{R}^d}\exp{\left(-\frac{\|u\|_{2}^{2}}{\epsilon}\right)}\Bigg(f(x) + \partial_if(x)u^i + \frac{1}{2}\partial_{ij}f(x)u^iu^j \\
    &+ \frac{1}{6}\partial_{ijk}f(x)u^iu^ju^k + \frac{1}{24} \partial_{ijkl}f(x)u^iu^ju^ku^l \Bigg)\,du + o(\epsilon^l) \\
    &= f(x) + \frac{\epsilon}{4}\Delta f(x) + \frac{\epsilon^2}{4}\left(\frac{1}{6}\left[\partial^{i}_{i}\partial^{j}_{j}f + 2\partial^{ii}_{ii}f(x)\right]\right) + O(\epsilon^3).
\end{align*}
The above argument also shows that under the condition $g \equiv I$ the term $\omega(x) = 0$. 
\end{proof}
Next, we prove Theorem \ref{thm: bias error} that gives the bias error \eqref{eq: full bias error}.

\begin{proof}(Theorem \ref{thm: bias error}) Lemma \ref{lem: expansion} implies that
\begin{align}
\rho_{\epsilon}(x) & = \int_{\mathcal{M}} k_{\epsilon}(x,y)\rho(y) dy \notag \\
&= (\pi\epsilon)^{d/2}\left[ \rho(x) + \frac{\epsilon}{4}\left(\Delta\rho(x) - \omega(x)\rho(x)\right) + \frac{\epsilon^2}{4}\mathcal{Q}(\rho) + O(\epsilon^3)  \right]. \label{c13}
\end{align}

Factoring out $\rho$ and using a geometric series expansion we obtain
\begin{align}
\rho_{\epsilon}^{-1}(x) &= (\pi\epsilon)^{-d/2}\rho^{-1}(x)\Big[1  + \left(-\frac{\epsilon}{4}\left(\frac{\Delta \rho}{\rho} - \omega\right) - \frac{\epsilon^2}{4}\frac{\mathcal{Q}(\rho)}{\rho} - O(\epsilon^3)\right) \notag \\ 
&+\left(-\frac{\epsilon}{4}\frac{\Delta \rho}{\rho} - \frac{\epsilon^2}{4}\frac{\mathcal{Q}(\rho)}{\rho} - O(\epsilon^3)\right)^2 + O(\epsilon^3)\Big] \label{eq:rhoeinv}\\
&= (\pi\epsilon)^{-d/2}\rho^{-1}(x)\Big[1 - \frac{\epsilon}{4}\left(\frac{\Delta \rho}{\rho} - \omega\right) - \frac{\epsilon^2}{4}\frac{\mathcal{Q}(\rho)}{\rho} + \frac{\epsilon^2}{16}\Big(\frac{\Delta \rho}{\rho}- \omega \Big)^2 + O(\epsilon^3)\Big]\notag
\end{align}

Now we compute the right-renormalized kernel operator that factors in the effect of the target density $\mu$:
\begin{align}
\label{eq:Kemuf}
K_{\epsilon,\mu}f(x) = \int_{\mathcal{M}}k_{\epsilon}(x,y)\mu^{1/2}(y)\rho_{\epsilon}^{-1}(y)f(y)\rho(y)d{\rm vol}(y)
\end{align}
Plugging \eqref{eq:rhoeinv} into \eqref{eq:Kemuf} we get:
\begin{align}
K_{\epsilon,\mu}f(x) 
&= \int_{\mathcal{M}}k_{\epsilon}(x,y)\frac{\mu^{1/2}f\rho}{(\pi\epsilon)^{d/2}\rho}\left[1 - \frac{\epsilon}{4}\left(\frac{\Delta \rho}{\rho} - \omega\right) - \frac{\epsilon^2}{4}\frac{\mathcal{Q}(\rho)}{\rho} \right.\notag \\
& \left.+ \frac{\epsilon^2}{16}\left(\frac{\Delta \rho}{\rho} - 
\omega\right)^2 + O(\epsilon^3)\right]\,d{\rm vol}(y) \nonumber \\
&= (\pi\epsilon)^{-d/2}\left[\int_{M} k_{\epsilon}(x,y)\mu^{1/2}f(y)\,d{\rm vol}(y) - \int_{\mathcal{M}}k_{\epsilon}(x,y)\mu^{1/2}f\frac{\epsilon}{4}\left(\frac{\Delta \rho}{\rho} - \omega\right)\,d{\rm vol}(y)\right. \nonumber \\ 
&- \int_{M}k_{\epsilon}(x,y)\mu^{1/2}f\frac{\epsilon^2}{4}\frac{\mathcal{Q}(\rho)}{\rho}\,d{\rm vol}(y)
\notag \\
& \left. + \int_{\mathcal{M}}k_{\epsilon}(x,y)\mu^{1/2}f\frac{\epsilon^2}{16}\left(\frac{\Delta \rho}{\rho} - \omega\right)^2 \, d{\rm vol}(y) + O(\epsilon^3)\right]. \label{Kemf0}
\end{align}
Computing the integrals in \eqref{Kemf0} using the second-order kernel expansion formula \eqref{eq: expansion equation} we obtain
\begin{align}
K_{\epsilon,\mu}f(x) &= \mu^{1/2}f + \frac{\epsilon}{4}\Delta(\mu^{1/2}f) + \frac{\epsilon^2}{4}\mathcal{Q}(\mu^{1/2}f) - \mu^{1/2}f\frac{\epsilon}{4}\left(\frac{\Delta \rho}{\rho} - \omega\right)  \label{Kf} \\
&- \frac{\epsilon^2}{16}\Delta\left(\mu^{1/2}f\left[\frac{\Delta \rho}{\rho} - \omega\right]\right) - \frac{\epsilon^2}{4}\mu^{1/2}f\frac{\mathcal{Q}(\rho)}{\rho} + \frac{\epsilon^2}{16}\mu^{1/2}f\left(\frac{\Delta \rho}{\rho} - \omega\right)^2 + O(\epsilon^3).\notag
\end{align}
Finally, we group the terms in \eqref{Kf} according to the order of $\epsilon$:
\begin{align}
\mathcal{K}_{\epsilon,\mu}f 
& = f\mu^{1/2} + \frac{\epsilon}{4}\left[\Delta(\mu^{1/2}f) - f\mu^{1/2}\left(\frac{\Delta \rho}{\rho} - \omega\right)\right] +\notag \\
&+\frac{\epsilon^2}{4}\left[\mathcal{Q}\left(f\mu^{1/2}\right) -
\frac{1}{4}\Delta\left(f\mu^{1/2}\left(\frac{\Delta \rho}{\rho} - \omega\right)\right) +f\mu^{1/2}\left(\frac{1}{4}\left(\frac{\Delta \rho}{\rho} - \omega\right)^2 -
\frac{\mathcal{Q}(\rho)}{\rho}\right)\right] \notag\\
&~+ O(\epsilon^3). \label{Km1}
\end{align}
To make the expression \eqref{Km1} more compact, we denote the operator in the last square brackets by
\begin{align}
\label{R}
\mathcal{R}_{\mu,\rho}f:=\mathcal{Q}\left(f\mu^{1/2}\right) -
\frac{1}{4}\Delta\left(f\mu^{1/2}\left(\frac{\Delta \rho}{\rho} - \omega\right)\right) +f\mu^{1/2}\left(\frac{1}{4}\left(\frac{\Delta \rho}{\rho} - \omega\right)^2 -
\frac{\mathcal{Q}(\rho)}{\rho}\right).
\end{align}
Then $\mathcal{K}_{\epsilon,\mu}f$ in \eqref{Km1} can be written as
\begin{align}
\label{Km2}
\mathcal{K}_{\epsilon,\mu}f  = f\mu^{1/2} + \frac{\epsilon}{4}\left[\Delta(\mu^{1/2}f) - f\mu^{1/2}\left(\frac{\Delta \rho}{\rho}-\omega\right)\right]
+\frac{\epsilon^2}{4}\mathcal{R}_{\mu,\rho}f + O(\epsilon^3).
\end{align}
Now we are ready to calculate the Markov operator $\mathcal{P}_{\epsilon,\mu}f$:
\begin{align}
 \label{P1}
 \mathcal{P}_{\epsilon,\mu}f  = \frac{\mathcal{K}_{\epsilon,\mu}f }{\mathcal{K}_{\epsilon,\mu}1}=
\frac{ f\mu^{1/2} + \frac{\epsilon}{4}\left[\Delta(\mu^{1/2}f) - f\mu^{1/2}\left(\frac{\Delta \rho}{\rho}-\omega\right)\right]
+\frac{\epsilon^2}{4}\mathcal{R}_{\mu,\rho}f + O(\epsilon^3)}
{
\mu^{1/2} + \frac{\epsilon}{4}\left[\Delta(\mu^{1/2}) - \mu^{1/2}\left(\frac{\Delta \rho}{\rho}-\omega\right)\right]
+\frac{\epsilon^2}{4}\mathcal{R}_{\mu,\rho}1 + O(\epsilon^3)
}.
\end{align}
Dividing the numerator and the denominator by $\mu^{1/2}$ and applying Taylor expansion we get:
\begin{align}
&~ \mathcal{P}_{\epsilon,\mu}f =
\left( f + \frac{\epsilon}{4}\left[\frac{\Delta(\mu^{1/2}f)}{\mu^{1/2}} - f\left(\frac{\Delta \rho}{\rho} - \omega\right)\right]
+\frac{\epsilon^2}{4}\frac{\mathcal{R}_{\mu,\rho}f}{\mu^{1/2}} + O(\epsilon^3)\right)\notag\\
&\times\left(
1 - \frac{\epsilon}{4}\left[\frac{\Delta(\mu^{1/2}) }{\mu^{1/2}}- \left(\frac{\Delta \rho}{\rho} - \omega\right)\right]
- \frac{\epsilon^2}{4}\frac{\mathcal{R}_{\mu,\rho}1}{\mu^{1/2}} +  \frac{\epsilon^2}{16}\left[\frac{\Delta(\mu^{1/2}) }{\mu^{1/2}}- \left(\frac{\Delta \rho}{\rho} - \omega\right)\right]^2 + O(\epsilon^3)\right) \notag \\
&=f +  \frac{\epsilon}{4}\left[\frac{\Delta(\mu^{1/2}f)}{\mu^{1/2}} - f\frac{\Delta(\mu^{1/2}) }{\mu^{1/2}}\right] +
\frac{\epsilon^2}{4}\left[\frac{\mathcal{R}_{\mu,\rho}f}{\mu^{1/2}} -f\frac{\mathcal{R}_{\mu,\rho}1}{\mu^{1/2}}\right]\notag\\
&+  \frac{\epsilon^2}{16}\left(f \left[\frac{\Delta(\mu^{1/2}) }{\mu^{1/2}}- \left(\frac{\Delta \rho}{\rho} - \omega\right)\right]^2  \right.\notag\\
&-\left.
\left[\frac{\Delta(\mu^{1/2}f)}{\mu^{1/2}} - f\left(\frac{\Delta \rho}{\rho} - \omega\right)\right]
\left[\frac{\Delta(\mu^{1/2}) }{\mu^{1/2}}- \left(\frac{\Delta \rho}{\rho} - \omega\right)\right]\right) + O(\epsilon^3).\label{P2}
\end{align}
The last term in \eqref{P2} can be simplified by noting that
\begin{align}
&f\left[\frac{\Delta(\mu^{1/2}) }{\mu^{1/2}}- \left(\frac{\Delta \rho}{\rho} - \omega\right)\right]^2 -
\left[\frac{\Delta(\mu^{1/2}f)}{\mu^{1/2}} - f\left(\frac{\Delta \rho}{\rho} - \omega\right)\right]
\left[\frac{\Delta(\mu^{1/2}) }{\mu^{1/2}}- \left(\frac{\Delta \rho}{\rho} - \omega\right)\right] \notag \\
&= \left[\frac{\Delta(\mu^{1/2}) }{\mu^{1/2}}- \left(\frac{\Delta \rho}{\rho} - \omega\right)\right]\left[f\frac{\Delta(\mu^{1/2}) }{\mu^{1/2}}- \frac{\Delta \mu^{1/2}f}{\mu^{1/2}}\right]. \label{simplify1}
\end{align}

Taking \eqref{simplify1} into account, we write out the generator$\mathcal{L}_{\epsilon,\mu}f$:
\begin{align}
\mathcal{L}_{\epsilon,\mu}f  = \frac{ \mathcal{P}_{\epsilon,\mu}f  - f}{\epsilon} &=
\frac{1}{4}\left[\frac{\Delta(\mu^{1/2}f)}{\mu^{1/2}} - f\frac{\Delta(\mu^{1/2}) }{\mu^{1/2}}\right]
+\frac{\epsilon}{4}\left[\frac{\mathcal{R}_{\mu,\rho}f}{\mu^{1/2}} -f\frac{\mathcal{R}_{\mu,\rho}1}{\mu^{1/2}}\right]\label{L1}\\
&+  \frac{\epsilon}{16}\left[\frac{\Delta(\mu^{1/2}) }{\mu^{1/2}}- \left(\frac{\Delta \rho}{\rho} - \omega\right)\right]\left[f\frac{\Delta(\mu^{1/2}) }{\mu^{1/2}}- \frac{\Delta \mu^{1/2}f}{\mu^{1/2}}\right] + O(\epsilon^2).\notag
\end{align}

It is straightforward to check that the limit $\epsilon \to 0$ is 
\begin{align}
    \lim_{\epsilon\rightarrow 0} \mathcal{L}_{\epsilon, \mu} =  \frac{1}{4}\left[\frac{\Delta(\mu^{1/2}f)}{\mu^{1/2}} - f\frac{\Delta(\mu^{1/2}) }{\mu^{1/2}}\right]
    = \frac{1}{4}\left[ \Delta f+\nabla \log \mu \cdot \nabla f \right]\equiv\frac{\beta}{4}\mathcal{L}.\label{finalop}
\end{align}
The term $\epsilon\mathcal{B}_3[f,\mu,\rho](x)$ in \eqref{eq: full bias error} is the first term in the second line of \eqref{L1}. The terms $\epsilon\mathcal{B}_1[f,\mu](x)$ and $\epsilon\mathcal{B}_2[f,\mu,\rho](x)$ in \eqref{eq: full bias error} are recovered by computing the term with the operators $\mathcal{R}_{\mu,\rho}$ in \eqref{L1}.
\end{proof}

\subsection{Variance error: proofs}
\label{subsec: variance error proofs}
\subsubsection{Bounds for the error in the kernel density estimate}
\label{sec:proofThm9}
To prove Theorem \ref{thm:M1} establishing error bound for the discrete kernel density estimate (KDE), we first derive concentration inequalities for the discrepancy between $\rho_{\epsilon}(x)$ and $\rho^{(n)}_{\epsilon}(x)$ in Lemmas \ref{lemma:M1} and \ref{lemma:M2} below. There are two cases requiring somewhat different reasoning. In Lemma \ref{lemma:M1}, the point $x$ is chosen independently of the point cloud $\mathcal{X}(n)$, while in Lemma \ref{lemma:M2}, the point $x$ belongs to $\mathcal{X}(n)$. 

\begin{lemma}
\label{lemma:M1}
Let  $\mathcal{X}(n): = \{x_j\}_{j=1}^n \subset\mathcal{M}$ be a point cloud sampled according to the density $\rho(x)$ and let $x\in\mathcal{M}$ be an arbitrary point selected independently of $\mathcal{X}(n)$.  Let $\rho_{\epsilon}(x)$ and $\rho^{(n)}_{\epsilon}(x)$ be the KDEs defined by \eqref{eq:rhoeps} and \eqref{eq:rhoepsn} respectively. Then  for any $\epsilon$ sufficiently small and any constant $\alpha = o(\epsilon^{-1})$, the following relationship holds:
\begin{equation}
\label{eq:lemmaM1}
\mathbb{P}\Bigl(\left|\rho^{(n)}_{\epsilon}(x) - \rho_{\epsilon}(x)\right| \ge (\pi\epsilon)^{d/2}\rho(x) \epsilon^2\alpha\Bigl)\thinspace \le \thinspace2\exp\left[-\frac{n(2\pi)^{d/2}\rho(x)\epsilon^{4+d/2}\alpha^2 }{2 + O(\epsilon^{\min\{1,d/2\}})}\right].
\end{equation}
\end{lemma}
\begin{proof}
The difference in the left-hand side of \eqref{eq:lemmaM1} can be written as
\begin{equation}
    \label{eq:rhoediff}
\rho^{(n)}_{\epsilon}(x) - \rho_{\epsilon}(x) = 
\frac{1}{n}\sum_{j=1}^n\left(k_{\epsilon}(x,x_j) - \rho_{\epsilon}(x)\right).
\end{equation}
The random variables
\begin{equation}
    \label{eq:X_j}
    X_j:= \frac{1}{n}\left(k_{\epsilon}(x,x_j) - \rho_{\epsilon}(x)\right),\quad 1\le j\le n,
\end{equation}
have zero mean and are bounded by $|X_j|\le n^{-1}$, $1\le j\le n$, as $\rho_{\epsilon}(x)\in(0,1)$ and $\rho^{(n)}_{\epsilon}(x)\in(0,1]$.
The variance of $X_j$ can be calculated as 
\begin{equation}
    \label{eq:VarXj}
    {\sf Var}(X_j) = \frac{1}{n^2}\left(\mathbb{E}\left[k_{\epsilon}(x,x_j)^2\right] - \mathbb{E}^2\left[k_{\epsilon}(x,x_j)\right]\right).
\end{equation}
The expectations in \eqref{eq:VarXj} are calculated with the aid of Lemma \ref{lem: expansion} and the observation that $k_{\epsilon}(\cdot,\cdot)^2 = k_{\epsilon/2}(\cdot,\cdot)$:
\begin{align}
        \mathbb{E}[k_{\epsilon}(x,\cdot)^2] &= \int_{\mathcal{M}}(k_{\epsilon}(x,y))^2 \rho(y) \,d\text{vol}(y)\notag\\
        &= \int_{\mathcal{M}}k_{\epsilon/2}(x,y) \rho(y) \,d\text{vol}(y) = \left(\frac{\pi \epsilon}{2}\right)^{d/2}\left(\rho(x)  + O(\epsilon)\right); \label{eq:Eke2}
\end{align}
\begin{align}
         \mathbb{E}^2[k_{\epsilon}(x,\cdot)]= \left(\int_{\mathcal{M}}k_{\epsilon}(x,y) \rho(y) \,d\text{vol}(y)\right)^2 = \left(\pi \epsilon\right)^{d}\left(\rho^2(x)  + O(\epsilon)\right). \label{E2ke} 
\end{align}
Therefore, 
\begin{equation}
    \label{eq:VarX_jfinal} 
    {\sf Var}(X_j) = \frac{1}{n^2} \left[\left(\frac{\pi \epsilon}{2}\right)^{d/2}\left(\rho(x)  + O(\epsilon)\right) +  O(\epsilon^d)\right] = \frac{1}{n^2} \left[\left(\frac{\pi \epsilon}{2}\right)^{d/2}\left(\rho(x)  + O(\epsilon^{\min\{1,d/2\})}\right)\right].
\end{equation}
The application of Bernstein's inequality \eqref{eq:Bernstein0} to the sum of  $X_j$s  gives
\begin{equation}
    \label{eq:Bernstein1}
    \mathbb{P}\left(\frac{1}{n}\sum_{j=1}^n\left[k_{\epsilon}(x,x_j) - \rho_{\epsilon}(x)\right] \ge t\right)
    \le \exp\left[-\frac{t^2}{\frac{1}{n}\left(\frac{\pi\epsilon}{2}\right)^{d/2}\left(\rho(x) + O(\epsilon^{\min\{1,d/2\}})\right) + \frac{2}{3}\frac{1}{n}t}\right].
\end{equation}
We want the relative error 
\begin{equation}
    \label{eq:rel_errM1}
    \frac{\left|\rho^{(n)}_{\epsilon}(x) -  \rho_{\epsilon}(x)\right|}{\rho_{\epsilon}(x)}
\end{equation}
to be small compared to $\epsilon$. Motivated by the formula \eqref{eq:rhoeps} for $\rho_{\epsilon}(x)$,  we choose 
\begin{equation}
    \label{eq:tchoice} t = (\pi\epsilon)^{d/2}\rho(x) \epsilon^2\alpha,
\end{equation}
where $\alpha$ is a constant satisfying $\alpha =o(\epsilon^{-1})$. In this case, the term $\tfrac{2}{3n}t$ in the denominator of \eqref{eq:Bernstein1} is $\tfrac{2}{3n}(\pi\epsilon)^{d/2}o(\epsilon)$.
Plugging \eqref{eq:rhoediff} and \eqref{eq:tchoice} into \eqref{eq:Bernstein1} we obtain
\begin{equation}
    \label{eq:Bernstein2}
    \mathbb{P}\Bigl(\rho^{(n)}_{\epsilon}(x) - \rho_{\epsilon}(x)  \ge(\pi\epsilon)^{d/2} \rho(x)\epsilon^2\alpha \Bigl)
    \le \exp\left[-\frac{n(2\pi)^{d/2}\rho(x)\epsilon^{4+d/2}\alpha^2 }{2 + O(\epsilon^{\min\{1,d/2\}})}\right].
\end{equation}
Redefining $X_j$ as $-X_j$, $1\le j\le n$ and repeating the argument above, we obtain the same upper bound for the probability 
$\mathbb{P}\left(\rho^{(n)}_{\epsilon}(x) - \rho_{\epsilon}(x)  \le - (\pi\epsilon)^{d/2}\rho(x) \epsilon^2\alpha \right)$.
Then the desired result \eqref{eq:lemmaM1} readily follows.
\end{proof}

Unfortunately, the estimate obtained in Lemma \ref{lemma:M1} is not suitable for the settings of any standard diffusion map algorithm including the TMD map as all quantities of interest including $\rho_{\epsilon}$ are evaluated at the points of the given point cloud. Therefore, $x$ is one of the points of the point cloud $\mathcal{X}(n)$, say, $x\equiv x_i$, which means that $x$ and $\mathcal{X}(n)$ are not sampled independently. 
The discrepancy between $\rho_{\epsilon}^{(n)}(x_i)$ and $\rho_{\epsilon}(x_i)$ is estimated in the following lemma.
\begin{lemma}
\label{lemma:M2}
Let  $\mathcal{X}(n): = \{x_j\}_{j=1}^n \subset\mathcal{M}$ be a point cloud sampled according to the density $\rho(x)$ and let $x_i\in\mathcal{X}(n)$ be any selected point of this point cloud.  Let $\rho_{\epsilon}(x_i)$ and $\rho^{(n)}_{\epsilon}(x_i)$ be the KDEs defined by \eqref{eq:rhoeps} and \eqref{eq:rhoepsn} respectively with $x\equiv x_i$. Then for $\epsilon$ small enough, $\alpha=o(\epsilon^{-1})$ and $n$ large enough so that $n(\pi\epsilon)^{d/2}\rho(x_i)\epsilon^2\alpha\gg 1$,  the following relationship holds:
\begin{equation}
\label{eq:lemmaM2}
\mathbb{P}\Bigl(\left|\rho^{(n)}_{\epsilon}(x_i) - \rho_{\epsilon}(x_i)\right| \ge (\pi\epsilon)^{d/2}\rho(x_i)\epsilon^{2}\alpha ~|~ x_i \Bigl)\thinspace \le \thinspace 2\exp\left[-\frac{n(2\pi)^{d/2}\rho(x_i)\epsilon^{4+d/2}\alpha^2 }{2 + O(\epsilon^{\min\{1,d/2\}})}\right].
\end{equation}
\end{lemma}

\begin{proof} We will use subscripts to denote the conditional probability:
\begin{align}
    \mathbb{P}_{x_i}\left(\cdot\right) := \mathbb{P}\left(\cdot \middle| x_i \right)
\end{align}
Under the distribution conditioned on $x_i$, for $1\le j\le n$, we define $X_j$ as in \eqref{eq:X_j}. For any $j\neq i$, $X_j$ is a random variable, while for $j=i$, $X_i$ is a deterministic number. Moreover, for $\epsilon$ small enough, 
\begin{equation}
    \label{eq:X_i}
    X_i = \frac{1}{n}\left(k_{\epsilon}(x_i,x_i)-\rho_{\epsilon}(x_i)\right) = \frac{1}{n} \left(1-(\pi\epsilon)^{d/2}\left[\rho(x_i)+O(\epsilon)\right]\right)\lessapprox \frac{1}{n}.
\end{equation}
In particular, $X_i > 0$ for $\epsilon$ small enough.
Therefore, we get
\begin{align}
    \mathbb{P}_{x_i}\left(\rho^{(n)}_{\epsilon}(x_i) - \rho_{\epsilon}(x_i) \ge t\right) &= 
    \mathbb{P}_{x_i}\left( \sum_{j\neq i}X_j  \ge t - X_i\right)
    <\mathbb{P}_{x_i}\left( \sum_{j\neq i}X_j  \ge t\right) \notag\\
    &\le \exp\left[-\frac{t^2}{2(n-1){\sf Var}(X_j) +\tfrac{2}{3n}t}\right].    \label{eq:Bernstein3}
\end{align}
We choose $t$ of the form \eqref{eq:tchoice} so that the relative error \eqref{eq:rel_errM1} is small in comparison with $\epsilon$. Plugging \eqref{eq:tchoice} and \eqref{eq:VarX_jfinal} into \eqref{eq:Bernstein3} we obtain
\begin{align}
\mathbb{P}_{x_i}\left(\rho^{(n)}_{\epsilon}(x_i) - \rho_{\epsilon}(x_i) \ge (\pi\epsilon)^{d/2}\rho(x_i)\epsilon^{2}\alpha\right) \le
\exp\left[-\frac{n(2\pi)^{d/2}\rho(x)\epsilon^{4+d/2}\alpha^2}{2\frac{n-1}{n}\left[1 + O(\epsilon^{\min\{1,d/2\}})\right]}\right]. \label{eq:Bernstein31}
\end{align}
The denominator in the fraction in the exponential in \eqref{eq:Bernstein31} can be increased by replacing the factor of $\tfrac{n-1}{n}$ with 1. Then the argument of the exponential becomes slightly less negative and hence the exponential increases amplifying the inequality. Therefore, we get
\begin{align}
\mathbb{P}_{x_i}\left(\rho^{(n)}_{\epsilon}(x_i) - \rho_{\epsilon}(x_i) \ge (\pi\epsilon)^{d/2}\rho(x_i)\epsilon^{2}\alpha\right) \le
\exp\left[-\frac{n(2\pi)^{d/2}\rho(x_i)\epsilon^{4+d/2}\alpha^2}{2 + O(\epsilon^{\min\{1,d/2\}})}\right]. \label{eq:Bernstein32}
\end{align}

Next, as in the proof of Lemma \ref{lemma:M1}, we redefine $X_j$ as $-X_j$, $1\le j\le n$. Then $X_i \in (-n^{-1},0)$ and $X_i\gtrapprox -n^{-1}$. Therefore,
\begin{align}
    \mathbb{P}_{x_i}\left(\rho^{(n)}_{\epsilon}(x_i) - \rho_{\epsilon}(x_i) \le -t\right) &=  \mathbb{P}_{x_i}\left( \sum_{j\neq i}X_j  \ge t - X_i\right)<  \mathbb{P}_{x_i}\left( \sum_{j\neq i}X_j  \ge t + \frac{1}{n}\right)\notag\\
    &\le  \exp\left[-\frac{(t+n^{-1})^2}{2(n-1){\sf Var}(X_j) +\tfrac{2}{3n}(t+n^{-1})}\right].    \label{eq:Bernstein4}
\end{align}
We observe that the assumption $n(\pi\epsilon)^{d/2}\rho(x_i)\epsilon^2\alpha\gg 1$ implies that for $t$ of the form \eqref{eq:tchoice} $t\gg n^{-1}$. Therefore, using similar reasoning as in the proof of Lemma \ref{lemma:M1} we calculate:
\begin{align}
&\mathbb{P}_{x_i}\left(\rho^{(n)}_{\epsilon}(x_i) - \rho_{\epsilon}(x_i)< - (\pi\epsilon)^{d/2}\rho(x_i)\epsilon^{2}\alpha\right) 
\le\notag \\
&\exp\left[-\frac{\left[(\pi\epsilon)^{d/2}\rho(x_i)\epsilon^{2}\alpha+\tfrac{1}{n}\right]^2}{\frac{2(n-1)}{n^2} \left[\left(\frac{\pi \epsilon}{2}\right)^{d/2}\left(\rho(x_i)  + O(\epsilon^{\min\{1,d/2\}})\right)\right] +\tfrac{2}{3n}\left[(\pi\epsilon)^{d/2}\rho(x_i)\epsilon^{2}\alpha+\tfrac{1}{n}\right]}\right]. \label{eq:Bernstein5}
\end{align}
The right-hand side of inequality \eqref{eq:Bernstein5} can be simplified by making use of the following observations. First, the last term in square brackets in the denominator is $(\pi\epsilon)^{d/2}o(\epsilon)$ and hence it can be incorporated into $ O(\epsilon^{\min\{1,d/2\})}$. Second, $\exp(-(\cdot))$ is a decreasing function of $(\cdot)$. Hence, if we remove $\tfrac{1}{n}$ from the enumerator we increase the exponent and amplify inequality \eqref{eq:Bernstein5}. Finally, we can amplify the inequality by replacing removing the factor of $\tfrac{n-1}{n}$. As a result, we obtain:
\begin{align}
\mathbb{P}_{x_i}\left(\rho^{(n)}_{\epsilon}(x_i) - \rho_{\epsilon}(x_i)< - (\pi\epsilon)^{d/2}\rho(x_i)\epsilon^{2}\alpha\right) 
\le
\exp\left[-\frac{n(2\pi)^{d/2}\rho(x_i)\epsilon^{4+d/2}\alpha^2}{2  + O(\epsilon^{\min\{1,d/2\}})}\right]. \label{eq:Bernstein6}
\end{align}

The right-hand sides of inequalities \eqref{eq:Bernstein32} and \eqref{eq:Bernstein6} are identical. Together, they imply the final result \eqref{eq:lemmaM2}.
\end{proof}

Now we are ready to prove Theorem \ref{thm:M1}.
\begin{proof}
    Let $\mathcal{X}(n)$ be a point cloud sampled according to the density $\rho$. By Lemmas \ref{lemma:M1} and \ref{lemma:M2}, for any fixed point $x\in \mathcal{M}$ or any point $x\in\mathcal{X}(n)$, for any $\epsilon$ small enough, any $\alpha = o(\epsilon^{-1})$, and any $n$ large enough so that $\pi^{d/2}n\epsilon^{2+d/2}\alpha\gg 1$, inequality \eqref{eq:lemmaM1} holds. This inequality can be amplified by replacing $\rho(x)$ with $\rho^{-}$ since this makes the argument of the exponential less negative and hence increases the exponential. As a result, the following uniform bound for all $x\in\mathcal{M}$ holds:
\begin{equation}
\label{eq:Bernstein7}
    \mathbb{P}_x\left(\left|\rho^{(n)}_{\epsilon}(x) - \rho_{\epsilon}(x)\right| \geq (\pi\epsilon)^{d/2}\rho(x)\epsilon^{2}\alpha\right) 
\le
2\exp\left[-\frac{n(2\pi)^{d/2}\rho(x)\epsilon^{4+d/2}\alpha^2}{2  + O(\epsilon^{\min\{1,d/2\}})}\right]. 
\end{equation}
Here 
\begin{align}
    \mathbb{P}_x := \mathbb{P}(\cdot \mid x)
\end{align}
whether $x \in \mathcal{M}$ or $x := x_i \in \mathcal{X}(n)$.  If $x \in \mathcal{M}, x \notin \mathcal{X}(n)$ then $\mathbb{P}_{x}(\cdot) = \mathbb{P}(\cdot)$. Now we will find $\alpha$ such that the right-hand side of inequality \eqref{eq:Bernstein7} is less or equal to $n^{-4}$. We set
    \begin{equation}
    \label{eq:find_alpha1}
    2\exp\left[-\frac{n(2\pi)^{d/2}\rho(x)\epsilon^{4+d/2}\alpha^2}{2  + O(\epsilon^{\min\{1,d/2\}})}\right] \le n^{-4}.
    \end{equation}
   Taking logarithms of both sides of \eqref{eq:find_alpha1} and rearranging terms gives the following bound for $\alpha$:
   \begin{equation}
       \label{eq:find_alpha2}
       \alpha \ge \sqrt{\frac{\left(4\log n+\log 2\right)\left(2+ O(\epsilon^{\min\{1,d/2\}})\right)}{n(2\pi)^{d/2}\rho(x)\epsilon^{4+d/2}}}.
   \end{equation}
   Bound \eqref{eq:find_alpha2} shows that for $\epsilon$ small enough and $n$ large enough,
   \begin{equation}
       \label{eq:alpha_choice}
       \alpha = \frac{5}{(2\pi)^{d/4}}\sqrt{\frac{\log n}{n\rho(x)\epsilon^{4+d/2}}}
   \end{equation}
   satisfies inequality \eqref{eq:find_alpha2}. 

Now we observe that 
$$
(\pi\epsilon)^{d/2}\rho(x)\epsilon^2\alpha = (\pi\epsilon)^{d/2}\rho^{1/2}(x)\epsilon^2\frac{5}{(2\pi)^{d/4}}\sqrt{\frac{\log n}{n\epsilon^{4+d/2}}}.
$$
We redefine $\alpha$ as
   \begin{equation}
       \label{eq:alpha_choice1}
       \alpha = \frac{1}{(2\pi)^{d/4}}\sqrt{\frac{\log n}{n\epsilon^{4+d/2}}}
   \end{equation}
and obtain the desired inequality \eqref{eq:thmM1} for each point $x_i$. To move from conditional probabilities to overall probabilities, we let $E(x_i)$ be the event that \eqref{eq:thmM1} holds. Moreover, let $\mathbb{I}_{E(x_i)}$ be the indicator function of this event. Then using the tower of expectations, 
\begin{align}
    \mathbb{P}(E(x_i)) = \mathbb{E}[\mathbb{I}_{E(x_i)}] = \mathbb{E}_{x_i}[\mathbb{E}[\mathbb{I}_{E(x_i)} \mid x_i]] \leq  \mathbb{E}_{x_i}[n^{-4}] = n^{-4}. 
\end{align}
Let 
\begin{align}
    E = \bigcap_{i=1}^{n}E(x_i)^c. 
\end{align}
Then 
\begin{align}
    \mathbb{P}(E) &= \mathbb{P}\left(\bigcap_{i=1}^{n}E(x_i)^c\right) 
    = 1 - \mathbb{P}\left(\bigcup_{i=1}^{n}E(x_i)\right) \notag \\
    &\geq 1 - \sum_{i=1}^{n}\mathbb{P}(E(x_i)) = 1 - n\cdot n^{-4} = 1 - n^{-3}.
\end{align}
In this event $E$, \eqref{eq:thmM1} holds for every $x_i \in \mathcal{X}(n)$ or $x \in \mathcal{M}$ as desired. It remains to check that $\alpha$ given by \eqref{eq:alpha_choice1} is $o(\epsilon^{-1})$, i.e.,
   \begin{equation}
   \label{eq:alpha_check} 
   \lim_{\substack{\epsilon\rightarrow0\\n\rightarrow\infty}}\alpha\epsilon 
   =0.
   \end{equation}
   This is so due to the condition \eqref{eq:thmM1_epsn}. Indeed, by   \eqref{eq:thmM1_epsn},
   \begin{equation}
       \label{eq:alpha_check1}
       \alpha^2\epsilon^2\propto \frac{\log n}{n\epsilon^{2+d/2}} \rightarrow 0.
   \end{equation}
   Hence $\alpha\epsilon\rightarrow 0$ as desired. 

The claim of Theorem \ref{thm:M1} readily follows.
\end{proof}

\subsubsection{Bound for the discrepancy $L_{\epsilon,\mu}^{(n)}f$ and $L_{\epsilon,\mu}f$}
\label{sec:proofThm10}
Theorem \ref{thm:M2} bounds the discrepancy between the discrete TMDmap generators with discrete and continuous KDEs. Its proof is given below.
\begin{proof}
Equation \eqref{eq:thmM2_L} implies that
\begin{equation}
    \label{eq:LPn}
    \left|\left[L^{(n)}_{\epsilon,\mu}f\right](x) -\left[L_{\epsilon,\mu}f\right](x)  \right| = \frac{1}{\epsilon} \left|\left[P^{(n)}_{\epsilon,\mu}f\right](x)- \left[P_{\epsilon,\mu}f\right](x)  \right|.
\end{equation}
Therefore, we proceed to calculate the difference between the discrete Markov operators:
\begin{align}
    &\left[P_{\epsilon,\mu}f\right](x)- \left[P^{(n)}_{\epsilon,\mu}f\right](x)  =   \frac{\sum_{j=1}^n F_j}{\sum_{j=1}^n G_j} - \frac{\sum_{j=1}^n F^{(n)}_j}{\sum_{j=1}^n G^{(n)}_j} =\notag \\
    &   \frac{\left(\sum_{j=1}^n F_j\right)\left(\sum_{l=1}^n G^{(n)}_l \right) - \left(\sum_{j=1}^n F^{(n)}_j\right)\left( \sum_{l=1}^n G_l\right)}{\left(\sum_{j=1}^n G_j\right)\left(\sum_{l=1}^n G^{(n)}_l \right)}= \notag\\
    & \frac{\sum_{j,l}\left( F_jG^{(n)}_l - F_jG_l + F_j G_l -F^{(n)}_jG_l\right)}{\sum_{j,l}G_jG^{(n)}_l} = \notag \\
    &\frac{\sum_{j,l}\left( F_j\left(G^{(n)}_l - G_l\right) + G_l\left(F_j  -F^{(n)}_j\right)\right)}{\sum_{j,l}G_jG^{(n)}_l}. \label{eq:Pndiff1}
\end{align}
The differences $G^{(n)}_l - G_l$ and $F_j  -F^{(n)}_j$ in \eqref{eq:Pndiff1}  can be bounded using Theorem \ref{thm:M1}.
We start with
\begin{align}
   \left|G^{(n)}_l - G_l\right| = \left|\frac{k_{\epsilon}(x,x_l)\mu^{1/2}(x_l)}{\rho_{\epsilon}^{(n)}(x_l)} - 
    \frac{k_{\epsilon}(x,x_l)\mu^{1/2}(x_l)}{\rho_{\epsilon}(x_l)}\right|=
    k_{\epsilon}(x,x_l)\mu^{1/2}(x_l)\frac{\left|\rho_{\epsilon}(x_l) - \rho_{\epsilon}^{(n)}(x_l) \right|}{\rho_{\epsilon}(x_l)\rho_{\epsilon}^{(n)}(x_l)}. \label{eq:Gndiff}
\end{align}
By Theorem \ref{thm:M1}, the probability of the event $E$ that the absolute value of the difference in the numerator in \eqref{eq:Gndiff} for any $x_l\in\mathcal{X}(n)$ is less than $5(\pi\epsilon)^{d/2}\rho^{1/2}(x_l)\epsilon^2\alpha$  is greater or equal to $1-n^{-3}$. In particular, since $\alpha$ given by \eqref{eq:thmM1_alpha} is $o(\epsilon^{-1})$, in the case of event $E$, this difference is $(\pi\epsilon)^{d/2}o(\epsilon)$. Recalling that
$$
\rho_{\epsilon}(x_l) = (\pi\epsilon)^{d/2}\left(\rho(x_l) + O(\epsilon)\right),
$$
the difference \eqref{eq:Gndiff} can be further bounded as follows:
\begin{align}
   \left|G^{(n)}_l - G_l\right| < k_{\epsilon}(x,x_l)\mu^{1/2}(x_l) \frac{5(\pi\epsilon)^{d/2}\rho^{1/2}(x_l)\epsilon^2\alpha}{(\pi\epsilon)^d\rho^2(x_l)(1+O(\epsilon))} = \frac{5k_{\epsilon}(x,x_l)\mu^{1/2}(x_l)\epsilon^2\alpha}{(\pi\epsilon)^{d/2}\rho^{3/2}(x_l)}(1+O(\epsilon)).
   \label{eq:Gndiff1}
\end{align}
Similarly, 
the difference $F_j - F^{(n)}_j$ is bounded by 
\begin{align}
   \left|F_j - F^{(n)}_j\right| <  \frac{5k_{\epsilon}(x,x_l)\mu^{1/2}(x_j){|f(x_j)|}\epsilon^2\alpha}{(\pi\epsilon)^{d/2}\rho^{3/2}(x_j)}(1+O(\epsilon)).
   \label{eq:Fndiff1}
\end{align}
Therefore we have:
\begin{align}
  & \left| \left[P_{\epsilon,\mu}f\right](x)- \left[P^{(n)}_{\epsilon,\mu}f\right](x) \right| \le
   \frac{\sum_{j,l}\left( |F_j|\left|G^{(n)}_l - G_l\right| + G_l\left|F_j  -F^{(n)}_j\right|\right)}{\sum_{j,l}G_jG^{(n)}_l}<\notag \\
   &\frac{\sum_{j,l}5\left( \frac{k_{\epsilon}(x,x_j)\mu^{1/2}(x_j){|f(x_j)|}}{(\pi\epsilon)^{d/2}\rho(x_j)} \frac{k_{\epsilon}(x,x_l)\mu^{1/2}(x_l)\epsilon^2\alpha}{(\pi\epsilon)^{d/2}\rho^{3/2}(x_l)} + \frac{k_{\epsilon}(x,x_l)\mu^{1/2}(x_l)}{(\pi\epsilon)^{d/2}\rho(x_l)}\frac{k_{\epsilon}(x,x_j)\mu^{1/2}(x_j){|f(x_j)|}\epsilon^2\alpha}{(\pi\epsilon)^{d/2}\rho^{3/2}(x_j)}\right)}{\sum_{j,l}\frac{k_{\epsilon}(x,x_l)k_{\epsilon}(x,x_j)\mu^{1/2}(x_l)\mu^{1/2}(x_j)}{(\pi\epsilon)^d\rho(x_l)\rho(x_j)}}\left(1+O(\epsilon)\right) =\notag \\
   & \frac{\sum_{j,l}5\frac{k_{\epsilon}(x,x_j)k_{\epsilon}(x,x_l)\mu^{1/2}(x_j)\mu^{1/2}(x_l){|f(x_j)|}}{\rho(x_j)\rho(x_l)}\left( \rho^{-1/2}(x_l) +\rho^{-1/2}(x_j)\right)}
   {\sum_{j,l}\frac{k_{\epsilon}(x,x_l)k_{\epsilon}(x,x_j)\mu^{1/2}(x_l)\mu^{1/2}(x_j)}{\rho(x_l)\rho(x_j)}}\epsilon^2\alpha\left(1+O(\epsilon)\right)
   \label{eq:Pndiff2}
\end{align}
Equations \eqref{eq:Pndiff2} and \eqref{eq:LPn} imply that
$$
\left|\left[L^{(n)}_{\epsilon,\mu}f\right](x) -\left[L_{\epsilon,\mu}f\right](x)  \right| < \tilde{C}\epsilon\alpha(1+O(\epsilon)),
$$
where the constant $\tilde{C}$ is the prefactor in \eqref{eq:Pndiff2}.
Now it remains to estimate $\tilde{C}$. Approximating the sums in the numerator and the denominator of $\tilde{C}$ we obtain
\begin{align}
     \tilde{C} =& 5\left(\left[\int_{\mathcal{M}}k_{\epsilon}(x,y)\mu^{1/2}(y){|f(y)|}d{\rm vol}(y)\right]\left[ \int_{\mathcal{M}}k_{\epsilon}(x,y)\frac{\mu^{1/2}(y)}{\rho^{1/2}(y)}d{\rm vol}(y)\right] \right. +\notag \\ 
    &\left.\left[\int_{\mathcal{M}}k_{\epsilon}(x,y)\frac{\mu^{1/2}(y)}{\rho^{1/2}(y)}{|f(y)|}d{\rm vol}(y)\right]\left[ \int_{\mathcal{M}}k_{\epsilon}(x,y)\mu^{1/2}(y)d{\rm vol}(y)\right]
    \right)/ \notag \\
    &\left(
    \left[ \int_{\mathcal{M}}k_{\epsilon}(x,y)\mu^{1/2}(y)d{\rm vol}(y)\right]^2
    \right) + O(n^{-1/2}).    \label{eq:tildeC1}
\end{align}
Further, the integrals in \eqref{eq:tildeC1} can be approximated using Expansion Lemma \ref{lem: expansion} resulting at
\begin{align}
    \tilde{C}& = 5\frac{\left(\frac{\mu(x){|f(x)|}}{\rho^{1/2}(x)} + O(\epsilon)\right) + 
    \left(\frac{\mu(x){|f(x)|}}{\rho^{1/2}(x)} + O(\epsilon)\right)
    }{\left(\mu^{1/2}(x) + O(\epsilon)\right)^2} + (n^{-1/2})\notag \\
    &= 10{|f(x)|}\rho^{-1/2}(x) + O(\epsilon) + O(n^{-1/2}). \label{eq:Ctilde_proofThm10}
\end{align}
{ At the points where $|f(x)|$ is nonsmooth, bound \eqref{eq:Ctilde_proofThm10} is obtained using any mollification of $|f(x)|$ that is greater or equal than $|f(x)|$. The difference between $|f(x)|$ and its mollification is incorporated in $O(\epsilon)$ (see Appendix C.4 in \cite{evans10}).}
Denoting $10|f(x)|\rho^{-1/2}(x)$ by $C$ and noting that $\epsilon\rightarrow 0$ and $n\rightarrow \infty$ we complete the proof.
\end{proof}

\subsubsection{Bound for the discrepancy between $L_{\epsilon,\mu}$ and $\mathcal{L}_{\epsilon,\mu}$}
\label{sec:proofThm11}
Theorem \ref{thm:M3} bounds the discrepancy between the discrete and continuous TMDmap generators with continuous KDE.
The proof of Theorem \ref{thm:M3} will make use of several lemmas. The first two lemmas propose random variables $Y_j$ and $Z_j$ that facilitate the proof of Theorem \ref{thm:M3} via the application of Bernstein's inequality \eqref{eq:Bernstein0}. 
\begin{lemma}
    \label{lemma:M3}
    Let $a > 0$ be any positive real number and $Y_j$ be defined by
    \begin{equation}
        \label{eq:Ydef}
        Y_j = F_jm_G -(m_F+a\epsilon m_G)G_j + a\epsilon m_G^2.
    \end{equation}
    Let $n\rightarrow\infty$ and $\epsilon\rightarrow0$ so that the condition \eqref{eq:thmM3_en} holds. 
Then the following are true:
\begin{enumerate}
    \item 
    \begin{equation}
        \mathbb{E}[Y_j] = 0\quad\forall 1\le j\le n,~j\neq i;
    \end{equation}
\item 
\begin{equation}
    \label{eq:Yi}
Y_i = \frac{\mu^{1/2}(x_i)}{\rho(x_i)} O(\epsilon^{1-d/2})\quad{\rm and}\quad\frac{\log n}{n\epsilon^3}Y_i = o(1);
\end{equation}
\item
    \begin{equation}
        \label{lemmaM3_i1}
        \mathbb{P}\left(\left[L_{\epsilon,\mu}f\right](x_i) -\left[\mathcal{L}_{\epsilon,\mu}f\right](x_i) \ge a\right) =
        \mathbb{P}\left( \sum_{j\neq i} Y_j \ge n\epsilon a m_G^2 -Y_i\right);
    \end{equation}
   \end{enumerate}
\end{lemma}
\begin{proof}
    \begin{enumerate}
        \item 
        To prove item 1 of Lemma \ref{lemma:M3} we take the expectation of $Y_j$ and readily find that $\mathbb{E}[Y_j] = 0$ for $j\neq i$.
        \item 
       Let us calculate $Y_i$. Noting that $k_{\epsilon}(x_i,x_i) = 1$ we write
       \begin{equation}
           \label{eq:Yi1}
           Y_i = \frac{\mu^{1/2}(x_i)f(x_i)}{\rho_{\epsilon}(x_i)}m_G - \frac{\mu^{1/2}(x_i)}{\rho_{\epsilon}(x_i)}(m_F+a\epsilon m_G) + a\epsilon m_G^2,
       \end{equation}
where 
\begin{equation}
    \label{eq:rhoepsLM3}
    \rho_{\epsilon}(x_i) = (\pi\epsilon)^{d/2}\left(\rho(x_i) + O(\epsilon)\right)
\end{equation}
and, according to  \eqref{Km2}, $m_F$ and  $m_G$ are 
\begin{align}
m_F & = f\mu^{1/2} + \frac{\epsilon}{4}\left[\Delta \left[f\mu^{1/2}\right] -
 \left[f\mu^{1/2}\right]\left(\frac{\Delta \rho}{\rho}-\omega\right)\right] + O(\epsilon^2)\quad{\rm and} \label{mFL3}\\
 m_G &=\mu^{1/2} + \frac{\epsilon}{4}\left[\Delta \mu^{1/2}- \mu^{1/2}\left(\frac{\Delta \rho}{\rho}-\omega\right)\right] + O(\epsilon^2).\label{mGL3}
\end{align}
All functions in  \eqref{mFL3} and \eqref{mGL3} are evaluated at $x_i$.
    Plugging \eqref{mFL3} and \eqref{mGL3} into \eqref{eq:Yi1} and taking into account that $f$ and its derivatives are $O(1)$ we obtain
    \begin{align}
        Y_i& = \frac{\mu^{1/2}f}{\rho_{\epsilon}}\left(\mu^{1/2} + O(\epsilon)\right) -\frac{\mu^{1/2}}{\rho_{\epsilon}}\left(f\mu^{1/2} +O(\epsilon)\right) + O(\epsilon)  \notag \\
        &= \frac{\mu^{1/2}}{\rho_{\epsilon}}O(\epsilon) +O(\epsilon) = \frac{\mu^{1/2}}{\rho} O(\epsilon^{1-d/2}).
    \end{align}
    Since $n\rightarrow\infty$ and $\epsilon\rightarrow0$ in such a manner that $n\epsilon^{2+d/2}\rightarrow \infty$ by assumption, we conclude that $n^{-1}\epsilon^{-2-d/2}\log^{-1}n\rightarrow 0$. This means that $n^{-1}\epsilon^{-3}(\log n) Y_i = o(1)$ as desired.

\item 
We will use identity \eqref{eq:LP_M3} to prove item 3.
First, we will show how the random variable $Y$ comes about.    
We start with several rearrangements:
\begin{align}
   &\thinspace\mathbb{P}\left(\left[L_{\epsilon,\mu}f\right](x_i) -\left[\mathcal{L}_{\epsilon,\mu}f\right](x_i) \ge a\right)  =
   \mathbb{P}\left(\frac{\sum_{j=1}^n F_j}{\sum_{j=1}^n G_j}  - \frac{m_F}{m_G} >\epsilon a\right)
   \notag \\
  = &\thinspace\mathbb{P}\Bigg(\frac{\sum_{j=1}^{n}\left(F(x_j)m_G - m_F G(x_j)\right)}{\sum_{j=1}^{n} G(x_j)m_G} \ge  a\epsilon \Bigg)\notag \\
    =&\thinspace \mathbb{P}\Bigg(\sum_{j=1}^{n}\left(F(x_j)m_G - (m_F + a\epsilon m_G) G(x_j)\right)  \ge  0 \Bigg) \notag\\
    =&\thinspace \mathbb{P}\Bigg(\sum_{j=1}^{n}\left(F(x_j)m_G - (m_F + a\epsilon m_G) G(x_j) - a\epsilon m_G^2 + a\epsilon m_G^2\right) \ge  0 \Bigg) \notag \\
    =&\thinspace \mathbb{P}\Bigg(\sum_{j=1}^{n}\left(F(x_j)m_G - (m_F + a\epsilon m_G) G(x_j) + a\epsilon m_G^2\right)  \ge  na\epsilon m_G^2 \Bigg).\label{devY}
\end{align}
The terms in the left-hand side of inequality \eqref{devY} are $Y_j$, $1\le j\le n$. All of them are random variables given $x_i$ except for $Y_i$. We transfer it to the right-hand side and take  item 2 into account:
\begin{align}
    &\thinspace\mathbb{P}\left(\left[L_{\epsilon,\mu}f\right](x_i) -\left[\mathcal{L}_{\epsilon,\mu}f\right](x_i) \ge a\right) \notag \\
    =&\thinspace
    \mathbb{P}\Bigg(\sum_{j\neq i}\left(F(x_j)m_G - (m_F + a\epsilon m_G) G(x_j) + a\epsilon m_G^2\right)  \ge  na\epsilon m_G^2 - Y_i \Bigg).
\end{align}

    \end{enumerate}

\end{proof}

\begin{lemma}
    \label{lemma:M4}
    Let $a > 0$ be any positive real number and $Z_j$ be defined by
    \begin{equation}
        \label{eq:Zdef}
        Z_j = F_jm_G -(m_F-a\epsilon m_G)G_j - a\epsilon m_G^2.
    \end{equation}
    Let $n\rightarrow\infty$ and $\epsilon\rightarrow0$ so that the condition \eqref{eq:thmM3_en} holds. 
Then the following are true:
\begin{enumerate}
    \item 
    \begin{equation}
        \mathbb{E}[Z_j] = 0\quad\forall 1\le j\le n,~j\neq i;
    \end{equation}
\item 
\begin{equation}
    \label{eq:Zi}
Z_i = \frac{\mu^{1/2}(x_i)}{\rho(x_i)} O(\epsilon^{1-d/2})\quad{\rm and}\quad\frac{\log n}{n\epsilon^3}Z_i = o(1);
\end{equation}
\item
    \begin{equation}
        \label{lemmaM4_i1}
        \mathbb{P}\left(\left[L_{\epsilon,\mu}f\right](x_i) -\left[\mathcal{L}_{\epsilon,\mu}f\right](x_i) \le  -a\right) =
        \mathbb{P}\left( \sum_{j\neq i} Z_j \le  -n\epsilon a m_G^2 -Z_i\right).
    \end{equation}
   \end{enumerate}
\end{lemma}

 \begin{proof}
     The first two items of Lemma \ref{lemma:M4} are proven similar to those of Lemma \ref{lemma:M3}. Let us work out the proof of item 3:
\begin{align}
   &\thinspace\mathbb{P}\left(\left[L_{\epsilon,\mu}f\right](x_i) -\left[\mathcal{L}_{\epsilon,\mu}f\right](x_i) \le -a\right)  =
   \mathbb{P}\left(\frac{\sum_{j=1}^n F_j}{\sum_{j=1}^n G_j}  - \frac{m_F}{m_G} \le-\epsilon a\right)
   \notag \\
  = &\thinspace\mathbb{P}\Bigg(\frac{\sum_{j=1}^{n}\left(F(x_j)m_G - m_F G(x_j)\right)}{\sum_{j=1}^{n} G(x_j)m_G} \le -a\epsilon \Bigg) \notag\\
    =&\thinspace \mathbb{P}\Bigg(\sum_{j=1}^{n}\left(F(x_j)m_G - (m_F - a\epsilon m_G) G(x_j)\right)  \le 0 \Bigg)\notag \\
    =&\thinspace \mathbb{P}\Bigg(\sum_{j=1}^{n}\left(F(x_j)m_G - (m_F - a\epsilon m_G) G(x_j) - a\epsilon m_G^2\right)  \le -na\epsilon m_G^2 \Bigg).\label{devZ}
\end{align}
The terms in the left-hand side of inequality \eqref{devZ} are $Z_j$, $1\le j\le n$. All of them are random variables given $x_i$ except for $Z_i$. As in the proof of Lemma \ref{lemma:M3}, we transfer it to the right-hand side, take  item 2 into account, and obtain the desired result:
\begin{align}
    &\thinspace\mathbb{P}\left(\left[L_{\epsilon,\mu}f\right](x_i) -\left[\mathcal{L}_{\epsilon,\mu}f\right](x_i) \le-a\right) \notag \\
    =&\thinspace
    \mathbb{P}\left(\sum_{j\neq i}\left(F(x_j)m_G - (m_F - a\epsilon m_G) G(x_j) - a\epsilon m_G^2\right)  \le -na\epsilon m_G^2 - Z_i \right).
\end{align}
     
 \end{proof}

The next two lemmas give estimates of the variance of random variables $Y$ and $Z$ respectively.
\begin{lemma}
\label{lemma:M5}
Let $x\in\mathcal{M}$ be any selected point on the manifold $\mathcal{M}$.
Let $y\in \mathcal{M}$ be any point sampled according to the density $\rho$.
Let $Y$ and $Z$ be defined, respectively,  by \eqref{eq:Ydef} and \eqref{eq:Zdef} where $x_j$ is replaced with $y$. Then 
\begin{equation}
    \label{eq:varY}
    {\sf Var}(Y) ={\sf Var}(Z) = \frac{\epsilon}{4}(2\pi\epsilon)^{-d/2}\left(\frac{\mu^2}{\rho}\|\nabla f\|^2 + O(\epsilon)\right).
\end{equation}
\end{lemma}
\begin{proof}
We will calculate the variance of $Y$. The variance of $Z$ is calculated likewise.

    Let $F$ and $G$ be defined by \eqref{F_j} and \eqref{G_j} respectively where $x_j$ is replaced with $y$:
    \begin{equation}
        \label{eq:FGdef}
        F: = \frac{k_{\epsilon}(x,y)\mu^{1/2}(y)f(y)}{\rho_{\epsilon}(y)}\quad{\rm and}\quad G: = \frac{k_{\epsilon}(x,y)\mu^{1/2}(y)}{\rho_{\epsilon}(y)}.
    \end{equation}
    First, we find $Y^2$:
    \begin{align*}
Y^2 &= \left[m_GF -(m_F+a\epsilon  m_G)G + a\epsilon m_G^2\right]^2\\
&=m_G^2F^2 + (m_F^2+2a\epsilon  m_Gm_F+(a\epsilon )^2m_G^2)G^2 + (a\epsilon )^2 m_G^4 \\
& -2m_G(m_F+a\epsilon  m_G)GF - 2a\epsilon (m_F+a\epsilon m_G)m_G^2G + 2a\epsilon  m_G^3F.
\end{align*}
Then, taking into account that $\mathbb{E}[Y] = 0$, we get
\begin{align}
{\sf Var}(Y)& = \mathbb{E}[Y^2] 
=m_G^2\mathbb{E}[F^2] + (m_F^2+2a\epsilon  m_Gm_F+(a\epsilon )^2m_G^2)\mathbb{E}[G^2] + (a\epsilon )^2 m_G^4 \notag\\
& -2m_G(m_F+(a\epsilon ) m_G)\mathbb{E}[GF] - 2a\epsilon (m_F+a\epsilon m_G)m_G^3 + 2a\epsilon  m_G^3m_F \notag\\
& = m_G^2\mathbb{E}[F^2]  -2m_Gm_F\mathbb{E}[GF]+ m_F^2\mathbb{E}[G^2]\notag\\
&+2a\epsilon m_G\left[m_F\mathbb{E}[G^2] -m_G\mathbb{E}[GF]\right] \notag \\
& + (a\epsilon )^2m_G^2\left[\mathbb{E}[G^2] - m_G^2\right]. \label{eq:varY}
\end{align}
Hence, to compute ${\sf Var}(Y)$, we need to calculate expectations of $F^2$, $G^2$, and $FG$. As in the proof of Lemma \ref{lemma:M1}, we will use the fact that $k_{\epsilon}^2(x,y) = k_{\epsilon/2}(x,y)$.
We also will need the expansion of $\rho^{-2}_{\epsilon}$:
\begin{equation}
\label{rhoe2}
\rho_{\epsilon}^{-2}(x): = 
(\pi\epsilon)^{-d}\rho^{-2}(x)
\left[1 -\frac{\epsilon}{2}
\left(\frac{\Delta\rho(x)}{\rho(x)} -\omega(x)\right) + O(\epsilon^2)\right].
\end{equation}
We start with $\mathbb{E}[G^2]$:
\begin{align}
\mathbb{E}[G^2] &= \int_{\mathcal{M}}k_{\epsilon/2}(x,y)\frac{\mu(y)}{(\pi\epsilon)^{d}\rho^2(y)}\left[1 -\frac{\epsilon}{2}
\left(\frac{\Delta\rho(y)}{\rho(y)} -\omega(y)\right) + O(\epsilon^2)\right]\rho(y)dy  \notag\\
& =(\pi\epsilon)^{-d} \int_{\mathcal{M}}k_{\epsilon/2}(x,y)\frac{\mu(y)}{\rho(y)}\left[1 -\frac{\epsilon}{2}
\left(\frac{\Delta\rho(y)}{\rho(y)} -\omega(y)\right) + O(\epsilon^2)\right]dy \notag \\
& = (\pi\epsilon)^{-d}\left(\frac{\pi\epsilon}{2}\right)^{d/2}\left[ \frac{\mu}{\rho} + \frac{\epsilon}{8}\left(\Delta\left[\frac{\mu}{\rho}\right] - \omega\frac{\mu}{\rho}\right)
-\frac{\epsilon}{2}\frac{\mu}{\rho}\left(\frac{\Delta\rho}{\rho} -\omega\right) + O(\epsilon^2)\right]\notag \\
& = (2\pi\epsilon)^{-d/2}\left[ \frac{\mu}{\rho} + 
\frac{\epsilon}{8}\left(\Delta\left[\frac{\mu}{\rho}\right] - 4\frac{\mu}{\rho}\frac{\Delta \rho}{\rho} +3\omega\frac{\mu}{\rho}\right) + O(\epsilon^2)\right].
\end{align}
The expectations of $F^2$ and $FG$ are computed likewise. The expressions for these expectations can be written  compactly as
\begin{align}
\mathbb{E}[G^2] & =  (2\pi\epsilon)^{-d/2}\left[ \frac{\mu}{\rho} + \frac{\epsilon}{8}\mathcal{S}\left(\frac{\mu}{\rho},\rho\right)+O(\epsilon^2)\right],\label{EG2}\\
\mathbb{E}[F^2] & =   (2\pi\epsilon)^{-d/2}\left[ \frac{f^2\mu}{\rho} + \frac{\epsilon}{8}\mathcal{S}\left(\frac{f^2\mu}{\rho},\rho\right)+O(\epsilon^2)\right],\label{EF2}\\
\mathbb{E}[FG] & =  (2\pi\epsilon)^{-d/2}\left[ \frac{f\mu}{\rho} + \frac{\epsilon}{8}\mathcal{S}\left(\frac{f\mu}{\rho},\rho\right)+O(\epsilon^2)\right],\label{EFG}
\end{align}
where the operator $\mathcal{S}(\cdot,\cdot)$ is defined by
\begin{equation}
\label{Sdef}
\mathcal{S}(u,v): = \Delta u- 4u\frac{\Delta v}{v} +3\omega u.
\end{equation}
Further,  ${\sf Var}(Y^2)$ given by \eqref{eq:varY} contains $m_F$ and $m_G$. The expressions for $m_F$ and $m_G$ given by \eqref{mFL3} and \eqref{mGL3} respectively can be compactified via the introduction of the operator $\mathcal{P}(\cdot,\cdot)$ defined by
\begin{align}
\mathcal{P}(u,v):= \Delta u-u\frac{\Delta v}{v}+\omega u.
\end{align}
Then
\begin{align}
m_F & = f\mu^{1/2} + \frac{\epsilon}{4}\mathcal{P}(f\mu^{1/2},\rho) + O(\epsilon^2). \label{mF1}\\
m_G &=\mu^{1/2} + \frac{\epsilon}{4}\mathcal{P}(\mu^{1/2},\rho) + O(\epsilon^2),\label{mG1} 
\end{align}
Plugging \eqref{EF2}, \eqref{EG2}, \eqref{EFG}, \eqref{mF1}, and \eqref{mG1} into \eqref{eq:varY}, doing some algebra, and tracking only the highest-order terms in $\epsilon$ we obtain 
\begin{align} 
{\sf Var}(Y) =  
 \frac{\epsilon}{8}  (2\pi\epsilon)^{-d/2}\left\{\mu\mathcal{S}\left(\frac{f^2\mu}{\rho},\rho\right)
-2f\mu\mathcal{S}\left(\frac{f\mu}{\rho},\rho\right)
+f^2\mu\mathcal{S}\left(\frac{\mu}{\rho},\rho\right)+ 
O(\epsilon)\right\}. \label{EY2}
\end{align}
Remarkably, all operators $\mathcal{P}$ cancel out.
Next, we decode the operator $\mathcal{S}$ in \eqref{EY2} using \eqref{Sdef} and get the desired result:
\begin{align}
{\sf Var}(Y)  & =  \frac{\epsilon}{8}  (2\pi\epsilon)^{-d/2} 
\left\{ \mu\Delta  \left(\frac{f^2\mu}{\rho}\right) -
2 f\mu \Delta  \left(\frac{f\mu}{\rho}\right) + f^2\mu\Delta  \left(\frac{\mu}{\rho}\right) + O(\epsilon)\right\} \notag \\
& =   \frac{\epsilon}{8}  (2\pi\epsilon)^{-d/2} \left\{ \mu(\Delta  f^2)\frac{\mu}{\rho} +
2\mu\nabla f^2\cdot\nabla \frac{\mu}{\rho} + \mu f^2\Delta \frac{\mu}{\rho}\right. \notag \\
&\left.-2f\mu(\Delta  f)\frac{\mu}{\rho} -4f\mu\nabla f\cdot\nabla \frac{\mu}{\rho} - 2 \mu f^2\Delta \frac{\mu}{\rho} 
 +f^2\mu\Delta \frac{\mu}{\rho}+O(\epsilon) \right\} \notag \\
 & =  \frac{\epsilon}{8}  (2\pi\epsilon)^{-d/2} \left\{ \frac{\mu^2}{\rho}\left(2\|\nabla f\|^2 +2f\Delta f\right) +4\mu f\nabla f\cdot\nabla \frac{\mu}{\rho}
 \right.\notag \\
 &\left.-2f\mu(\Delta  f)\frac{\mu}{\rho} -4f\mu\nabla f\cdot\nabla \frac{\mu}{\rho}+O(\epsilon) \right\} \notag \\
 & = \frac{\epsilon}{4}  (2\pi\epsilon)^{-d/2}\left( \frac{\mu^2}{\rho}\|\nabla  f\|^2 + O(\epsilon)\right). \label{eq:varY_2}
 \end{align}
\end{proof}


Now we proceed to prove Theorem \ref{thm:M3}.
\begin{proof}
    By assumption, $\rho$ and $f$ are bounded. Therefore, the random variables $Y$ and $Z$ are bounded as well. After conditioning on the point $x_i$ we drop into the regime of Lemmas \ref{lemma:M3}, \ref{lemma:M4}, and \ref{lemma:M5}.  Let $|Y_j|\le M_Y$ and $|Z_j|\le M_Z$. 
    Then, according to Bernstein's inequality \eqref{eq:Bernstein0} we have:
    \begin{align}
    \label{eq:BM3_1}
         \mathbb{P}_{x_i}\left( \sum_{j\neq i} Y_j \ge na\epsilon m_G^2 -Y_i\right)\le 
         \exp\left(
         - \frac{(na\epsilon m_G^2 -Y_i)^2}{2(n-1){\sf Var}(Y) + 2M_Y(na\epsilon m_G^2 -Y_i)}
         \right)
    \end{align}
Let us simplify the argument of the exponent. According to \ref{lemma:M3},  $Y_i = O(\epsilon^{1-d/2})$. 
Therefore, the expression $na\epsilon m_G^2 -Y_i$ can be estimated as
\begin{align}
\label{eq:p_rhs1}
    na\epsilon m_G^2 -Y_i =na\epsilon m_G^2\left(1-\frac{O(\epsilon^{1-d/2})}{na\epsilon m_G^2}\right) =  na\epsilon m_G^2\left(1+o(\epsilon)\right)  
\end{align}
provided that $na\epsilon^{d/2} \rightarrow\infty$ as $\epsilon\rightarrow 0$ and $n\rightarrow \infty$. We will verify this condition later.
Plugging \eqref{eq:varY} and \eqref{eq:p_rhs1} into \eqref{eq:BM3_1} and canceling $n$ we obtain
\begin{align}
         &\mathbb{P}_{x_i}\left( \sum_{j\neq i} Y_j \ge na\epsilon m_G^2 -Y_i\right)\le\notag \\ 
         &\exp\left(
         - \frac{ na^2\epsilon^2 m_G^4\left(1+o(\epsilon)\right)}{2\frac{n-1}{n}\frac{\epsilon}{4}  (2\pi\epsilon)^{-d/2}\left( \frac{\mu^2}{\rho}\|\nabla f\|^2 + O(\epsilon)\right) + 2M_Ya\epsilon m_G^2\left(1+o(\epsilon)\right)}
         \right) = \notag \\
          &\exp\left(
         - \frac{ 2(2\pi)^{d/2}na^2\epsilon^{1+d/2} m_G^4\left(1+o(\epsilon)\right)}{ \frac{\mu^2}{\rho}\|\nabla f\|^2 + O(\epsilon) + o(1)} \right) =\notag \\
         &\exp\left(
         - \frac{ 2(2\pi)^{d/2}na^2\epsilon^{1+d/2} m_G^4}{ \frac{\mu^2}{\rho}\|\nabla f\|^2} \left(1+o(1)\right)\right). \label{eq:BM3_2}        
\end{align}
Similarly, the following bound for the sum of $Z_j$s is obtained:
\begin{align}
         &\mathbb{P}_{x_i}\left( \sum_{j\neq i} Z_j \le -na\epsilon m_G^2 -Z_i\right)\le\exp\left(
         - \frac{ 2(2\pi)^{d/2}na^2\epsilon^{1+d/2} m_G^4}{ \frac{\mu^2}{\rho}\|\nabla f\|^2} \left(1+o(1)\right)\right).
         \label{eq:BM3_3} 
\end{align}

Lemmas \ref{lemma:M3} and \ref{lemma:M4} and inequalities \eqref{eq:BM3_2} and \eqref{eq:BM3_3} imply that 
\begin{align}
    \mathbb{P}_{x_i}\Bigl(\Big|\left[L_{\epsilon,\mu}f\right](x_i) -\left[\mathcal{L}_{\epsilon,\mu}f\right](x_i)\Big| \ge a\Bigl)  \le 2\exp\left(
         - \frac{ 2(2\pi)^{d/2}na^2\epsilon^{1+d/2} m_G^4}{ \frac{\mu^2}{\rho}\|\nabla f\|^2} \left(1+o(1)\right)\right) \label{eq:BM3_4}
\end{align}
The next step is to find $\alpha$ such that the right-hand side of inequality \eqref{eq:BM3_4}  is less or equal that $n^{-3}$ as it was done in the proofs of Theorems \ref{thm:M1} and \ref{thm:M2}. Hence, we set
\begin{align}
    2\exp\left(
         - \frac{ 2(2\pi)^{d/2}na^2\epsilon^{1+d/2} m_G^4}{ \frac{\mu^2}{\rho}\|\nabla f\|^2} \left(1+o(1)\right)\right) \le n^{-3}.
\end{align}
Taking logarithms of both sides, recalling that $m_G = \mu^{1/2}(x_i) +O(\epsilon)$, and doing some algebra we obtain
\begin{equation}
    \label{eq:alpha1_thmM3}
    a \ge \frac{\|\nabla f\|}{\rho^{1/2}(x_i)(2\pi)^{d/4}}\sqrt{\frac{(3\log n + \log 2)}{2n\epsilon^{1+d/2}}}(1+o(1)).
\end{equation}
Hence, for large enough $n$ and small enough $\epsilon$ we can choose
\begin{equation}
    \label{eq:alpha2_thmM3}
    a = \frac{2\|\nabla f\|}{\rho^{1/2}(x_i)(2\pi)^{d/4}}\sqrt{\frac{\log n}{n\epsilon^{1+d/2}}}.
\end{equation}

{Let us check that the term in \eqref{eq:p_rhs1} is indeed $o(\epsilon)$. To do so, we check that $\epsilon na\epsilon^{d/2}\rightarrow\infty$ as $\epsilon\rightarrow 0$ and $n\rightarrow\infty$. Indeed, with the use of by assumption \eqref{eq:thmM3_en} we get
\begin{equation}
    \label{eq:nae}
    \epsilon na\epsilon^{d/2}\propto \sqrt{\frac{n^2\epsilon^{2 + d}\log n}{n\epsilon^{1+d/2}}} = \sqrt{n\epsilon^{1+d/2}\log n}
    \rightarrow\infty.
\end{equation}}
Finally, we observe that $a$ given by \eqref{eq:alpha2_thmM3} can be expressed as 
\begin{equation}
    \label{eq:alpha3_thmM3}
    a = \frac{2\|\nabla f\|}{\rho^{1/2}(x_i)}\epsilon^{3/2}\alpha,
\end{equation}
where $\alpha$ is given by \eqref{eq:thmM3_alpha}. This gives the desired result and completes the proof.
\end{proof}

\subsection{Numerical solution of BVPs: proofs}
\label{subsec: numerical solutions of bvp proofs}
Theorem \ref{thm: BVP error estimate} 
establishes an error bound for the numerical solution by the TMDmap algorithm to the general boundary-value problem (BVP) of the form \eqref{eq: general boundary value problem}.
The proof presented below utilizes a \emph{maximum principle}-based argument, specifically, { the \emph{method of comparison functions}} commonly used in stability estimates for finite difference methods {\color{blue} \cite{morton_mayers_2005}}. A closely related argument was proposed in \cite[p. 35]{jiang2023ghost} for \emph{Ghost Point Diffusion Maps}.

Under Assumptions 1-5, we consider the BVP \eqref{eq: general boundary value problem}. For cleaner notation, we set 
\begin{align}
    L &:= L^{(n)}_{\epsilon,\mu}, \\
    \mathcal{I} &:= \{i \mid x_i \in \mathcal{M} \setminus \Omega\}, \\
    \mathcal{D} &:= \{i \mid x_i \in \Omega \}. 
\end{align}

{To prove Theorem \ref{thm: BVP error estimate} we will need several lemmas.
\begin{lemma}[Discrete maximum principle]
\label{lem:max_principle}
Let $\mathcal{D} \neq \emptyset$ and $v =  [v^\mathcal{I} \ v^{\mathcal{D}}]^{\top} \in \mathbb{R}^n$.
\begin{enumerate}
    \item If $[Lv]_i \ge 0$ for all $i\in\mathcal{I}$ then $v_i$ reaches its maximum on $\mathcal{D}$ i.e.,
    \begin{equation}
    \label{eq:maxv}
        \max_{i\in\mathcal{I}\cup \mathcal{D}} v_i =  \max_{i\in \mathcal{D}} v_i. 
    \end{equation}
    \item If $[Lv]_i \le 0$  for all $i\in\mathcal{I}$ then $v_i$ reaches its minimum on $\mathcal{D}$, i.e.,
    \begin{equation}
    \label{eq:minv}
        \min_{i\in\mathcal{I}\cup\mathcal{D}} v_i  =   \min_{i\in \mathcal{D}} v_i. 
    \end{equation}
\end{enumerate}
\end{lemma}
}
\begin{proof}
We will prove item 1. The proof of item 2 readily follows by redefining $v$ as $-v$.

By construction, $L = \epsilon^{-1}(P - I)$ where $P$ is a Markov matrix. All entries of any Markov matrix are nonnegative and row sums are 1. In our case, $P = P^{(n)}_{\epsilon,\mu}$. Its off-diagonal entries are positive. Hence, all off-diagonal entries of $L$ are positive and row sums of $L$ are zero. This means that $\sum_{j=1}^nL_{ij} = 0$ for all $i\in\mathcal{I}\cup\mathcal{D}$, and diagonal entries of $L$ are negative. 
To construct a contradiction, we assume that there is $i^{\ast} \in\mathcal{I}$ such that 
$$
v_{i^\ast}  = \max_{i\in \mathcal{I}\cup\mathcal{D}} v_i > \max_{i\in\mathcal{D}}v_i.
$$   
Since $i^\ast\in\mathcal{I}$, all entries $L_{i^\ast j} > 0$ for $j\in \mathcal{D}$ and for all $j\in\mathcal{I}\backslash\{i^{\ast}\}$. 
Therefore,
\begin{align}
    0\le \sum_{j\in\mathcal{I}\cup\mathcal{D}}L_{i^\ast j}v_j  = L_{i^{\ast}i^{\ast}} v_{i^{\ast}} + \sum_{j\neq i^{\ast} }L_{i^\ast j}v_j  < v_{i^{\ast}}\sum_{j\in\mathcal{I}\cup\mathcal{D}}L_{i^\ast j} = 0.
\end{align}
Hence, the maximum of $v$ must be reached of $\mathcal{D}$. This implies \eqref{eq:maxv}.
\end{proof}
\begin{remark}
    In practice, the matrix $L$ is often sparsified for computational efficiency. In this case, the discrete maximum principle in Lemma \ref{lem:max_principle} still holds. Its proof requires a slightly longer argument involving the irreducibility of $L$ as in the proof of Lemma 6.1 in \cite{morton_mayers_2005}. 
\end{remark}
\begin{lemma}[The maximum principle]
    \label{lem:max_principle_cont}
    \begin{enumerate}
        \item  Let $\mathcal{L}\phi\ge 0$ in $\mathcal{M}\backslash\Omega$ where $\mathcal{L}$ is the generator given by \eqref{eq: generator}. 
    Then the maximum of $\phi$ is reached on the boundary $\partial\Omega$, i.e.,
    \begin{equation}
        \label{eq:max_pr_cont}
        \max_{x\in\mathcal{M}\backslash\Omega} \phi = \max_{\partial\Omega}\phi.
    \end{equation}
\item  Let $\mathcal{L}\phi\le 0$ in $\mathcal{M}\backslash\Omega$ where $\mathcal{L}$ is the generator given by \eqref{eq: generator}. 
    Then the minimum of $\phi$ is reached on the boundary $\partial\Omega$, i.e.,
    \begin{equation}
        \label{eq:max_pr_cont}
        \min_{x\in\mathcal{M}\backslash\Omega} \phi = \min_{\partial\Omega}\phi.
    \end{equation}
    \end{enumerate}
   \end{lemma}

    This is a standard result for elliptic operators (see Theorem 1 in Section 6.4.1 in \cite{evans10}). 

\begin{lemma}
\label{lem: uniform TMD estimation}    
{ Let $\epsilon\rightarrow 0$ and $n\rightarrow \infty$ so that the condition \eqref{eq:thmM4_en} holds. Then}
the probability that for every $x_i \in \mathcal{X}(n)$ 
\begin{align}
    |4\beta^{-1}\left[L^{(n)}_{\epsilon,\mu}u\right](x_i) - \mathcal{L}u(x_i)| &\medspace{\le} 
    \medspace\frac{\alpha\epsilon}{\rho^{1/2}(x_i)}\left(2\|\nabla u(x_i)\|\epsilon^{1/2} + 11{|u(x_i)|}\right) \label{eq:lemma8}\\
    &+ \epsilon \left|\mathcal{B}_{1}[u,\mu] + \mathcal{B}_{2}[u,\mu, \rho] + \mathcal{B}_{{3}}[u,\mu, \rho] \right|  {+ O(\epsilon^2)} \notag
\end{align}
is greater or equal to $1 -2n^{-2}$. 
\end{lemma}

\begin{proof}
    Let $E(x_i)$ be the event 
        that where inequality \eqref{eq:lemma8} holds for $x_i$. According to Theorem \ref{thm:main}, $\mathbb{P}(E(x_i))>1-2n^{-3}$. Hence    
    \begin{align}
        \mathbb{P}\left(\bigcap_{i=1}^nE(x_i)\right)& = 1 - \mathbb{P}\left(\bigcup_{i=1}^nE(x_i)^c\right) \ge 1 - \sum_{i=1}^n\mathbb{P}(E(x_i)^c)\notag\\
        &\ge 1- \sum_{i=1}^n2n^{-3} = 1- 2n^{-2}.
    \end{align}
\end{proof}

\begin{lemma}\label{lem:prob of sets lemma}
    Let $\mathcal{X}(n)$ be a point cloud on $\mathcal{M}$ of $n$ points sampled i.i.d. with density $\rho$. The probability of $\mathcal{D}$ being non-empty tends to 1 exponentially: 
    \begin{align}
        \mathbb{P}(\mathcal{D} \neq \emptyset)  \geq  1 - \exp{(-n\mathbb{E}_{\rho}[\mathbb{I}_{\Omega}])} \label{eq:prob of D nonempty}.
    \end{align}
\end{lemma}

\begin{proof}
Obviously,
$$
\mathbb{P}(\mathcal{D} \neq \emptyset) =  1 - \mathbb{P}(\mathcal{D} =\emptyset) = 1- \mathcal{P}\left(\bigcap_{i} x_i \notin \Omega\right).
$$ 
Since $x_i$, $1\le i\le n$, are i.i.d. samples, the probability that none of $x_i$ lies in $\Omega$ is the product of probabilities that each $x_i$ does not lie in $\Omega$, i.e.,
$$
\mathcal{P}\left(\bigcap_{i} x_i \notin \Omega\right) = \prod_{i}\mathcal{P}(x_i\notin\Omega).
$$
In turn, $\mathcal{P}(x_i\notin\Omega) = 1 - \mathbb{E}_{\rho}[\mathbb{I}_{\Omega}]$.
Therefore, 
$$
\mathbb{P}(\mathcal{D} \neq \emptyset) = 1- \prod_{i}\left(1 - \mathbb{E}_{\rho}[\mathbb{I}_{\Omega}] \right) = 1 -\left( 1 - \mathbb{E}_{\rho}[\mathbb{I}_{\Omega}]\right)^n  \ge 1 - \exp{(-n\mathbb{E}_{\rho}[\mathbb{I}_{\Omega}])}.
$$
The last inequality follows from the fact that $1-x\le e^{-x}$.
\end{proof}

{ Now we are ready to prove Theorem \ref{thm: BVP error estimate}.}
\begin{proof}
Let $u$ be the exact solution to \eqref{eq: general boundary value problem}.
Let $v$ be the numerical solution, i.e., $v$ satisfies \eqref{eq:unum}.
For brevity, we will denote $L_{\epsilon,\mu}^{(n)}$ by $L$. This argument uses the approach from \cite[p.195-196]{morton2005numerical}. We have
\begin{align}
   \begin{cases} \mathcal{L}u =f,& x\in\mathcal{M}\backslash\Omega\\
   u = g,& x \in\Omega\end{cases}\quad{\rm and}\quad
   \begin{cases}[4\beta^{-1}Lv]_i =f_i,& x_i\in\mathcal{M}\backslash\Omega\\
   v(x_i) = g(x_i),& x_i \in\Omega\end{cases}.   
\end{align}
By Lemmas \ref{lem: uniform TMD estimation} and \ref{lem:prob of sets lemma}, with high probability greater or equal to $1 - 2n^{-2} - \exp{(-n\mathbb{E}[I_{\Omega}])}$, for all $x_i \in\mathcal{X}(n)$, all $n$ large enough and $\epsilon$ small enough, there is a positive function $C_{u}(x)$ that depends only  on $x$, $u$, $\mu$, and $\rho$ but not on $\epsilon$ and $n$ such that
\begin{align}
    |4\beta^{-1}[Lu](x_i) - \mathcal{L}u(x_i)| 
    \le
    C_u(x_i)\epsilon.  \label{eq:lemma8u}
\end{align}
For brevity, we will use the subscript $i$ to denote the value at $x_i$. Inequality \eqref{eq:lemma8u} implies that 
\begin{equation}
\label{eq:longchain1}
  f_i -\epsilon C_{u,i} = [\mathcal{L}u]_i - \epsilon C_{u,i} \le  [4\beta^{-1}Lu]_i \le [\mathcal{L}u]_i + \epsilon C_{u,i} = f_i + \epsilon C_{u,i}. 
\end{equation}

To construct comparison functions to which we will apply the discrete maximum principle, we will need an auxiliary function $\phi$. It is convenient to choose $\phi$  to be the solution to the boundary-value problem
\begin{equation}
\label{eq:aux1}
   \begin{cases} \mathcal{L}\phi = 1,& x\in\mathcal{M}\backslash\Omega\\
   \phi = 0,& x \in\partial \Omega\end{cases}.    
\end{equation}
According to the maximum principle, precisely item 1 of Lemma \ref{lem:max_principle_cont}, the maximum of $\phi$ is reached on the boundary. Since $\phi$ is zero on the boundary, $\phi(x) \le 0 $ for $x\in \mathcal{M}\backslash\Omega$. 
Moreover, according to Lemmas \ref{lem: uniform TMD estimation} and \ref{lem:prob of sets lemma}, for $\epsilon$ small enough and $n$ large enough, there exists a positive function $C_{\phi}(x)$ that depends on $x$, $u$, $\mu$, and $\rho$ but not on $\epsilon$ and $n$ such that
\begin{equation}
    \label{eq:Lphi_bound}
     1 -\epsilon C_{\phi,i} = [\mathcal{L}\phi]_i - \epsilon C_{\phi,i} \le  [4\beta^{-1}L\phi]_i \le [\mathcal{L}\phi]_i + \epsilon C_{\phi,i} = 1 + \epsilon C_{\phi,i}. 
\end{equation}
with probability greater of equal to $1 - 2n^{-2} - \exp{(-n\mathbb{E}[I_{\Omega}])}$. 

Now we define the first comparison function 
\begin{equation}
    \label{eq:aux2}
    \psi = u - v - \epsilon \left(\max_{1\le i\le n}C_{u,i}+1\right)\phi \equiv u-v-\epsilon(C_{\max} + 1)\phi.
\end{equation}
Equations \eqref{eq:longchain1} and \eqref{eq:Lphi_bound} imply that,  with high probability,  
$[L\psi]_i \le 0$ for all $i\in\mathcal{I}$ if $\epsilon$ is small enough.
Indeed, with probability greater or equal to $1 - 4n^{-2} -\exp(-n\mathbb{E}[I_{\Omega}])$,  
\begin{align*}
    [4\beta^{-1}L\psi]_i &= [4\beta^{-1}L u]_i - [4\beta^{-1}L v]_i - \epsilon ( C_{\max} + 1)[4\beta^{-1}L \phi]_i \\
    &\le f_i +\epsilon C_{u,i} - f_i -\epsilon ( C_{\max} + 1)(1-C_{\phi,i}\epsilon)  \notag\\
    &\le -\epsilon \left(1 - C_{\phi,i}\epsilon (1 + C_{\max})\right) 
    \le 0.
\end{align*}
We did not multiply the exponent by 2 in the probability estimate $1 - 4n^{-2} -\exp(-n\mathbb{E}[I_{\Omega}])$ because the same set of points is used in the estimates for $u$ and $\phi$.

Hence, by the discrete maximum principle in Lemma \ref{lem:max_principle}, the minimum of $\psi$ is achieved on $\mathcal{D}$. Since $u_i = v_i = g_i$ for all  $i\in\mathcal{D}$ and $\phi_i = 0$ for $i\in \mathcal{D}$, it follows that $\psi_i = 0$ for all $i\in \mathcal{D}$. Therefore, $\psi_i \ge 0$ for all $1\le i\le n$, i.e.,
\begin{equation}
    \label{eq:uv1}
    u_i-v_i \ge \epsilon(C_{\max} + 1)\phi_i\equiv -\epsilon(C_{\max} + 1)|\phi_i| \quad \forall 1\le i\le n.
\end{equation}

Next, we define the second comparison function
\begin{equation}
    \label{eq:aux3}
    \psi = u - v + \epsilon \left(\max_{1\le i\le n}C_{u,i}+1\right)\phi \equiv u-v +\epsilon(C_{\max} + 1)\phi
\end{equation}
and argue that $[L\psi]_i \ge 0$ for all  $i\in\mathcal{I}$ with high probability. Indeed, \eqref{eq:longchain1} and \eqref{eq:Lphi_bound} imply that with probability greater or equal to $1 - 4n^{-2} -\exp(-n\mathbb{E}[I_{\Omega}])$,  
\begin{align*}
    [4\beta^{-1}L\psi]_i &= [4\beta^{-1}L u]_i - [4\beta^{-1}L v]_i + \epsilon ( C_{\max} + 1)[4\beta^{-1}L \phi]_i \\
    &\ge f_i -\epsilon C_{u,i} - f_i +\epsilon ( C_{\max} + 1)(1 -C_{\phi,i}\epsilon)  \notag\\
    &\ge \epsilon \left(1 - C_{\phi,i}\epsilon (1 + C_{\max})\right) 
    \ge 0
\end{align*}
for $\epsilon$ small enough. By Lemma \ref{lem:max_principle}, the maximum of $\psi_i$ is achieved on $\mathcal{D}$. Since $u_i = v_i = g_i$ and $\phi_i =0 $ for $i\in\mathcal{D}$, $\psi_i \le 0$ for all $i\in\mathcal{D}$. Hence $\phi_i\le 0$ for all $1\le i\le n$, i.e.,
\begin{equation}
    \label{eq:uv2}
    u_i-v_i \le -\epsilon(C_{\max} + 1)\phi_i \equiv \epsilon(C_{\max} + 1) |\phi_i| \quad \forall 1\le i\le n.
\end{equation}
Inequalities \eqref{eq:uv1} and \eqref{eq:uv2} imply that, for $\epsilon$ small enough, with probability greater or equal to $1 - 4n^{-2} -\exp(-n\mathbb{E}[I_{\Omega}])$,  
\begin{equation}
    \label{eq:uvbound}
    |u(x_i) - v(x_i)| \le  \epsilon(C_{\max} + 1)|\phi(x_i)| \quad\forall 1\le i\le n 
\end{equation}
where $\phi$ is the solution to \eqref{eq:aux1}.
Finally, we note that 
$$
C_{\max} = \max_{1\le i\le n}C_{u,i} \le \max_{x\in\mathcal{M}}C_{u,\mu,\rho}(x)
$$
where $C_{u,\mu,\rho}(x)$ is defined be \eqref{eq:C} in the statement of Theorem \ref{thm: BVP error estimate}.
\end{proof}

\section{Test problems}\label{sec: applications}
\subsection{Illustrating error model}\label{subsec: illustrating model}
In Corollary \ref{cor: reducing bias} it was shown that bias error prefactors  $\mathcal{B}_2$ and $\mathcal{B}_3$ can be eliminated when $\rho$ is uniform on $\mathcal{M}$ and when the test function $f$ is the committor. 
{ One can expect that
these cancellations reduce the magnitude of the overall bias error prefactor $\mathcal{B}_1 + \mathcal{B}_2 + \mathcal{B}_3$, though this might not be necessary.} 
The example presented in this section demonstrates that these cancellations indeed reduce the magnitude of the prefactor, thus showing that using uniform densities and approximating the committor function is a viable strategy for improving the accuracy of TMDmap. 

\subsubsection{Theoretical considerations} The leading order terms in the error model in Theorem \ref{thm:main} can be briefly stated as follows: 
\begin{align}
    |4\beta^{-1}L^{(n)}_{\epsilon,\mu}f(x) - \mathcal{L}f(x)| \leq \sqrt{\frac{\log n}{n \epsilon^{2+d/2}}}|\mathcal{V}[f,\mu,\rho](x)| + \epsilon |\mathcal{B}[f,\mu,\rho](x)|. \label{eq: summarized error model}
\end{align}
Here $\mathcal{V}$ and $\mathcal{B}$ are bias and variance prefactors respectively given by 
\begin{align}
    \mathcal{V} &:= \frac{1}{\rho(x)^{1/2}}\left(2\|\nabla f(x)\|\epsilon^{1/2} + 11|f(x)|\right), \\
    \mathcal{B} &:= \mathcal{B}_1 + \mathcal{B}_{2} + \mathcal{B}_3. 
\end{align}
$\mathcal{V}$ and $\mathcal{B}$ are both $O(1)$ in $n$ and $\epsilon$. While \eqref{eq: summarized error model} is an upper bound on the absolute error $|4\beta^{-1}L^{(n)}_{\epsilon,\mu}f(x) - \mathcal{L}f(x)|$ due to the presence of the \emph{variance error}, the \emph{bias error} estimate from Theorem \ref{thm: bias error} is asymptotically sharp and does not involve absolute values. Thus if $\epsilon$ and $n$ are scaled such that the variance error is fixed, then with high probability the error behaves approximately as a linear function in $\epsilon$, i.e.
\begin{align}
    4\beta^{-1}L^{(n(\epsilon))}_{\epsilon,\mu}f(x) - \mathcal{L}f(x) \sim a + \epsilon b. \label{eq:linearmodel}
\end{align}
To fix the variance error, the number of data points $n$ must depend on $\epsilon$ such that 
\begin{align}
    \frac{\log n(\epsilon)}{n(\epsilon) \epsilon^{2+d/2}} \sim O(1). \label{eq:n,epsilon choice}
\end{align}
After fixing the variance error, the slope of the linear function \eqref{eq:linearmodel} yields an estimate for the prefactor of the bias error. Using this method, we will obtain a prefactor estimate for a 1D test system embedded in 2D. 

\subsubsection{1D test system}
\label{sec:unitcircle}
Consider a system 
on the unit circle $\mathbb{S}^1$ parameterized by arc length (Figure \ref{fig: 1D potential mu}).
Let it evolve according to the overdamped Langevin dynamics \eqref{eq: old} with 
\begin{align}
    V(\theta) = \left(4\cos^2\left(\frac{\theta}{2}\right)- \frac{3}{2}\right)^2.\label{eq: 1d potential}
\end{align}
\begin{figure}[h]
    \centering
    \includegraphics[width=0.435\textwidth]{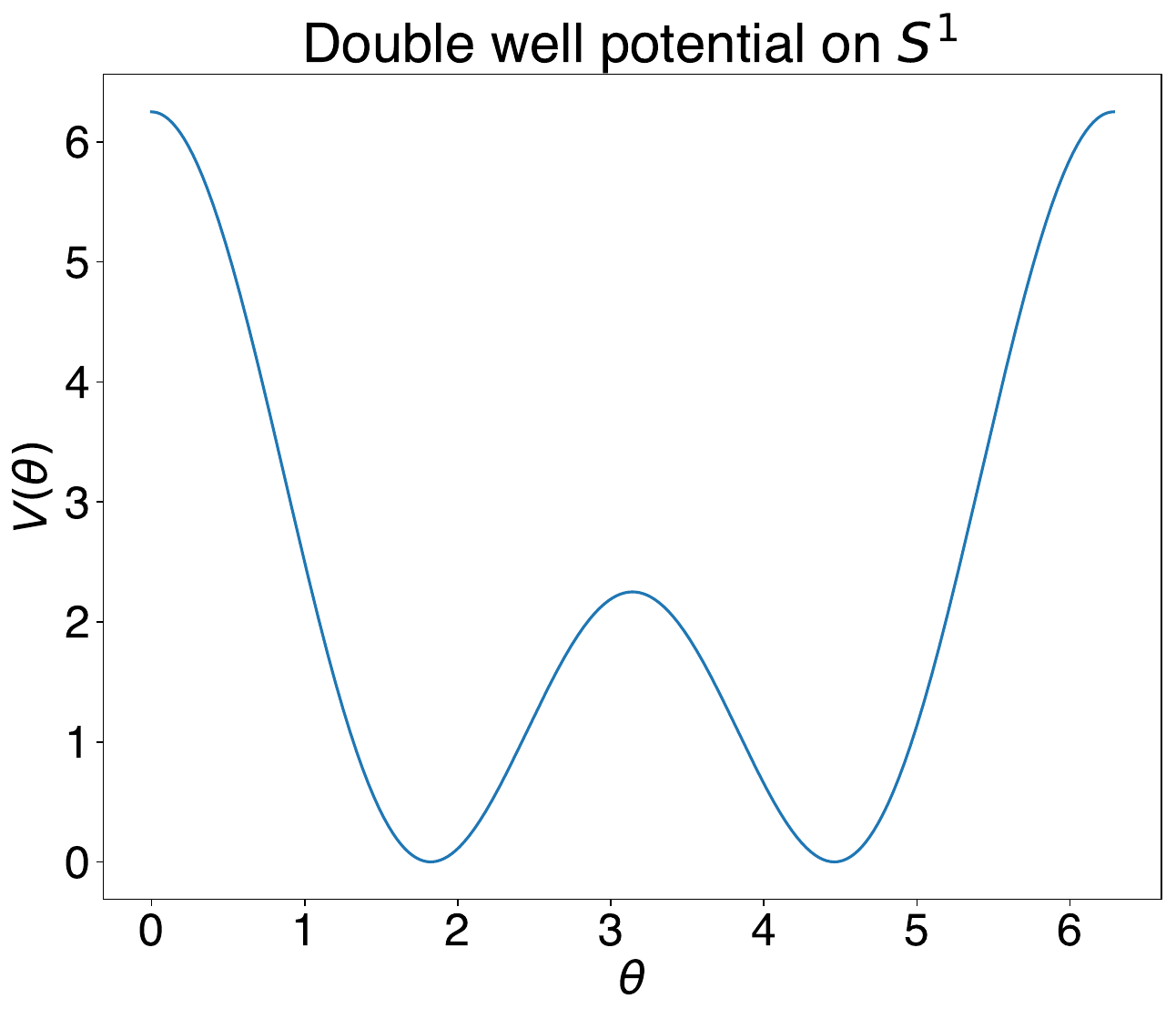}
    \includegraphics[width=0.45\textwidth]{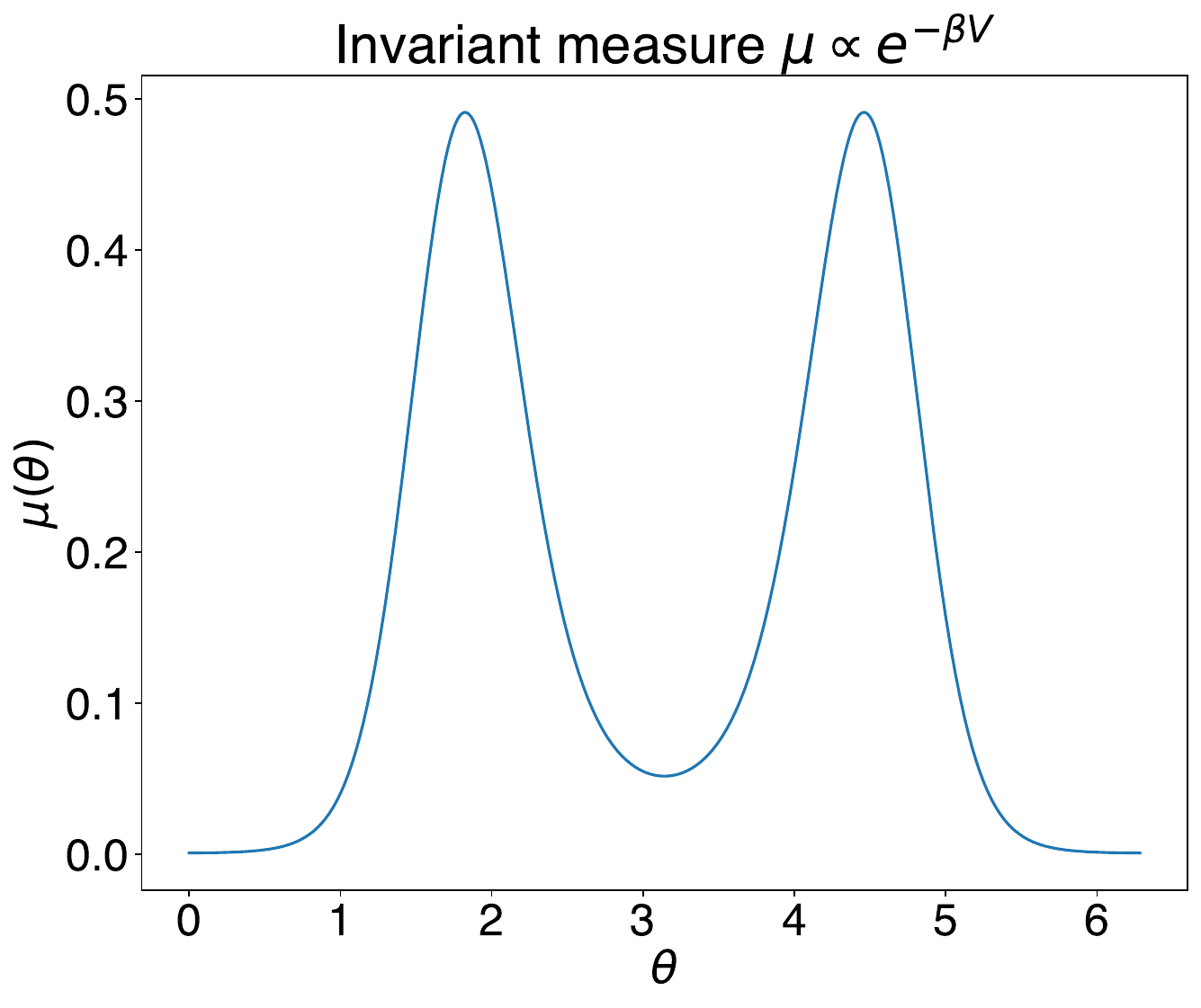}
    \caption{The potential $V(\theta)$, $0\le\theta<2\pi$, given by \eqref{eq: 1d potential} and the invariant measure $\mu$ for a 1D system on $\mathbb{S}^1$ following the overdamped Langevin dynamics \eqref{eq: old} with the drift $\nabla V$.
    } 
    \label{fig: 1D potential mu}
\end{figure}

In this case, the committor can be derived analytically. Let $0 < \theta_1 < \theta_2 < 2\pi$ be the minima of $V$ where 
\begin{align}
    \theta_{1} = 2\cos^{-1}\left(-\frac{\sqrt{3}}{2\sqrt{2}}\right)\approx 1.8235, \quad \theta_{2} = 2\cos^{-1}\left(\frac{\sqrt{3}}{2\sqrt{2}}\right)\approx 4.4597.
\end{align}
Set $A = [\theta_1-r, \theta_1+r]$ and $B = [\theta_2-r, \theta_2+r]$ where $r = 0.1$. Moreover, fix $\beta = 1$. Then the committor is 
\begin{align}
   q(\theta) & = \frac{\int_{\theta_{1}+r}^{\theta}\exp{\left(\beta V(\theta')\right)}\,d\theta'}{\int_{\theta_{1}+r}^{\theta_{2}-r}\exp{\left(\beta V(\theta')\right)}\,d\theta'}, \quad &\theta \in [\theta_1 + r, \theta_2 - r],\notag\\
   q(\theta) & =\frac{\int_{\theta}^{2\pi}\exp{\left(\beta V(\theta')\right) }\,d\theta'+\int_{0}^{\theta_1-r}\exp{\left(\beta V(\theta')\right)}\,d\theta'}{\int_{\theta_{2}+r}^{2\pi}\exp{\left(\beta V(\theta')\right)}\,d\theta'    +\int_{0}^{\theta_1-r}\exp{\left(\beta V(\theta')\right)}\,d\theta'
   }, \quad &\theta \in [\theta_2 + r, 2\pi), \label{qtheta}\\
   q(\theta) & =\frac{\int_{\theta}^{\theta_1-r}\exp{\left(\beta V(\theta')\right)}\,d\theta'}{\int_{\theta_{2}+r}^{2\pi}\exp{\left(\beta V(\theta')\right)}\,d\theta'    +\int_{0}^{\theta_1-r}\exp{\left(\beta V(\theta')\right)}\,d\theta'
   }, \quad &\theta \in [0, \theta_1-r]. \notag
\end{align}

\subsubsection{The experimental setup}
To demonstrate the error model in Theorem \ref{thm:main} and an improvement in the bias error prefactor \eqref{eq: bias error term}, the following experiment is conducted.
\begin{enumerate}
    \item  For each $\epsilon$, the number $n(\epsilon)$ is chosen such that the error behaves according to \eqref{eq:linearmodel}:
    \begin{align}
        \frac{n(\epsilon)}{\log(n(\epsilon))} = 0.25\epsilon^{-5/2}. \label{eq:nepschoice}
    \end{align}
    The constant prefactor of $0.25$ is chosen as the $O(1)$ term in \eqref{eq:n,epsilon choice} to make the resulting number of points $n(\epsilon)$ fall into the interval $[10^4, 3\cdot10^4]$ for $\epsilon \in [0.023, 0.033]$.  
    \item To measure the effect of changing the sampling density, for each choice of $\epsilon$, $n(\epsilon)$ points are drawn i.i.d. on $\mathbb{S}^1$ using two sampling densities: the uniform density $\rho^{u}$ and a non-uniform density $\rho^{n.u}$ specified in \eqref{eq:rho.n.u.}. 
    The resulting datasets are defined as follows: 
    \begin{align}
      \mathcal{X}^{u}& = \psi\left(\{\theta^{u}_i\}_{i=1}^{n}\right), 
      \quad & 
      \theta^{u}_i &\sim \rho^{u} := \mathcal{U}[0,2\pi], \label{eq:rho.u} \\
     \mathcal{X}^{n.u}& = \psi\left(\{\theta^{n.u}_i\}_{i=1}^{n}\right), \quad & 
     \theta^{n.u}_i &\sim \rho^{n.u} := \pi + 0.1 + \sigma\pi\{\mathcal{N}(0,1)\}.  \label{eq:rho.n.u.}
    \end{align}
    The function $\psi(\theta) = [\cos\theta,\sin\theta]^\top$ is the embedding map $\psi: \mathbb{S}^1\rightarrow\mathbb{R}^2$ and $\{\mathcal{N}(0,1)\}$ denotes the fractional part of the standard normal distribution. The fractional normal distribution is shifted such that the resulting distribution $\rho^{n.u}$ oversamples around the point $\theta = \pi + 0.1$ lying in the transition region and undersamples around the sets $A$ and $B$. The parameter $\sigma$ modulates the variance of the fractional normal distribution. It was set to $\sigma=0.2$.
    The empirical densities of the data are visualized in Figure \ref{fig: 1D sampling densities}.  
    \begin{figure}[h]
    \centering
    \includegraphics[width=0.45\textwidth]{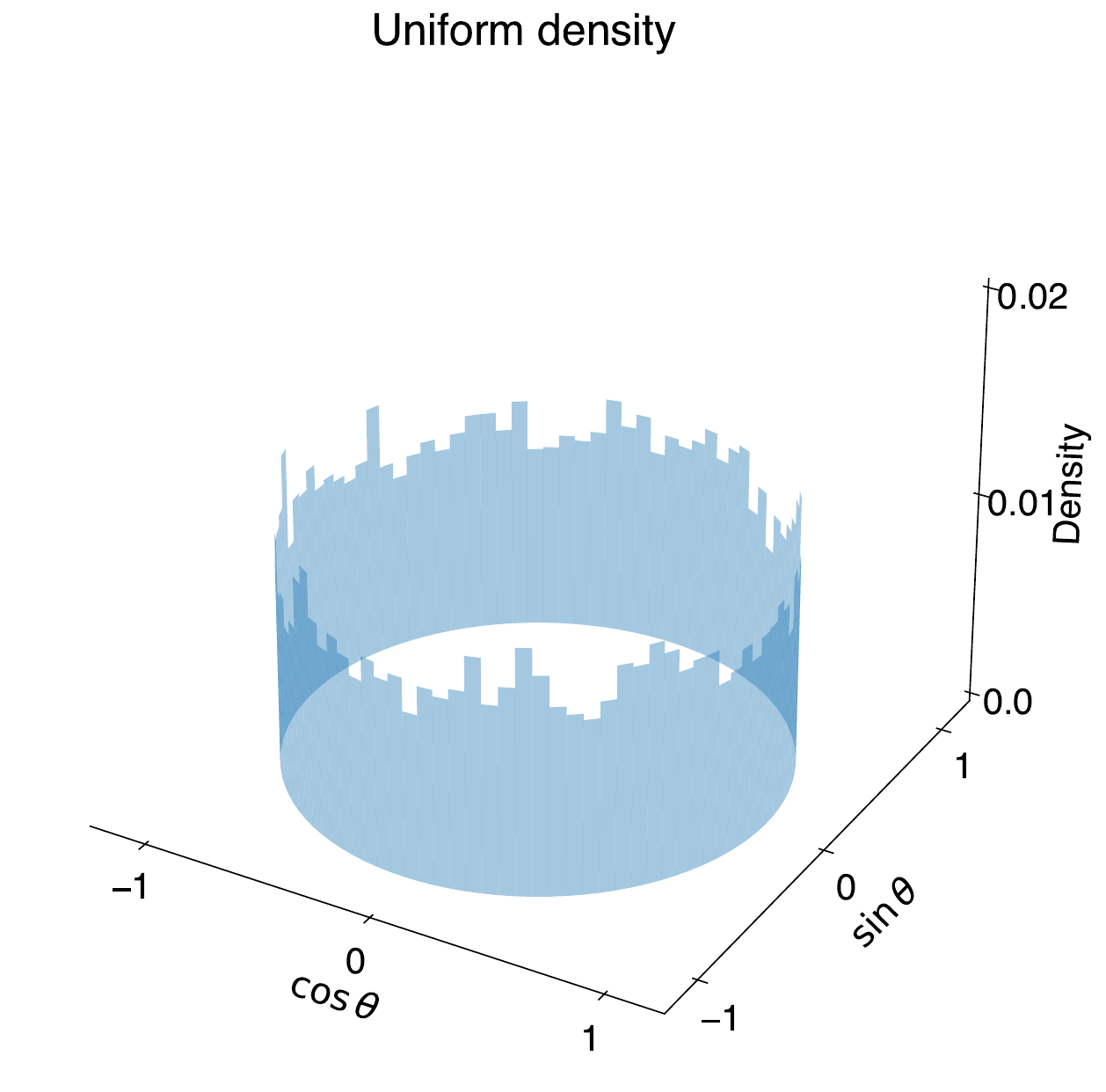}
    \includegraphics[width=0.45\textwidth]{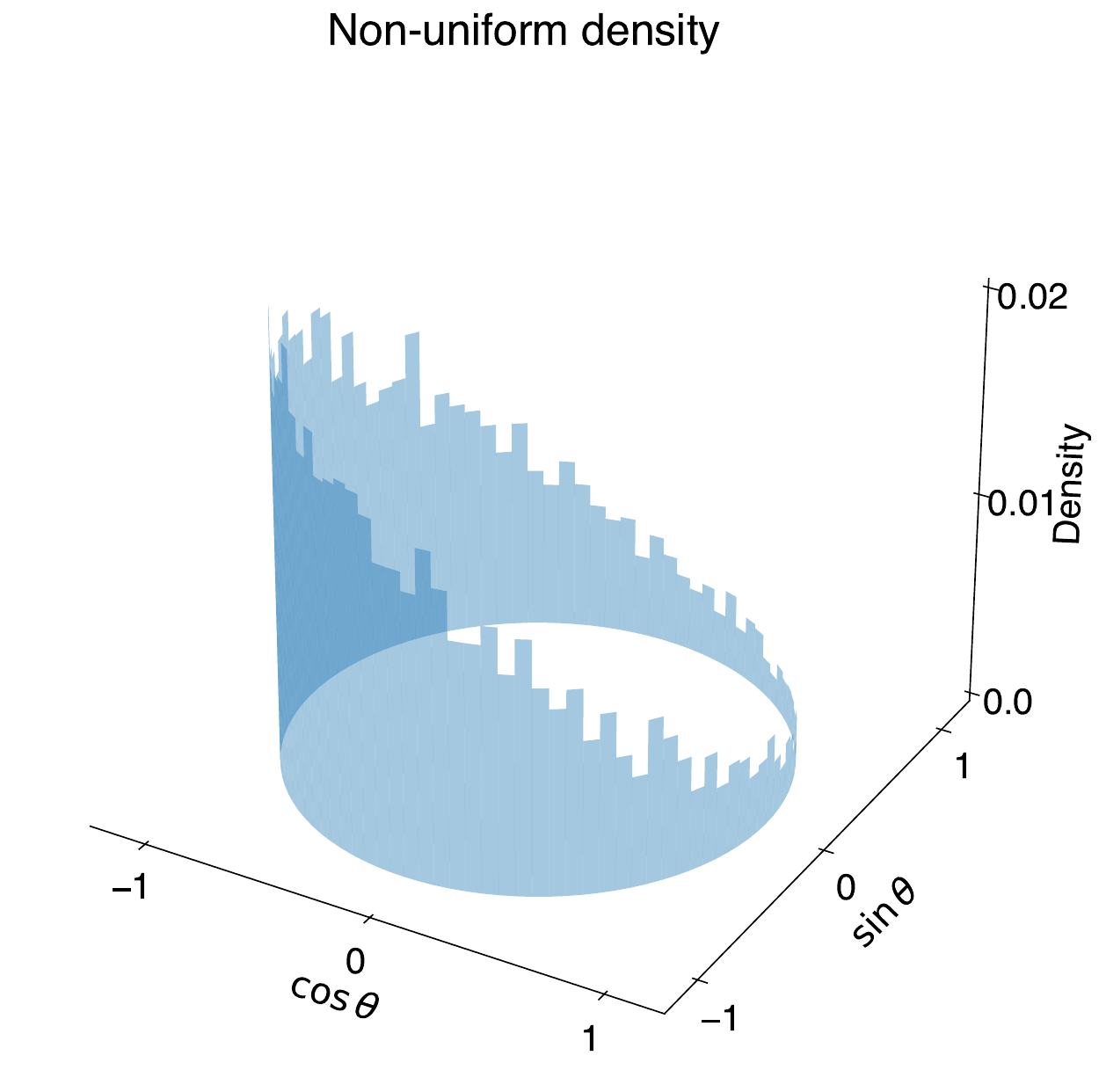}
    \caption{The empirical densities of $\mathcal{X}^{u}$ and $\mathcal{X}^{n.u}$. 
    Note that the nonuniform points $\mathcal{X}^{n.u}$ are clustered in the region around $\theta = \pi + 0.1$.}
    \label{fig: 1D sampling densities}
\end{figure}

    \item To measure the effect of changing the test function, $L^{(n(\epsilon))}_{\epsilon,\mu}f(x)$ was computed for two choices of test function: the committor $q$ given by formula \eqref{qtheta} and 
    $t(\theta) = \sin(\theta)$ (Figure \ref{fig: 1D choices of f}). The evaluation point was fixed to be $\theta = \pi$. 
\begin{figure}[h]
    \centering
    \includegraphics[width=0.430\textwidth]{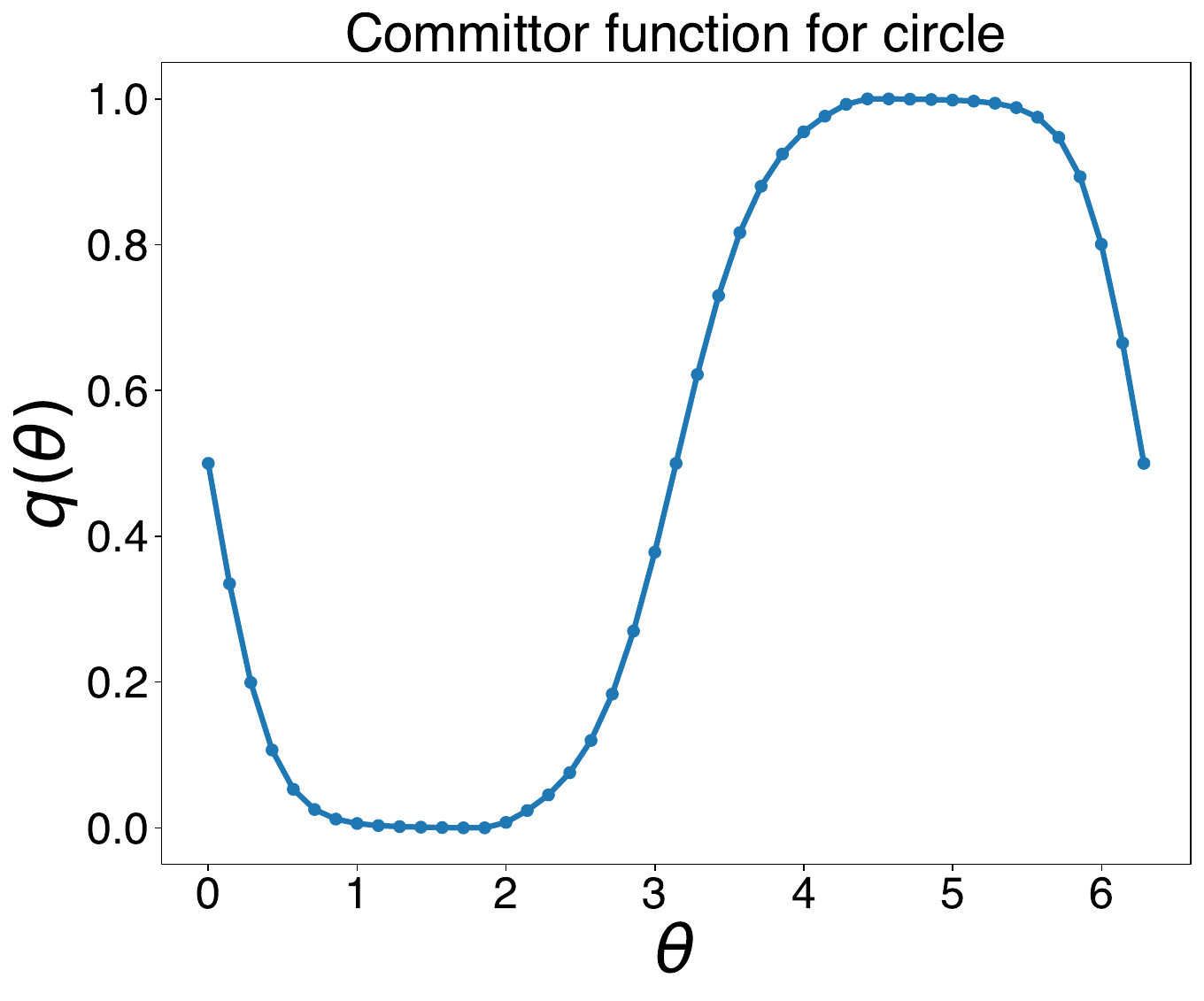}
    \includegraphics[width=0.45\textwidth]{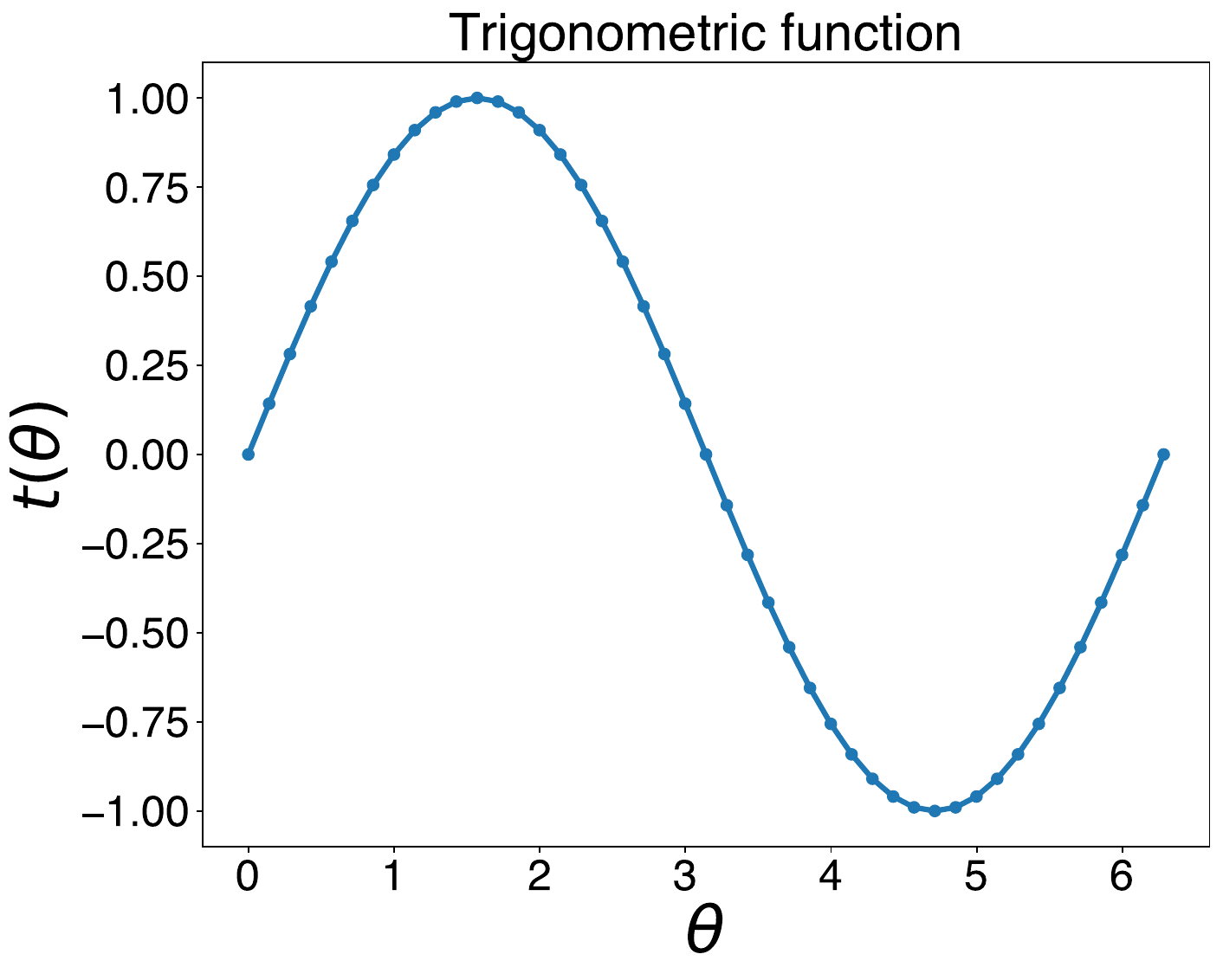}
    \caption{
    The two choices of $f$ for evaluating the consistency error as in~\eqref{eq:linearmodel}: the committor $q$ (left) and the trigonometric polynomial $t(\theta) = \sin \theta$ (right).}
    \label{fig: 1D choices of f}
\end{figure}
    
    \item For a given combination of sampling density and test function, the error in the left-hand side of \eqref{eq:linearmodel} was computed. This error is a function of $\epsilon$ only. 
\end{enumerate}
The above experiment was conducted for 10 equispaced values of $\epsilon$ in the range $[0.023, 0.033]$. Moreover, since $4\beta^{-1}L^{(n(\epsilon))}_{\epsilon,\mu}f(x) - \mathcal{L}f(x)$ is random due to the i.i.d. draws from $\rho$, the experiment was repeated 50 times for every value of $\epsilon$ to mitigate the statistical variation. Then a linear fit  to \eqref{eq:linearmodel} was obtained to this data as a function of $\epsilon$ and the slope of this fit was taken as the estimate of the prefactor of the bias error for the given combination of $f$ and $\rho$. A total of four linear fits were produced.  The results are presented in Figure \ref{fig: 1D bias prefactor estimaton of f}. 

\subsubsection{Summary of numerical results for reduction in bias error} 
We computed the error $4\beta^{-1}L^{n(\epsilon)}_{\epsilon,\mu}f(x) - \mathcal{L}f(x)$ for $f = q,t$ and $\rho = \rho^{u}, \rho^{n.u}$ as a function of $\epsilon$ ranging in $[0.023,0.033]$. In Figure \ref{fig: 1D bias prefactor estimaton of f}, a linear fit is obtained for $4\beta^{-1}L^{n(\epsilon)}_{\epsilon,\mu}f(x) - \mathcal{L}f(x)$ for each choice of $f$ and $\rho$, where the slope of the linear regression is then an approximation of the bias error prefactor involved. Note that the bias error prefactor in general may be negative since the bias error formula in Theorem \ref{thm: bias error} does not require taking absolute values. Here, it is of interest to quantify the prefactor of lowest \emph{magnitude} to measure which combination of sampling density and test function gives faster convergence of the bias error. The estimates of the prefactor magnitude are presented in Table \ref{tab: 1D bias prefactor estimaton of f}. We find that the bias error prefactor of the lowest magnitude occurs when using a uniform density as input to TMDmap and applying the resulting generator to the committor function, presumably due to the uniform sampling and the committor function canceling the bias prefactors $\mathcal{B}_2$ and $\mathcal{B}_3$ at the given point $x$.

\begin{table}[h]
\centering
\begin{tabular}{|l|llll|}
\hline
                          & \multicolumn{4}{c|}{Sampling density, Test function}                                                                                                  \\ \cline{2-5} 
                          & \multicolumn{1}{l|}{$\rho^{n.u}$,
                           $t(\theta)$.} & \multicolumn{1}{l|}{$\rho^u$, 
                           $t(\theta)$} & \multicolumn{1}{l|}{$\rho^{n.u}$, 
                           $q$} & 
                           $\rho^u$, 
                          $q$     \\ \hline
Abs. bias error prefactor & \multicolumn{1}{l|}{1.024}                     & \multicolumn{1}{l|}{0.778}                & \multicolumn{1}{l|}{1.148}            & $\textbf{0.398}$ \\ \hline
\end{tabular}
\caption{The bias error prefactor $b$ in \eqref{eq:linearmodel} is estimated for two choices of $\rho$ and two choices of $f$ for the 1D system described in Section \ref{sec:unitcircle}. The overall prefactor of lowest magnitude is obtained when using a uniform sampling density and applying the TMDmap generator to the committor function. This fact can be explained by the cancellations of the prefactors $\mathcal{B}_2$ and $\mathcal{B}_3$ in Theorem \ref{thm: bias error}.} 
\label{tab: 1D bias prefactor estimaton of f}
\end{table}

\begin{figure}[h]
    \centering
    \includegraphics[width=0.7\textwidth]{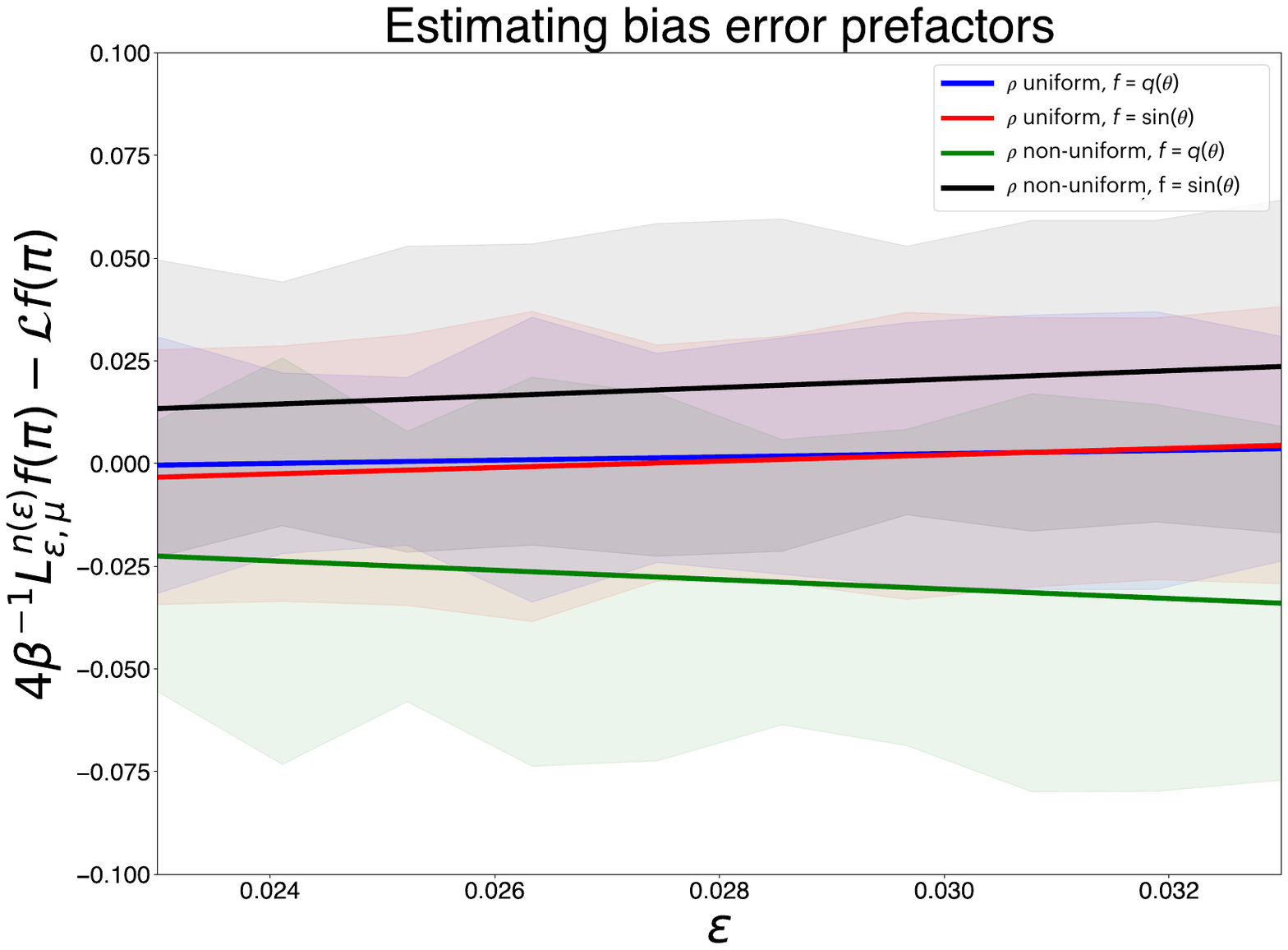}
    \caption{The consistency error for the committor, $f = q(\theta)$, and for $f = \sin(\theta)$ at $\theta = \pi$ as a function of $\epsilon$. The number of points $n(\epsilon)$ for each $\epsilon$ is chosen according to \eqref{eq:nepschoice} to keep the variance error constant. The datasets we sampled according to the uniform and the nonuniform densities in \eqref{eq:rho.u} and \eqref{eq:rho.n.u.} respectively.}
    \label{fig: 1D bias prefactor estimaton of f}
\end{figure}

\subsection{Calculating the committor function}
\label{subsec: calc committor}
The goal of this section is to compute the committor as the solution to  \eqref{eq: tmdmap committor bvp} using TMDmap. The main focus will be on studying the discrepancy between the numerical solution $q^{{\sf TMD}}$ and the true solution $q$ as a function of the sampling density $\rho$ and kernel bandwidth $\epsilon$. 

\subsubsection{Theoretical considerations}
{ Corollary \ref{cor: committor error estimate} suggests that solving the committor problem numerically with a uniform sampling density $\rho$ will likely produce improvements in the accuracy of the numerical solution. Furthermore, bandwidth $\epsilon$ needs to be commensurate with the fixed number of points $n$ to ensure that the error formula in Theorem \ref{thm:main} holds.}

\subsubsection{Optimal bandwidth} Given fixed $n, \rho, \mu, $ and $x$, the expression in \eqref{eq: summarized error model} may be minimized for $\epsilon$ to obtain an optimal bandwidth $\epsilon^*$ given by 
     \begin{align}
         \epsilon^* \sim (1+d/4)^{\frac{1}{2+d/4}}\left(\frac{\log n}{n}\right)^{\frac{1}{4 + d/2}}\left|\frac{V}{\mathcal{B}_{1} + \mathcal{B}_2 + \mathcal{B}_{3}}\right|^{\frac{1}{2 + d/4}}.\label{eq: optimal bandwidth}
     \end{align}
     A similar version of the optimal bandwidth formula was derived in \cite{singer2006graph} for Laplacian eigenmaps. {That} work, however, stopped short of expanding further on this formula due to the complicated dependence of the optimal bandwidth on manifold-related parameters. { In contrast, this paper offers} a full enumeration of these parameters given by $\mathcal{V},$ $ \mathcal{B}_1,$ $ \mathcal{B}_2$ and $\mathcal{B}_3$. Moreover, the pointwise error in the left-hand side of \eqref{eq: summarized error model} splits into two regimes locally around $\epsilon^*$.
     \begin{enumerate}
         \item If $\epsilon \leq \epsilon^*$, the error is dominated by the divergence of the variance error term $\epsilon^{-(1+d/4)}$. 
         \item If $\epsilon \geq \epsilon^*$, the error is dominated by the linear bias error $\epsilon \mathcal{B}$. A reduction in the magnitude of $\mathcal{B}$ will make the error diverge more slowly from the minimum at $\epsilon^*$. Consequently, a smaller bias error prefactor will make the error more stable to small perturbations around $\epsilon^*$. 
     \end{enumerate}
     
It is difficult to explicitly compute prefactors in Theorem \ref{thm:main}. Hence it is not advisable to select the bandwidth using \eqref{eq: optimal bandwidth}. The common practice is to use a heuristic estimate for $\epsilon^\ast$ instead. We use the \emph{kernel sum}  or the \emph{Ksum} test \cite{singer2009detecting, berry2016variable, evans2022computing}: \begin{align}
    \epsilon^{*}_{\sf Ksum}  =  \underset{\epsilon}{\text{argmax}} \frac{\partial \log S(\epsilon)}{\partial \log \epsilon},\quad S(\epsilon) := \sum_{ij}[K^{(n)}_{\epsilon}]_{ij}. \label{eq: ksum test}
\end{align}

To implement the Ksum test numerically, the sum $S(\epsilon)$ of all entries of the kernel matrix $[K^{(n)}_{\epsilon}]_{ij} =\exp(-\|x_i-x_j\|^2/\epsilon)$ is computed at a range of $\epsilon$ values and its logarithm is plotted against $\log\epsilon$. The value of $\epsilon$ for which the slope of this plot is the largest is taken to be $\epsilon^{*}_{\sf Ksum}$. Therefore, at $\epsilon = \epsilon^{*}_{\sf Ksum}$, the kernel matrix is the most sensitive to bandwidth for the entire input point cloud. Although this test is not exact, it gives a good initial guess for a more exhaustive search for the optimal $\epsilon$. We highlight $ \epsilon^{*}_{\sf \sf Ksum}$ in Figures \ref{fig: muller comparison}, \ref{fig: muller metadynamics error}, \ref{fig: twowell comparison}, and \ref{fig: twowell metadynamics error}. The RMS error is flatter around $\epsilon^*$ and is nearly optimal at  $\epsilon^{*}_{\sf Ksum}$ in all our test problems when we use quasi-uniform sampling densities.

\subsubsection{Optimal sampling} The preceding discussion on choosing optimal $\epsilon$ reveals that the sensitivity of the error to small perturbations around $\epsilon^*$ can be ameliorated through reductions in the bias error prefactors. A quasi-uniform sampling density can be used to reduce $\mathcal{B}_2$ and $\mathcal{B}_1$, thus necessitating increased sampling of rare high-energy configurations lying in the transition region away from the metastable states. There are numerous ways of such rare event sampling: some examples relevant to molecular simulation are temperature acceleration \cite{maragliano2006temperature}, importance sampling \cite{botev2013markov}, umbrella sampling~\cite{kastner2011umbrella}, splitting methods~\cite{webber2020splitting},  as well as recent work utilizing generative models such as 
normalizing flows \cite{falkner2022conditioning}, 
Generative Adversarial Networks (GANs) or diffusion models \cite{sohl2015deep,goodfellow2020generative}. See~\cite{henin2022} for a recent comprehensive survey on enhanced sampling in molecular dynamics. In the present case, the situation is complicated because the low-dimensional manifold $\mathcal{M}$ where $\rho$ is supported is unknown and can only be accessed through sampling.  Consequently, in this setting, metadynamics \cite{laio2002escaping} emerges as a cheap and robust option to generate samples on $\mathcal{M}$. These samples can then be postprocessed to $\delta$-nets to enhance spatial uniformity. 

Metadynamics can be described as a modification of the Euler-Maruyama scheme of discretizing overdamped Langevin dynamics \eqref{eq: old}. The goal of metadynamics is to bias the system towards exploring regions of the energy landscape that are not sufficiently sampled under the original potential \(V(x)\). To do so, the metadynamics algorithm introduces a bias potential \(W(x,t)\) that is added to the original potential \(V(x)\) at each step of the Euler-Maruyama algorithm. The bias potential is constructed as a sum of Gaussian potentials deposited during the simulation time \(t\):
\begin{equation}
\label{metad}
{
W(x,t) = \sum_{t_j < t} w_0 \, \exp\left(-\frac{({\theta(x) - \theta(x(t_j)}))^2}{2\sigma^2}\right),
}
\end{equation}
where $\theta(\cdot)$ is a vector of collective variables used for biasing the trajectory,  $x(t_j)$ is the set of atomic positions at time \(t_j\), \(w_0\) is the height of each deposited Gaussian potential and \(\sigma\) controls the width of the Gaussians. The bias potential \(W(x,t)\) modifies the original potential energy, leading to an effective potential energy \(V_{\sf eff}(x, t) = V(x) + W(x,t)\). 
{ Thus, the metadynamics fills up the free energy wells with Gaussian bumps enabling the system to escape from metastable states and explore the configurational space. It is summarized in Algorithm \ref{alg:metad}.}

\begin{algorithm}
  \caption{Metadynamics}
  \SetAlgoLined
  \SetKwInOut{Input}{Input}
  \SetKwInOut{Output}{Output}

  \Input{Gaussian height $w_0$, Gaussian width $\sigma_0$, deposition stride $s$, maximum number of steps $n$, timestep $\Delta t$. 
  }
  \Output{Visited states trajectory $\mathcal{X}(n) = \{x_1, \ldots x_n\}$}

  Initialize the visited states trajectory: $\mathcal{X}(n) \gets [\,]$\;
  Initialize the history of visited states: $H \gets \emptyset$\;
  Initialize the potential: $V_{\sf eff} \gets V$\;
  \For{$i \gets 1$ \KwTo $N_{\text{steps}}$}{
    Sample a new state: $x_i \gets $ Euler-Maruyama step of size $\Delta t$ with potential $V_{\sf eff}$\;
    Append the current state to the trajectory: $\mathcal{X}(n) \gets \mathcal{X}(n) \cup \{x_i\}$\;

    \If{$i \bmod s = 0$}{
      Deposit Gaussian: $V_{\sf eff} \gets V_{\sf eff}  + w_0 \exp\left(-\frac{(\theta(\cdot) -\theta(x_i))^2}{2\sigma^2}\right)$}
  }
  \label{alg:metad}
\end{algorithm}

The $\delta$-net algorithm outlined in Algorithm \ref{alg:deltanet} is a greedy procedure used to select a sample of points from the dataset $\mathcal{X}(n)$ such that any pair of points is at least a distance of $\delta$ apart. The output point cloud tends to be spatially uniform and approximates samples from the uniform density on $\mathcal{M}$. 

\begin{algorithm}
  \caption{$\delta$-net Algorithm}
  \SetAlgoLined
  \SetKwInOut{Input}{Input}
  \SetKwInOut{Output}{Output}

  \Input{Set of points $P$, threshold distance $\delta$}
  \Output{$\delta$-net $D \subseteq P$}

  Initialize an empty set $D \gets \emptyset$\;
  
  \ForEach{$p \in P$}{
    \If{$\text{no point in } D \text{ is within distance } \delta \text{ from } p$}{
      Add $p$ to the $\delta$-net: $D \gets D \cup \{p\}$\;
    }
  }

  \KwRet{$D$}\;
  \label{alg:deltanet}
\end{algorithm}

\subsubsection{Experimental setup.}\label{subsec:experimental} To understand the effect of sampling density $\rho$ and the bandwidth parameter $\epsilon$ on the quality of the numerical committor $q^{{\sf TMD}}$ we will examine the root mean squared error (RMSE): 
\begin{align}
    \text{RMSE}\: |q - q^{{\sf TMD}}| = \left(\sum_{i=1}^{n}|q(x_i) - q^{{\sf TMD}}(x_i)|^2\right)^{1/2}. \label{eq: rmse}
\end{align}
To investigate the effect of $\rho$, the underlying point cloud $\mathcal{X}(n)$ for computing $q^{{\sf TMD}}$ via the TMDmap algorithm will be drawn from three different sampling densities via the following sampling algorithms.
\begin{enumerate}
    \item The {\bf Euler-Maruyama} sampling ~(see e.g. \cite{kloeden1992stochastic}) leads to the data being sampled essentially through the invariant {\bf Gibbs density} $\mu \propto \exp{(-\beta V(x))}$. 
    \item {\bf Metadynamics} outlined in Algorithm \ref{alg:metad} results in a more uniform distribution of data on $\mathcal{M}$. 
    \item {\bf Metadynamics with $\delta$-net} first generates data through Algorithm \ref{alg:metad} and then  post-processes them through Algorithm \ref{alg:deltanet}. We will illustrate that such configurations tend to be spatially quasi-uniform on $\mathcal{M}$. 
    \end{enumerate}
For each choice of sampling density, we will draw a fixed number of points and feed the resulting point cloud as input to the TMDmap algorithm. To study the effect of $\epsilon$, we will then vary $\epsilon$ and compute the RMSE as a function of $\epsilon$ for two test systems governed by the overdamped Langevin dynamics \eqref{eq: old} with 
M\"{u}ller's potential and a two-well potential in $\mathbb{R}^2$.
{Since both of these examples are 2D, the function $\theta(\cdot)$ in \eqref{metad} is chosen to be the identity map.}

\subsubsection{M\"{u}ller's potential} We consider a system governed by \eqref{eq: old} 
with $V$ being M\"{u}ller's potential \cite{muller1979location}:
\begin{equation}
\label{mueller}
    V(x_1, x_2) = \sum\limits_{i=1}^4 D_i \exp\left\{a_i(x_1 - X_i)^2 + b_i(x_1 - X_i)(x_2 - Y_i) + c_i(x_2 - Y_i)^2 \right\}
\end{equation}
where
\begin{align*}
    [a_1, a_2, a_3, a_4] &= [-1, -1, -6.5, 0.7], \\
    [b_1, b_2, b_3, b_4] &= [0, 0, 11, 0.6], \\
    [c_1, c_2, c_3, c_4] &= [-10, -10, -6.5, 0.7], \\
    [D_1, D_2, D_3, D_4] &= [-200, -100, -170, 15], \\
    [X_1, X_2, X_3, X_4] &= [1, 0, -0.5, -1], \\
    [Y_1, Y_2, Y_3, Y_4] &= [0, 0.5, 1.5, 1].
\end{align*}

We computed $q^{{\sf TMD}}$ with the three types of input density specified in Section \ref{subsec:experimental}. The timestep for the Euler-Maruyama sampling was $\Delta t = 10^{-4}$, and the trajectory length was $10^6$ timesteps. The trajectory was subsampled to keep every $100$th point resulting in a point cloud of size $10^4$. In the metadynamics approach, a Gaussian bump with { $\sigma = 0.1$} and $w_0 = 0.5$ was deposited at every \(100\)th timestep. The dataset was then post-processed to a $\delta$-net. 
The reactant set $A$ and product set $B$ were chosen to be the balls of radius $0.1$ centered at $[-0.558, 1.441]$ and $[0.623, 0.028]$ respectively.
To validate the results, the numerical solution to the committor problem was computed using the finite element method (FEM). The computational domain for the FEM was $\{x\in\mathbb{R}^2 ~|~V(x) \le 10\}$. The homogeneous Neumann boundary condition was imposed on the outer boundary  where $V =  10$ (Figure \ref{fig: muller's potential}). This FEM solution was used as a surrogate for the true solution in computing the RMSE in \eqref{eq: rmse}. 

\begin{figure}[h]
\centering
    \includegraphics[width=0.65\textwidth]{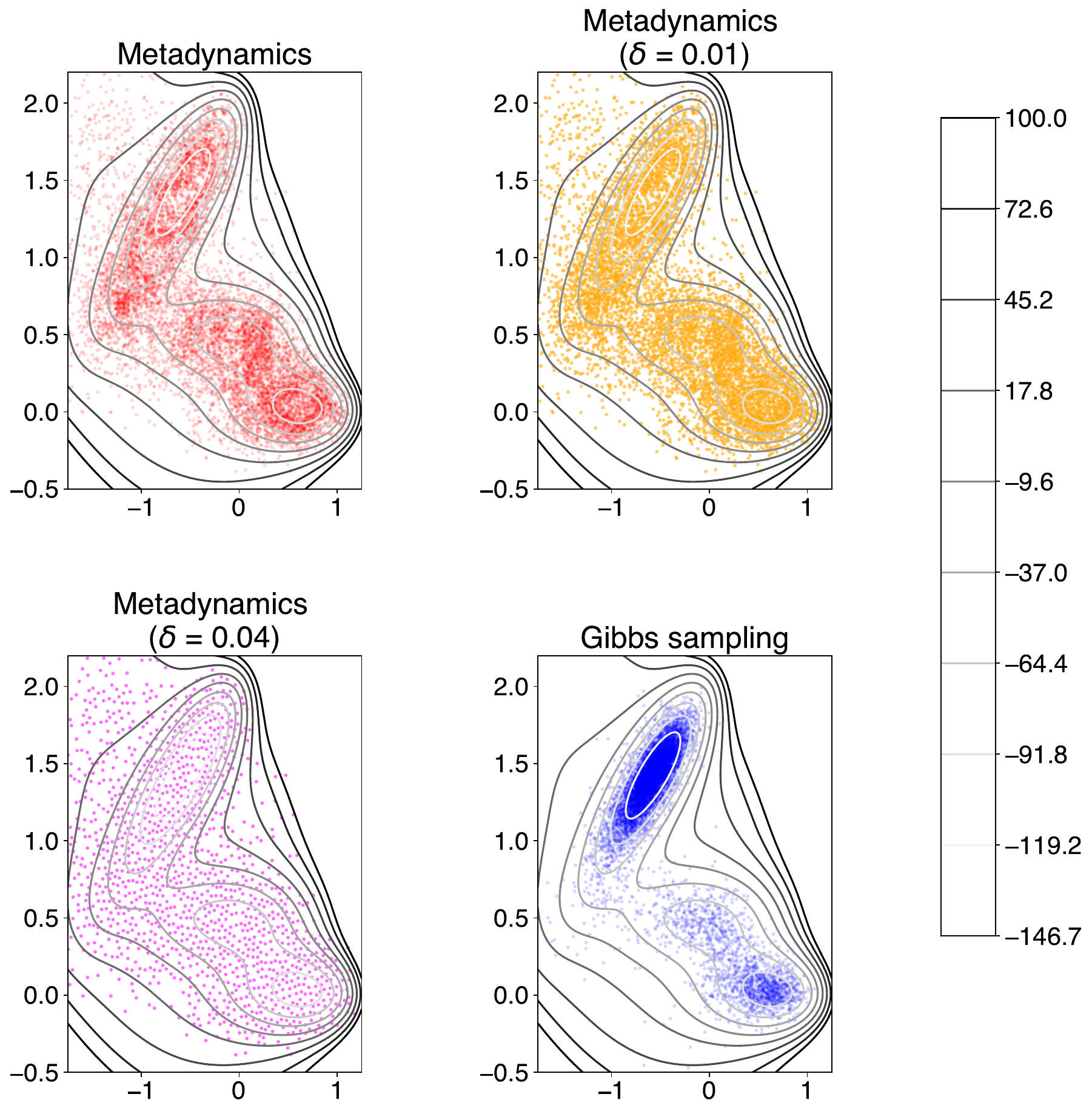}
    \caption{Landscape of M\"{u}ller's potential with different sampling densities}
    \label{fig: muller's potential}
\end{figure}

Since it is of interest to study the effect of uniformization on the RMSE, we modulate $\delta$ in the $\delta$-net algorithm used for post-processing the metadynamics dataset containing $10^4$ points. For each $\delta$, the RMSE over a varying $\epsilon$ is computed and presented in Figure \ref{fig: muller metadynamics error}. The $\delta$-net postprocessing tends to result in a flatter error curve around the minimum hence making the numerical solution more robust with respect to the choice of the bandwidth.

The TMDmap solutions obtained using datasets generated with the three sampling types and as well as the FEM mesh consisting of nearly equilateral triangles of nearly equal sizes are compared in Figure \ref{fig: muller comparison}.
The error curve for the metadynamics plus $\delta$-net dataset has a flat region around its minimum and achieves almost as small a minimum as the one for the FEM mesh point cloud. 

\begin{figure}[h!]
    \centering
    \includegraphics[width=0.6\textwidth]{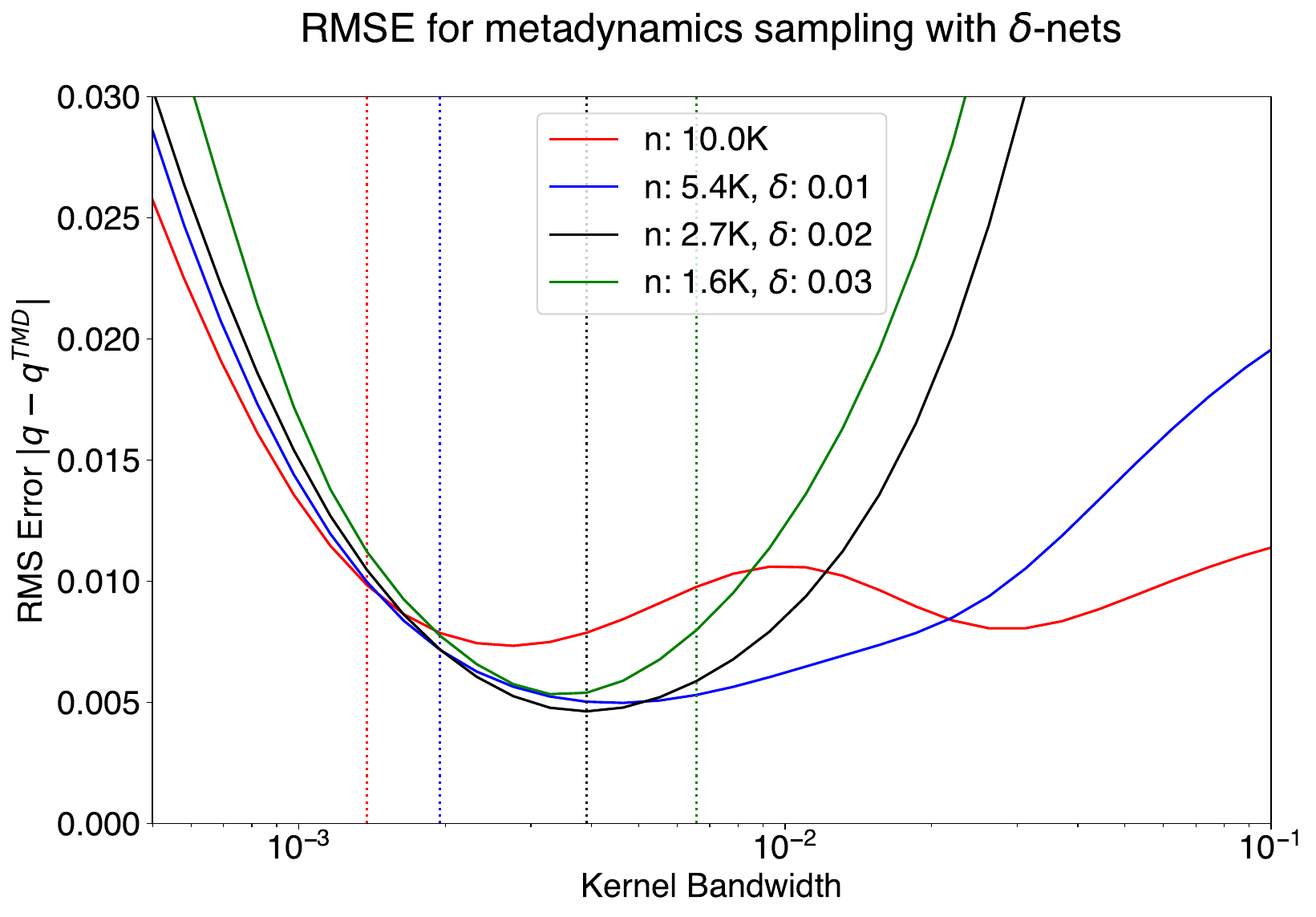}
    \caption{We track the RMSE for $q^{\sf TMD}$ computed for M\"uller's potential as a function of $\epsilon$ for different degrees of spatial uniformity. The dotted line marks the $\epsilon^{*}_{Ksum}$ computed using \eqref{eq: ksum test}.}
    \label{fig: muller metadynamics error}
\end{figure}

\begin{figure}[h!]
    \centering
    \includegraphics[width=0.7\textwidth]{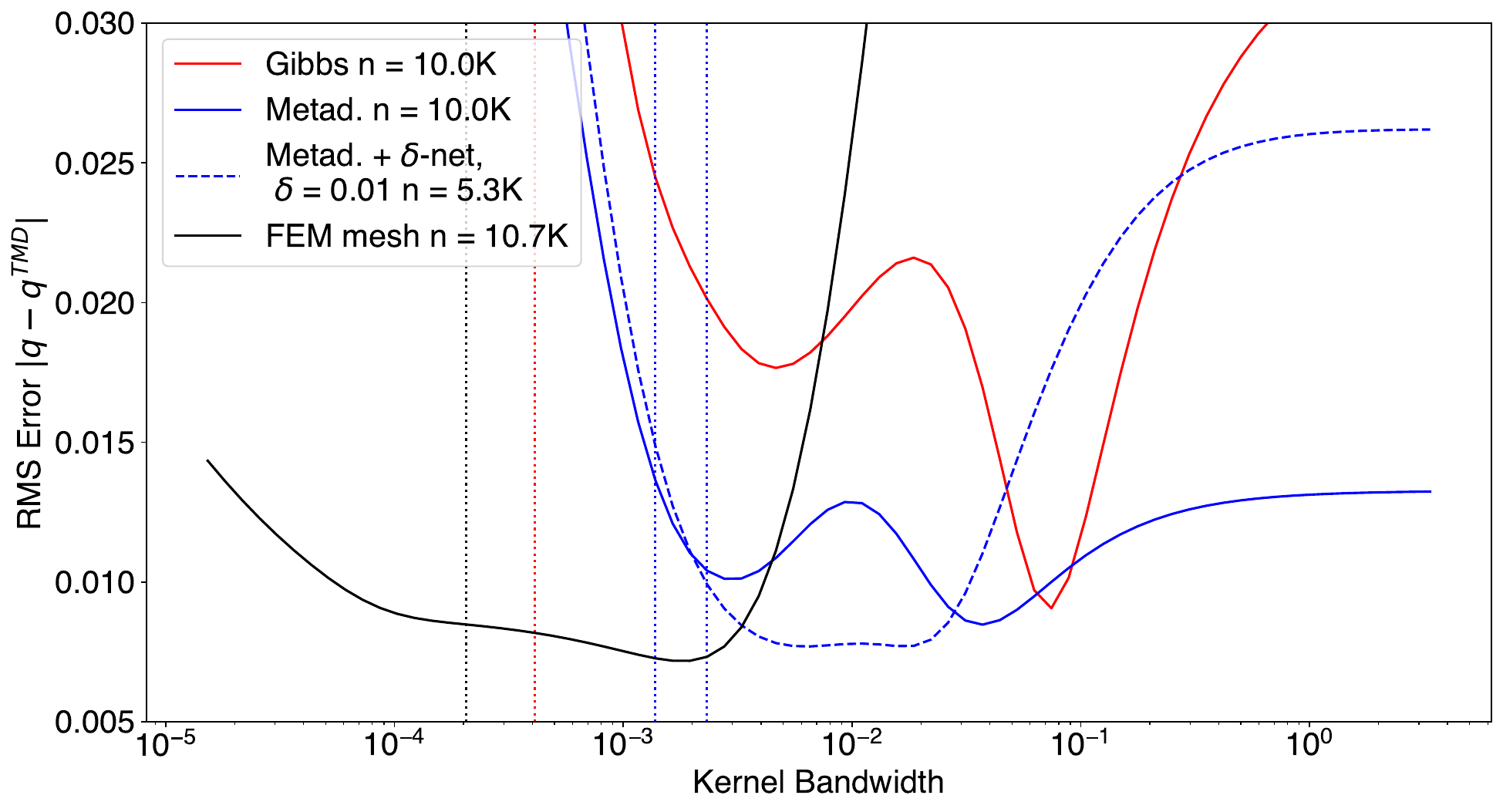}
    \caption{Tracking the RMSE for $q^{\sf TMD}$ in M\"uller's potential for different sampling densities as a function of $\epsilon$. Note that using the mesh from the FEM algorithm gives the least error.}
    \label{fig: muller comparison}
\end{figure}


\subsubsection{Two-well potential in 2D} 
A 2D test system governed by \eqref{eq: old} with a two-well potential 
\begin{equation}
\label{eq:two-well}
    V(x_1, x_2) = (x^{2}_1 - 1)^{2}  
\end{equation}
at $\beta = 1.0$ magnifies the advantage of using point clouds generated by metadynamics plus $\delta$-net even stronger. The sets $A$ and $B$ are balls of radius $0.1$ centered at $[-1,0]$ and $[1,0]$ respectively. {The datasets are sampled from the Gibbs density, metadynamics, and metadynamics with $\delta$-net with $\delta = 0.1$, $0.02$, $0.03$, and $0.04$. Four of these datasets are displayed in Figure \ref{fig: two well}.} 

\begin{figure}[h!]
    \centering
    \includegraphics[width=0.8\textwidth]{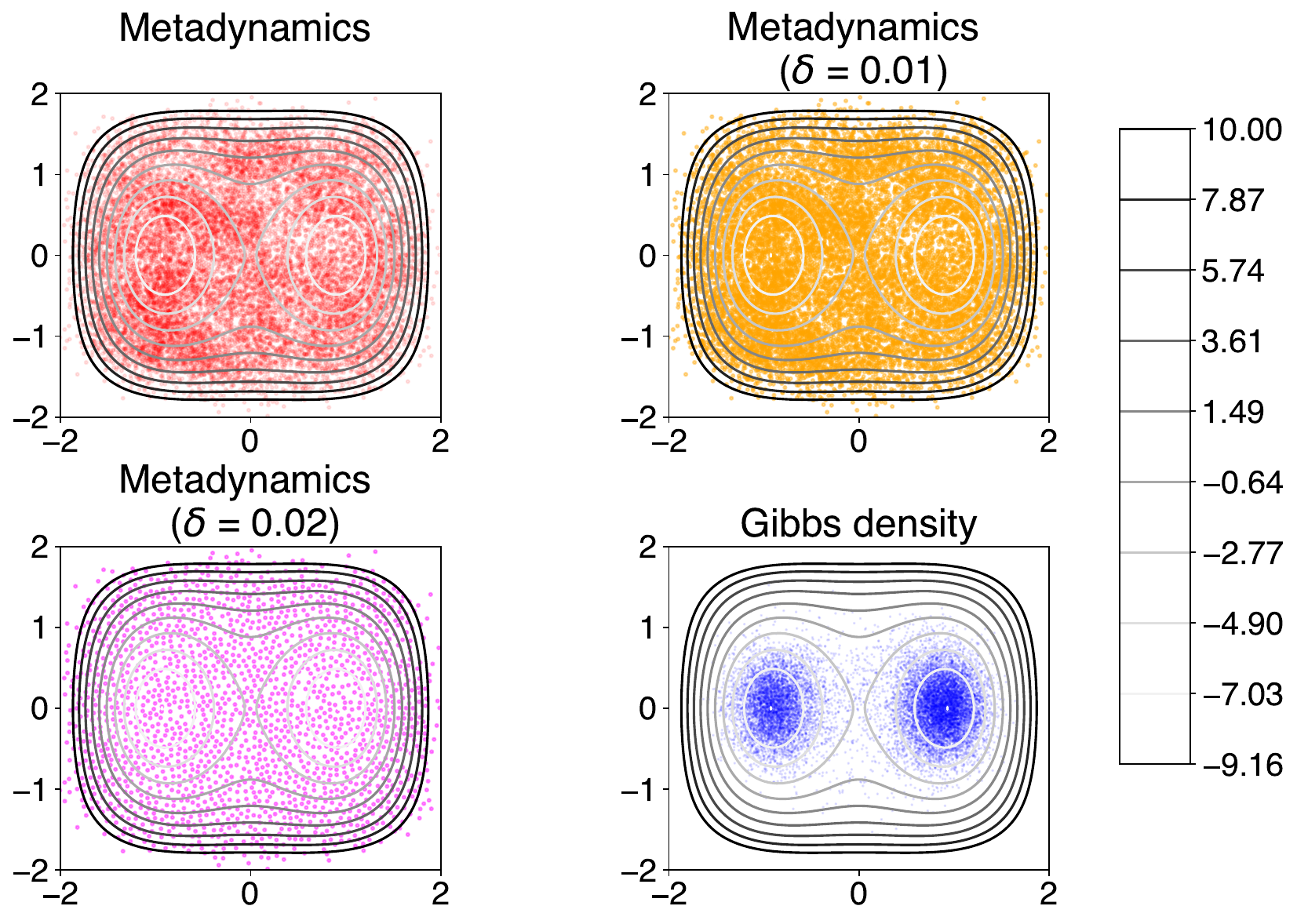}
    \caption{The two-well potential landscape with different sampling densities}
    \label{fig: two well}
\end{figure}

In Figure \ref{fig: twowell metadynamics error}, the effect of $\delta$-net is quantified by tracking the RMSE as a function of $\epsilon$ for each value of $\delta$. We find a similar pattern as for the test problem with M\"{u}ller's potential as in Figure \ref{fig: muller metadynamics error} that moderate levels of uniformization (at $\delta = 0.02$ in the case of $V$ given by \eqref{eq:two-well}) improves and stabilizes the RMSE at the optimal $\epsilon$. 

\begin{figure}[h!]
    \centering \includegraphics[width=0.8\textwidth]{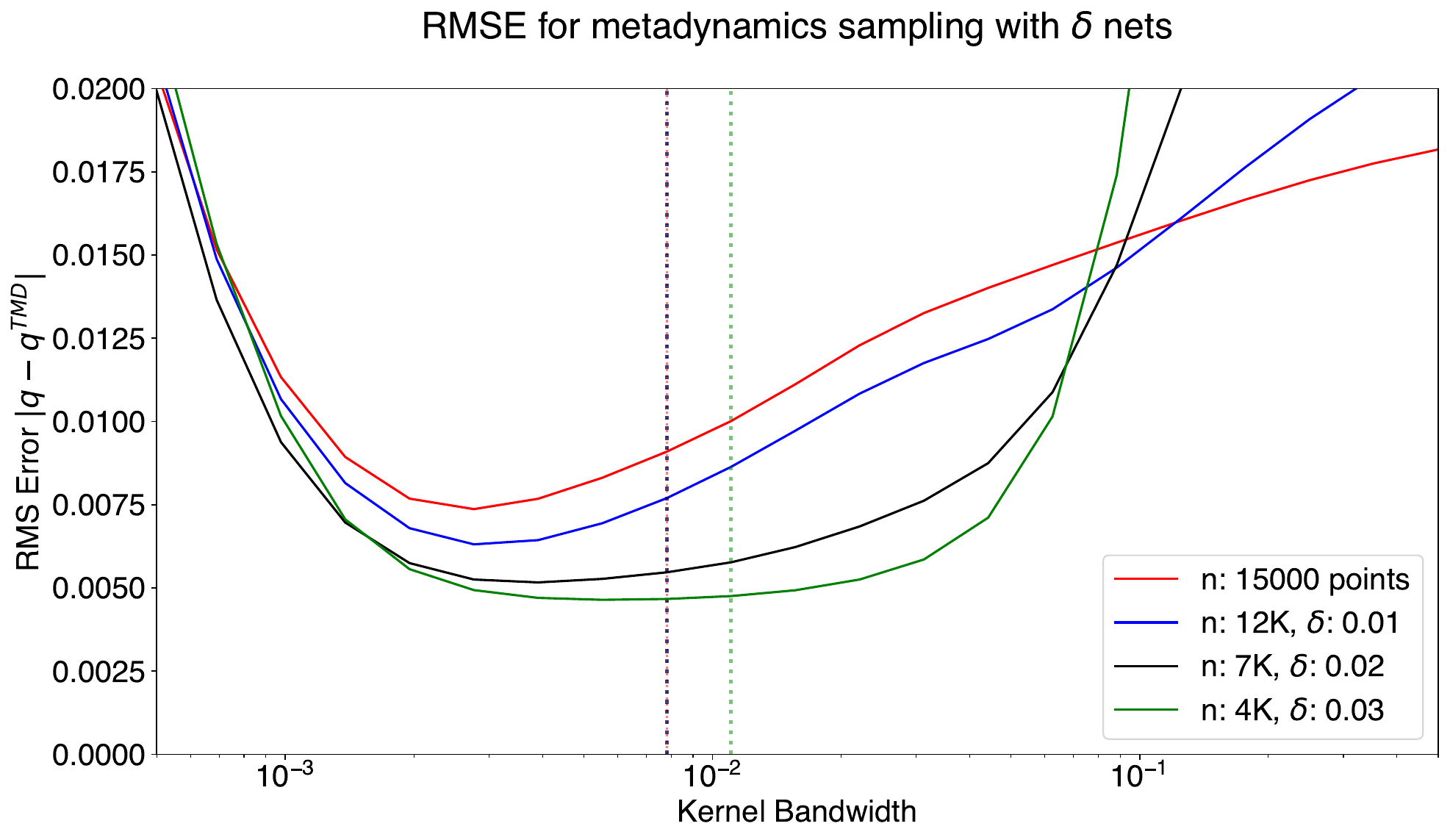}
    \caption{The RMSE for $q^{{\sf TMD}}$ the test system with the two-well potential \eqref{eq:two-well}  obtained using datasets generated via metadynamics and metadynamics plus $\delta$-net with $\delta=0.01$, $0.02$, and $0.03$.     }
    \label{fig: twowell metadynamics error}
\end{figure}

In Figure \ref{fig: twowell comparison}, the three choices of sampling density are compared. Adding uniformization using $\delta$-nets produces the same pronounced effect as for M\"uller's potential in Figure \ref{fig: muller comparison}. Additionally, since this is a 2D example, it is also possible to provide to TMDmap the set of points from the nearly regular FEM mesh used for producing the reference solution. The TMDmap committor $q^{TMD}$ obtained using the FEM mesh dataset has a notably smaller error for a broad range of $\epsilon$ values than it is for the other point clouds used here.
This reinforces the hypothesis that uniform sets are error-optimal for TMDmaps. The caveat, however, is that, in higher dimensions, such uniform meshing is infeasible to generate and thus practitioners must resort to randomly sampled data on $\mathcal{M}$. 

\begin{figure}[h!]
    \centering \includegraphics[width=0.8\textwidth]{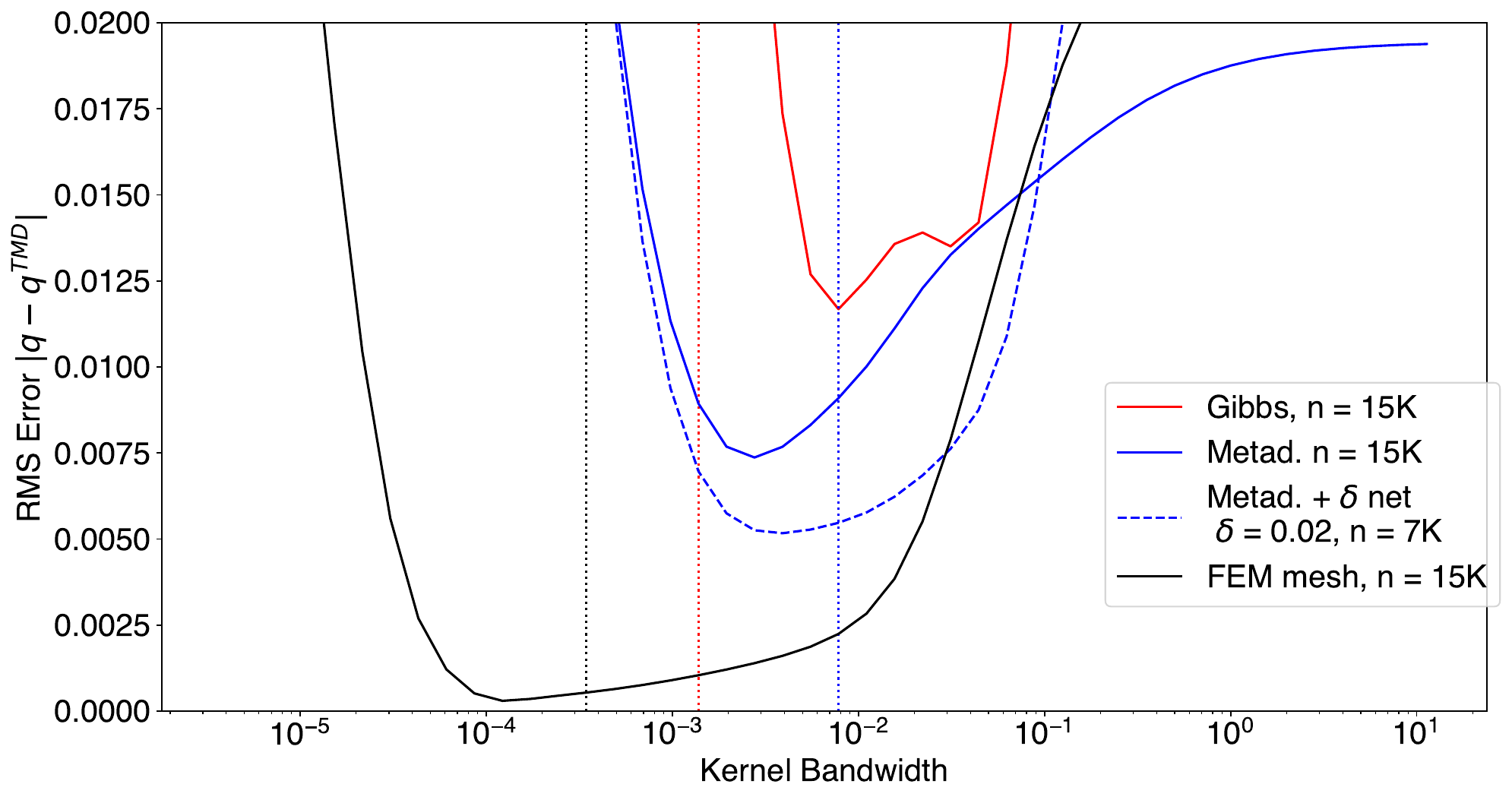}
    \caption{Tracking the RMSE for different sampling densities as a function of $\epsilon$. Note that using the mesh from the FEM algorithm gives the least error.}
    \label{fig: twowell comparison}
\end{figure}

\clearpage
\subsubsection{Summary of numerical results} Our numerical experiments in Figures \ref{fig: muller metadynamics error}, \ref{fig: muller comparison}, \ref{fig: twowell metadynamics error}, and \ref{fig: twowell comparison} confirm that: 

\begin{enumerate}
    \item Spatial uniformization improves not only the accuracy of the RMSE $|q - q^{{\sf TMD}}|$ at $\epsilon^*$ but also the \emph{stability}. Our error estimate suggests that this results from the faster bias error convergence due to quasi-uniform $\rho$. 
    \item 
    Metadynamics with $\delta$-net is a viable strategy to realize the faster bias error convergence rate. 
    \item If the original data $\mathcal{X}$ has good coverage of the manifold $\mathcal{M}$ 
    then \emph{deleting} some points can \emph{improve} the RMSE as long as those deletions improve the spatial uniformity of the data. In other words, reducing the number of data points can enhance the accuracy of the solution!
\end{enumerate}

\section{Conclusion}
In this paper, we have derived {sharp error estimates revealing} the precise error rates for the target measure diffusion map. Our results extend the consistency theory for graph Laplacians to a manifold learning setting involving density reweighting. The { main advantage} of incorporating the TMDmap density reweighting is the free choice of the sampling density, enabling the use of arbitrary enhanced sampling algorithms for generating the input data. 

We have provided a principled approach for tuning the sampling density $\rho$ and kernel bandwidth parameter $\epsilon$ to improve the accuracy of the TMDmap algorithm because our results are \emph{asymptotically sharp}. The obtained error formula contains a complete quantification of all prefactors involved in the bias and variance errors. This allows us to find strategies for reducing the approximation error when discretizing the backward Kolmogorov operator $\mathcal{L}$ or when solving a Dirichlet boundary-value problem (BVP) with this discretization. These formulas illuminate that (a) solving a homogeneous BVP such as the committor problem and/or (b) sampling $\mathcal{M}$ with a uniform density are the settings in which bias and variance errors are attenuated. Importantly, we have demonstrated that these improvements in accuracy are attainable in practice.

A significant aspect of our work has also been the exploration of uniform subsampling via $\delta$-nets as a robust and simple technique for reducing the error in the TMDmap approximation to the committor. The consistency formula from Theorem \ref{thm:main} shows that \emph{a priori} quasi-uniform densities may yield pointwise speedups in the convergence of both bias and variance. 
Here we obtain a quasi-uniform point cloud with expansive coverage of $\mathcal{M}$ via a simple greedy procedure that subsamples the input dataset into a $\delta$-net. This substantially improves the stability and accuracy of TMDmap. Such distance-based quasi-uniform sampling has previously been used {for model reduction \cite{crosskey2017atlas}} and in molecular dynamics applications \cite{rydzewski2023manifold}. 

Additionally, we have probed the interplay between $\delta$ and $\epsilon$ numerically and illustrated that, for low dimensional cases, the $\delta$--net tends to yield more favorable scaling between the optimal bandwidth and number of points beyond the optimal scaling for general densities in Theorem \ref{thm:main}. Although our justification for the use of uniform densities is grounded in the i.i.d. assumption, we find that the scaling between $n$ and $\epsilon$ error is also improved with $\delta$-nets where the i.i.d. assumption no longer holds. We aim to address this interplay of $n$, $\epsilon$, and $\delta
$ in future work. 

\section{Acknowledgements}
This work was partially supported by AFOSR MURI grant FA9550-20-1-0397.

\section{Code availability}
The code and data for generating the figures in this paper are published on GitHub \cite{tmdmaprepo,distmesh}.

 \bibliographystyle{elsarticle-num} 
 \bibliography{DataML}




%

\end{document}